\documentclass{amsart}
\usepackage{amsmath}
\usepackage{amssymb}
\usepackage[all]{xy}
\usepackage{graphicx}
\usepackage{mathrsfs}
\usepackage{stmaryrd}
\usepackage[active]{srcltx}

\RequirePackage{ifpdf}
\ifpdf
  \IfFileExists{pdfsync.sty}{\RequirePackage{pdfsync}}{}
  \RequirePackage[pdftex,
   colorlinks = true,
   pagebackref,
   bookmarksopen=true]{hyperref}
\else
  \RequirePackage[hypertex]{hyperref}
\fi

\usepackage{tikz}
\tikzset{
	MyPersp/.style={scale=2,x={(0.8cm,0cm)},y={(0cm,0.25cm)},
    z={(0cm,1cm)}},
	MyPoints/.style={fill=white,draw=black,thick}
		}

\newcommand{\C}{\mathbb{C}}
\newcommand{\F}{\mathbb{F}}
\newcommand{\K}{\mathbb{K}}
\renewcommand{\O}{\mathbb{O}}
\newcommand{\E}{\mathbb{E}}
\newcommand{\Fq}{\F_q}

\newcommand{\Fqb}{\overline{\mathbb{F}}_q}
\newcommand{\Fpb}{\overline{\mathbb{F}}_p}

\newcommand{\Ql}{\mathbb{Q}_\ell}
\newcommand{\Zl}{\mathbb{Z}_\ell}
\newcommand{\Fl}{\mathbb{F}_\ell}

\newcommand{\Z}{\mathbb{Z}}

\newcommand{\V}{\mathbb{V}}

\newcommand{\cO}{\mathcal{O}}
\newcommand{\Fr}{\mathsf{Fr}}

\newcommand{\bA}{\mathbb{A}}

\DeclareMathOperator{\Spec}{Spec}
\newcommand{\pt}{\mathrm{pt}}

\newcommand{\Projfr}{\mathsf{Projf-}}
\newcommand{\Mod}{\mathsf{-Mod}}
\newcommand{\Modr}{\mathsf{Mod-}}
\newcommand{\Modrz}{\mathsf{Mod^\Z-}}
\newcommand{\Modf}{\mathsf{-Modf}}
\newcommand{\Modfr}{\mathsf{Modf-}}
\newcommand{\Modfrz}{\mathsf{Modf^{\Z}}-}
\newcommand{\Modfrdec}{\mathsf{Modf}_{\mathrm{dec}}\mathsf{-}}
\newcommand{\Modfrdecproj}{\mathsf{Modf}_{\mathrm{dec}}^{\mathrm{proj}}\mathsf{-}}

\newcommand{\cN}{{\mathcal{N}}}

\newcommand{\cB}{\mathcal{B}}

\newcommand{\cDb}{\cD^{\mathrm{b}}}

\newcommand{\cKb}{\mathcal{K}^{\mathrm{b}}}

\newcommand{\Perv}{\mathsf{Perv}}
\newcommand{\sfP}{\mathsf{P}}
\newcommand{\sfD}{\mathsf{D}}

\newcommand{\IC}{\mathrm{IC}}

\newcommand{\pH}{{}^p\! \mathcal{H}}

\newcommand{\free}{{\mathrm{free}}}

\newcommand{\sA}{\mathsf{A}}

\newcommand{\DGD}{\mathsf{-dgDer}}
\newcommand{\DGDr}{\mathsf{dgDer-}}
\newcommand{\DGHr}{\mathsf{dgHo-}}
\newcommand{\DGDZr}{\mathsf{dgDer}^{\mathbb{Z}}\mathsf{-}}

\newcommand{\DGDfr}{\mathsf{dgDerf-}}

\newcommand{\cF}{\mathcal{F}}
\newcommand{\cE}{\mathcal{E}}
\newcommand{\cG}{\mathcal{G}}
\newcommand{\cK}{\mathcal{K}}
\newcommand{\cP}{\mathcal{P}}
\newcommand{\cQ}{\mathcal{Q}}
\newcommand{\cH}{\mathcal{H}}
\newcommand{\cS}{\mathcal{S}}

\newcommand{\cM}{\mathcal{M}}

\newcommand{\cL}{\mathcal{L}}

\newcommand{\simto}{\xrightarrow{\sim}}
\newcommand{\lotimes}{{\stackrel{_L}{\otimes}}}

\newcommand{\lan}{\langle}
\newcommand{\ran}{\rangle}

\newcommand{\End}{\mathrm{End}}

\newcommand{\For}{\mathrm{For}}

\newcommand{\cD}{\mathcal{D}}

\newcommand{\onto}{\twoheadrightarrow}
\newcommand{\into}{\hookrightarrow}

\newcommand{\id}{\mathrm{id}}

\newcommand{\scS}{\mathscr{S}}

\newcommand{\cC}{\mathcal{C}}

\DeclareMathOperator{\Hom}{Hom}
\DeclareMathOperator{\Ext}{Ext}

\newcommand{\qis}{\mathrm{qis}}

\newcommand{\triright}{\stackrel{[1]}{\to}}
\newcommand{\excise}[1]{}

\newcommand{\res}{\mathsf{res}}

\newcommand{\Th}{\Theta}

\def\coH#1{\mathsf{H}^{#1}}
\def\coHc#1{\mathsf{H}_{\mathrm{c}}^{#1}}
\newcommand{\Ho}{\mathrm{H}}

\newtheorem{thm*}{Theorem}
\numberwithin{equation}{subsection}
\newtheorem{thm}[equation]{Theorem}
\newtheorem{lem}[equation]{Lemma}
\newtheorem{prop}[equation]{Proposition}
\newtheorem{cor}[equation]{Corollary}
\theoremstyle{definition}

\theoremstyle{remark}
\newtheorem{rmk}[equation]{Remark}
\newtheorem{ex}[equation]{Example}

\title[Modular Koszul Duality]{Modular Koszul duality}

\author{Simon Riche}
\address{Clermont Universit{\'e}, Universit{\'e} Blaise Pascal, Laboratoire de  
Math{\'e}ma\-tiques, BP 10448, F-63000 Clermont-Ferrand. \newline
\indent CNRS, UMR 6620, Laboratoire de Math{\'e}matiques, F-63177 Aubi{\`e}re.
}
\email{simon.riche@math.univ-bpclermont.fr}

\author{Wolfgang Soergel}
\address{Mathematisches Institut, Universit{\"a}t Freiburg, Eckerstra{\ss}e
1, D-79104 Freiburg, Germany.}
\email{Wolfgang.Soergel@math.uni-freiburg.de}

\author{Geordie Williamson}
\address{Max-Planck-Institut f\"ur Mathematik, Vivatsgasse 7, 53111,
  Bonn, Germany.
}
\email{geordie@mpim-bonn.mpg.de}

\begin{document}

\begin{abstract}
We prove an analogue of Koszul duality for category $\cO$ of a reductive group $G$ in positive
characteristic $\ell$ larger than $1$ plus the number of roots of $G$. However there are no Koszul rings, and we do not prove
an analogue of the Kazhdan--Lusztig conjectures in this context. The
main technical result is the formality of the dg-algebra of extensions
of parity sheaves on the flag variety if the characteristic
of the coefficients is at least the number of roots of $G$ plus $2$.
\end{abstract}

\maketitle

\section{Introduction}

\subsection{The scaffolding of Koszul duality}
\label{ss:scaffolding}

Given a $\Z$-graded ring $E$ let us consider the abelian category
$\Modrz E$ of all $\Z$-graded  right $E$-modules. There are two
ways to ``forget a part of the grading'' on its derived
category. That is, there are two triangulated functors:
\[
\xymatrix@=0.6cm{
\cD(\Modr E) & \cD(\Modrz
E) \ar[r]^-{\overline{v}} \ar[l]_-{v} & \DGDr E.
}
\]
On the left is simply the derived functor $v$ of forgetting the grading
$\Modrz E \to  \Modr E$. On the right is the derived category of the
differential graded ring $(E, d=0)$, which is obtained as the
localisation of homotopy category of differential graded right modules at
quasi-isomorphisms (see for example \cite{BL} for a thorough
discussion). The right-hand functor $\overline{v}$ sends a complex of graded
modules, thought of as a bigraded abelian group $(M^{i,j})$ (with
$(\cdot a) : M^{i,j} \to M^{i,j+|a|}$ for all $a \in E$ homogeneous of
degree $|a|$ and differential $d:M^{i,j} \to M^{i+1,j}$) to the
differential graded $E$-module $\overline{v} M$ with $(\overline{v} M)^n :=
\bigoplus_{i+j=n}M^{i,j}$ and obvious differential. It is this
construction which provides the basic homological scaffolding of
Koszul duality.

In order to obtain Koszul duality in the sense of \cite{BGS} we equip
the above picture with some finiteness conditions. Let $\Bbbk$ be a
field and let $E$ be
a finite dimensional $\Z$-graded $\Bbbk$-algebra of finite global
dimension. Consider the categories $\Modfr E$ (resp.~$\Modfrz E$) of finite
dimensional (resp.~finite dimensional $\Z$-graded) right
$E$-modules. Let $\DGDfr E \subset \DGDr E$ denote the full triangulated
subcategory with objects finite dimensional differential graded right
$E$-modules (see \S \ref{ss:DGDf}). In this setting Koszul duality for category $\cO$ 
means the existence of a finite dimensional $\Z$-graded $\C$-algebra
$E$ together with vertical equivalences of categories:
\begin{equation}
\label{eqn:diagram-Koszul-duality}
\vcenter{
\xymatrix@=0.6cm{
\cDb(\Modfr E)  & \ar[l]_-{v} \cDb(\Modfrz E ) \ar[r]^-{\overline{v}} & \DGDfr E  \\
\cDb(\cO_0) \ar[u]^{\wr} & & \cDb_{(B)} (G/B, \C) \ar[u]^\wr
}
}
\end{equation}
Here $G$ denotes a complex semi-simple group with Borel
subgroup $B \subset G$ and $\cDb_{(B)}(G/B,\C)$ denotes the full subcategory of the bounded derived category of
sheaves of $\C$-vector spaces on the flag variety $G/B$ whose
cohomology sheaves are constructible with respect to the
stratification by Bruhat cells. Such complexes will be called ``Bruhat
constructible'' from now on. On the left $\cO_0$  denotes the
principal block of category $\cO$ for the Langlands dual group
$G^\vee$.

We can choose our $\Z$-graded $\C$-algebra $E$ and our vertical
equivalences of categories so that the left-hand equivalence preserves
the t-structures (and hence is induced from an equivalence of
abelian $\C$-categories $\Modfr E \simto \cO_0$) and such
that for all $x$ in the Weyl group $W$ there exists a finite
dimensional $\Z$-graded $E$-module $\widetilde{M}_x$, which
specializes on the left-hand side to the Verma module with highest
weight $x \cdot 0$, and on the right-hand side to the derived direct
image of the constant perverse sheaf on the Bruhat cell $BxB/B$ under
the embedding in $G/B$. We can depict the situation as follows:
\[
M(x \cdot 0) \quad \mapsfrom \quad \widetilde{M}_x  \quad \mapsto \quad i_{x*} \underline{\C}_{BxB/B}[\ell(x)].
\]
Here $x \cdot 0$ means as usual $x \rho - \rho$ where $\rho$ is half the sum of
the positive roots, that is those roots whose weight spaces act locally
nilpotently on all objects in $\cO$.

\subsection{The modular setting}

The goal of the current work is to establish analogous statements in
the modular setting. To do this, let us choose a field $\F$ of 
characteristic $\ell > 0$ and consider the full subcategory
\[
\cDb_{(B)}(G/B,\F) \subset \cDb(G/B,\F)
\]
of Bruhat constructible complexes of sheaves on the complex flag
variety as above, the only difference being
that now we consider sheaves with coefficients in $\F$. On the
other side we consider $\cO_0(\F)$, the ``subquotient around
the Steinberg weight'' as defined in \cite{So2}. This is a
subquotient of the category of finite dimensional rational
representations of the group $G^\vee_{\F}$ over $\F$. In order for this to
make sense we need to restrict to the case where the characteristic is
bigger than the Coxeter number. 
This ensures that we can find a dominant weight in the interior of an
alcove obtained by stretching an alcove of the affine Weyl group by
$\ell$.

Under more restrictive assumptions on the characterstic of our
coefficients we prove modular analogues of the previous statements:

\begin{thm}[``Modular Koszul duality''] 
\label{thm:intro}
  Suppose that $\ell > |R|+1$, where $R$ is the root system of $G$. Then there exists a graded
  finite dimensional $\F$-algebra $E$ of finite global dimension together with vertical
  equivalences of categories:
\begin{equation}
\label{eqn:modular-Koszul-duality}
\vcenter{
\xymatrix@=0.6cm{
\cDb(\Modfr E)  & \ar[l]_-{v} \cDb(\Modfrz E) \ar[r]^-{\overline{v}} & \DGDfr E  \\
\cDb(\cO_0(\F)) \ar[u]^{\wr} & & \cDb_{(B)} (G/B, \F) \ar[u]^\wr
}
}
\end{equation}
Moreover, for all $x \in W$ there exists a finite dimension graded $E$-module
$\widetilde{M}_x$ which specialises to the standard object $M_x$ in
$\cO_0(\F)$ on the left and to $i_{x*} \underline{\F}_{BxB/B}[\ell(x)]$
on the right. In formulas:
\[
M_x\quad \mapsfrom \quad \widetilde{M}_x  \quad \mapsto \quad i_{x*} \underline{\F}_{BxB/B}[\ell(x)].
\]
\end{thm}

Let us emphasize, however, the difference to the non-modular case: in
the modular case it is not clear whether the $\Z$-graded algebra $E$
can be chosen such that it vanishes in negative degrees and is
semi-simple in degree 0 (unless we make some strong assumptions, see \S\ref{ss:koszulity} and \S\ref{ss:koszulity-standard}). Hence we are in no position to discuss the
Koszulity of $E$ in general.

We believe that the subquotient around the Steinberg
weight has no special meaning in itself. However this subquotient does
see some of the multiplicities of simple modules in Weyl modules. One
of the goals of this paper was to  try to develop an accessible model
which might help us approach Lusztig's modular conjecture. (Indeed
Lusztig's conjecture for multiplicities around the Steinberg weight is
equivalent to the existence of a positive grading as discussed above.)

A large part of the proof of this modular Koszul duality can already be found in
\cite{So2}. More precisely in \cite{So2} the second author considered
the full subcategory $\cK \subset \cDb_{(B)}(G/B,\F)$ consisting of all
direct images of constant sheaves on Bott--Samelson varieties, together
with their direct sums, summands and shifts. It was then showed that
for all $y \in W$ there exists a unique indecomposable object $\cE_y$
in $\cK$ whose support is contained in $\overline{ByB/B}$ and whose
restriction to $ByB/B$ is isomorphic to $\underline{\F}_{ByB/B}[\ell(y)]$. In
modern language $\cE_y$ is a parity sheaf and in \cite{JMW} their
existence and uniqueness is shown by purely geometric arguments,
without restriction on the characteristic. Finally, one of the main
results of \cite{So2} is an equivalence of categories
\[
\Modfr E \simto \cO_0(\F)
\]
where $E = \Ext^\bullet(\cE,\cE)$ is the algebra of self-extensions of
$\cE := \bigoplus_{y \in W} \cE_y$. More precisely, $E^i$ consists of
all morphisms $\cE \to \cE[i]$ in the derived category
$\cDb_{(B)}(G/B,\F)$ and $\cO_0$ denotes the regular subquotient around
the Steinberg weight from \cite[\S 2.3]{So2}. Hence all that remains in
order to construct a diagram as in \eqref{eqn:diagram-Koszul-duality}
is the construction of an equivalence of triangulated $\F$-categories
\[
\DGDfr E \simto \cDb_{(B)}(G/B,\F).
\]

\subsection{Some homological algebra}
\label{ss:homological-algebra}

We briefly recall some general constructions of
homological algebra, which allow us to establish such an
equivalence. Let $\mathsf{A}$ be an abelian category. We call a set of
complexes $\cC$ ``end-acyclic'' if, for all $T, T' \in \cC$ and $n \in
\Z$ the natural map gives an isomorphism
\[
\Hom_{\cK(\mathsf{A})}(T, T'[n]) \simto \Hom_{\cD(\mathsf{A})}(T, T'[n])
\]
between morphisms in the homotopy and 
derived categories. Given a finite end-acyclic family $T_1, \dots,
T_m$ of complexes we can consider the complex $T = \bigoplus_i T_i$. Its
endomorphism complex $E := \End_{\mathsf{A}}^{\bullet}(T)$ has a natural structure of a
dg-ring with idempotents $1_i \in E$ given by the
projection to each factor. Then the functor $\Hom_{\mathsf{A}}^{\bullet}(T, -)$
induces an equivalence of triangulated categories between the full
triangulated subcategory
\[
\langle T_1, \dots, T_m \rangle_{\Delta} \subset \cD(\mathsf{A})
\]
generated by the objects $T_i$ and the full subcategory
\[
\langle 1_1 E, \dots, 1_m E \rangle_{\Delta} \subset \DGDr E
\]
generated by the right dg-modules $1_i E$. An important stepping stone in
proving this claim is provided by the full triangulated subcategory 
\[
\langle T_1, \dots, T_m \rangle_{\Delta} \subset \cK(\mathsf{A})
\]
which by assumption is equivalent to the first triangulated subcategory
above. It is on this category that the functor $\Hom_{\mathsf{A}}^{\bullet}(T, -)$ first
really makes sense.

Furthermore, it it known that for any quasi-isomorphism $D \xrightarrow{\qis} E$ of
dg-rings (that is, a homomorphism of dg-rings which induces an
isomorphism on cohomology) the restriction induces an equivalence of
triangulated categories
\[
\DGDr E  \to \DGDr D.
\]
Possible sources for statements of this type are \cite{Rickard, Keller}.

\subsection{Construction of the equivalence}

We apply the above to
the abelian category of all sheaves of $\F$-vector spaces. It is
straightforward to see that the parity sheaves generate the
triangulated subcategory of Bruhat constructible sheaves. In formulas:
\[
\langle \cE_y \; | \; y \in W \rangle_{\Delta} = \cDb_{(B)}(G/B,\F).
\]
Let us choose a bounded below injective resolution $\cE_y^{\bullet}$ of each
parity sheaf $\cE_y$ and set $\cE^{\bullet} := \bigoplus_{y \in W}
\cE_y^{\bullet}$. If we consider the endomorphism dg-ring $E^{\bullet} := \End^{\bullet}(\cE^{\bullet})$
together with the obvious idempotents $1_y$ then the constructions of \S \ref{ss:homological-algebra} give an equivalence of triangulated categories
\[
\cDb_{(B)}(G/B,\F) \ \simto \ \langle 1_yE^{\bullet} \; | \; y \in W \rangle_{\Delta}
\subset \DGDr E^{\bullet}.
\]
By definition the cohomology of the dg-ring $E^{\bullet}$ is the graded ring $E =
\Ext^\bullet(\cE, \cE)$ from above. Suppose that we can find another
dg-ring $D$ and quasi-isomorphisms $E \overset{\sim}{\leftarrow} D \overset{\sim}{\rightarrow} E^{\bullet}$
together with homogeneous idempotents $1_y \in D$ which are sent to the
appropriate idempotents in $E^\bullet$ and $E$. Then we obtain a further
equivalence of derived categories
\[
\cDb_{(B)}(G/B,\F) \simto \langle 1_yE \; | \; y \in W \rangle_{\Delta}
\subset \DGDr E.
\]
Hence, having found $D$ together with these idempotents, all that
remains is to compare finiteness conditions and deduce that the above
can also be described using $\DGDfr E$

In order to find our dg-ring $D$ and the desired quasi-isomorphisms we
adapt the techniques of \cite{DGMS} to the context of modular \'etale
sheaves. 
Here the Frobenius action plays the part which
Hodge theory plays in \cite{DGMS}. In particular, we do not actually
work on a complex flag variety, but rather on the flag variety over a finite
field $\F_p$ of characteristic different from the characteristic
$\ell$ of $\F$. Also, in order to construct our dg-ring $E^{\bullet}$ with
compatible Frobenius action we instead work with perverse
sheaves (and projective resolutions). These technical demands account in part for the length of this paper. We explain in Remark \ref{rmk:etale-classical} how to deduce Theorem \ref{thm:intro} from its {\'e}tale counterpart.

\subsection{Formality in characteristic zero}

Except in special situations it is a difficult problem to
establish whether or not a given dg-algebra is formal (i.e.~quasi-isomorphic to its cohomology). However there
is a trick, orginally due to Deligne, which gives a general
method of establishing formality: any bigraded dg-algebra $R =
\bigoplus R^{i,j}$ with $d(R^{i,j}) \subset R^{i+1,j}$ and cohomology
concentrated on the diagonal $\{i = j\}$ is formal. The proof is
easy: it suffices to shear the bigrading by setting $R^{(i,j)}=R^{i-j,j}$ to arrive at
a bigraded dg-ring with cohomology 
concentrated in the cohomological degree $i=0$, which is well known to be formal.

In particular this trick can
be applied to establish the formality of the extension algebra of
Bruhat constructible intersection cohomology sheaves on the flag
variety. The idea being that if $E$ denotes this extension algebra
then $E$ can often be
equipped with an additional ``weight'' grading. The pointwise purity
of intersection cohomology complexes means that this grading is
diagonal. Hence if the dg algebra $E^\bullet$ computing $E$ can also
be equipped with a ``weight'' grading then the above trick can be used
to conclude that $E^\bullet$ is formal.

There are at least two settings where this approach can be carried
out. The first one is mixed \'etale sheaves on the flag variety over
the algebraic closure of a finite field $\Fq$, and the second one is mixed
Hodge modules on the complex flag variety. In both settings a key role
is played by the ``realization'' equivalence
\[
\cDb(\Perv_{(B)}(G/B)) \simto \cDb_{(B)}(G/B)
\]
between $\cDb_{(B)}(G/B)$ and the corresponding bounded derived category of perverse sheaves. 
In this context, we define $\cDb_{(B)}(G/B)$ to consist of
all constructible complexes of sheaves on $G/B$ 
satisfying the condition that 
when restricted to a Bruhat cell their cohomoloy sheaves  
are constant, and $\Perv_{(B)}(G/B)$ means the category of all perverse
sheaves inside. These are the objects we call
Bruhat-constructible. It will turn out that
 the Bruhat-constructible 
complexes form a triangulated category and the
Bruhat constructible perverse sheaves form the heart of a t-structure.
It is much easier to work in the
abelian category of Bruhat constructible perverse sheaves than the
abelian category of all sheaves on $G/B$. For example,
$\Perv_{(B)}(G/B)$ has enough projective objects and one even has an
inductive algorithm for their construction. 

In both settings $\Perv_{(B)}(G/B)$ has an ``enhancement''
(given by perverse sheaves on the flag variety over the finite field $\Fq$ in the
\'etale case, and by mixed Hodge modules in the complex
case). Moreover, one
can construct resolutions of the intersection cohomology complexes in
this enhanced category, which become projective resolutions when one
forgets the enhancement. It is the existence of these enhanced
resolutions that allows one to equip the dg algebra $E^\bullet$ with an
extra grading and deduce formality. (For a lucid explanation
of this argument in the case of mixed Hodge modules the reader is
referred to \cite{Schnuerer}.)

\subsection{Formality in characteristic $\ell$}

In this paper we adapt the above arguments to coefficients in a finite
field $\F$ of positive characteristic $\ell$. In order to have a suitable
enhancement we work on a flag variety defined over a finite field. (As
we have already remarked, the case of the flag variety over $\C$
can be deduced from this.) More precisely, fix
a finite field $\Fq$ of characteristic different from $\ell$ and
let $G_\circ$ denote a split semi-simple group over $\Fq$ with Borel
subgroup $B_\circ \subset G_\circ$ and flag variety $G_\circ / B_\circ$. We
follow the convention of \cite{BBD}: a 
subscript ``$\circ$'' denotes an object (variety or sheaf) defined over
$\Fq$, and suppression of the subscript denotes the
extension of scalars to the algebraic closure. Recall that our goal is
to obtain an algebraic description of $\cDb_{(B)}(G/B,\F)$ in terms of
the extension algebra of parity sheaves.

Here the first natural question arises: why do we consider
parity sheaves rather than intersection cohomology complexes?
With coefficients in $\F$ there seems to be no good notion of
purity and intersection cohomology complexes can be badly behaved.
(For example, their stalks may not satisfy parity
vanishing as they do in characteristic 0.) Instead, one can consider
indecomposable direct summands of ``Bott--Samelson sheaves'': those
sheaves obtained as direct images from Bott--Samelson
resolutions. These are the parity sheaves and it is  
straightforward to see that their $\Ext$ algebra satisfies a weak form of
purity (Theorem \ref{thm:H-diagonal}): it vanishes in odd degree and all eigenvalues of 
Frobenius in degree $2n$ are equal to the image of $q^{-n}$
in $\F_\ell$.

Analogously to the case of coefficients of characteristic
zero one can construct
resolutions of Bott--Samelson sheaves in $\Perv_{(B_\circ)}(G_\circ / B_\circ, \F)$ which
become projective resolutions when pulled back to $G/B$. Taking
endomorphisms we obtain a dg algebra $E^\bullet$ of $\F$-vector
spaces with compatible
automorphism induced by the Frobenius. It turns out that all
eigenvalues of Frobenius on $E^\bullet$ belong to the subgroup $\Sigma
\subset \F^\times$ generated by the image of $q$. Here we encounter a
problem: with coefficients in characteristic zero the
decomposition into generalized Frobenius eigenspaces gives a $\Z
\times \Z$-grading, whereas with coefficients in $\Fl$ we only obtain a
$\Z \times \Sigma$-grading. We have tried to depict the difference between
these two situations in the figure. In order to apply the trick to deduce formality we
would need to know, for example, that any fixed Frobenius eigenspace
of $E^\bullet$ (the columns in the figure) has cohomology concentrated
in only one degree. It is possible to conclude in this way but leads to unpleasant
bounds on $\ell$ in terms of the dimension of $G/B$.

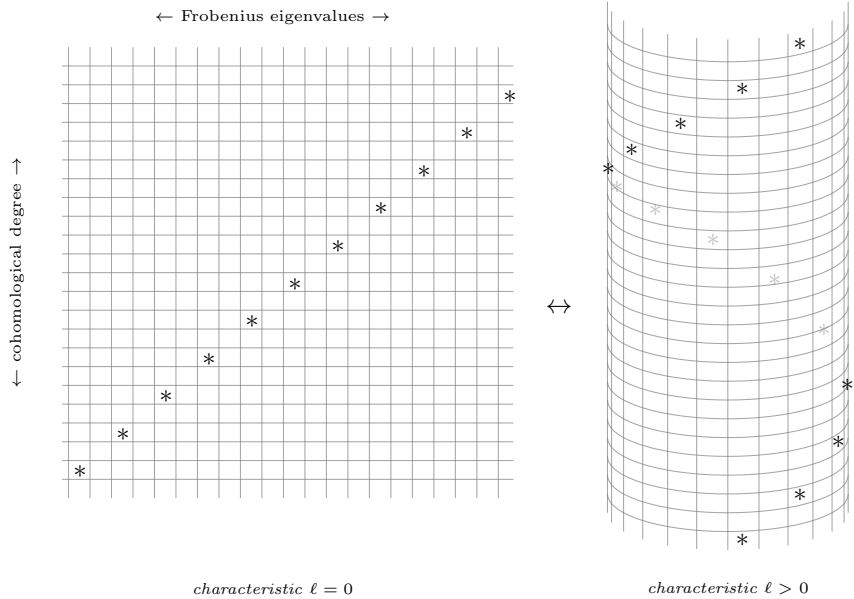
\begin{figure}[h]
\[
\begin{array}{c}
\begin{tikzpicture}[scale=2]

       \def\h{3}

       \foreach \t in {1,2,...,23} 
       \draw[gray,opacity=0.7] (0.1,\t/8)--(3.1,\t/8);

       \foreach \t in {1,2,...,21} 
       \draw[gray,opacity=0.7] (\t/7,0)--(\t/7,\h);

       \foreach \t in {0,2,4,...,20}
       \node (a) at (0.22+\t/7,0.15+\t/8) {*};

       \node (a) [rotate=90] at (-0.2,1.5) {\tiny{$\leftarrow$  cohomological degree $\to$}};
       \node (a) at (1.5,3.2) { \tiny{$\leftarrow$ Frobenius eigenvalues $\to$}};

       \node (a) at (1.5,-0.6) {{\tiny \emph{characteristic $\ell =0$}}};

\end{tikzpicture}\end{array}
\leftrightarrow
\begin{array}{c}
\begin{tikzpicture}[MyPersp]
	\def\h{1.7}

	\foreach \t in {180,195,210,...,360}
		\draw[gray,opacity=0.7] ({cos(\t)},{sin(\t)},0)
      --({cos(\t)},{sin(\t)},{2.0*\h});

     \foreach \hi in {1,2,...,26} 
        \draw[gray,opacity=0.7] (-1,0,\hi/8) 
        	\foreach \t in {180,185,190,...,360}
        		{--({cos(\t)},{sin(\t)},\hi/8)};
     \foreach \hi in {0,1,2,3}
        \node (a) at ({cos(277+30*\hi)},{sin(277+30*\hi)},{\hi/4+0.0675}) {${\tiny *}$};
     \foreach \hi in {4,5,6,7,8}
        \node[gray,opacity=0.5] (a) at ({cos(277+30*\hi)},{sin(277+30*\hi)},{\hi/4+0.0675}) {${\tiny *}$};
     \foreach \hi in {9,10,...,13}
        \node (a) at ({cos(277+30*\hi)},{sin(277+30*\hi)},{\hi/4+0.0675}) {${\tiny *}$};

 \node (a) at (0,0,-0.5) {{\tiny \emph{characteristic $\ell > 0$}}};

\end{tikzpicture}
\end{array}
\]
\end{figure}

\subsection{Formality over $\O$}

Instead we fix a finite extension $\O$ of the $\ell$-adic integers $\Zl$
with residue field $\F$ and work with coefficients in $\O$. It turns
out that all essential parts of the construction of $E^\bullet$ can be lifted to
$\O$: projective perverse sheaves over $\F$
admit lifts to $\O$, morphism spaces between these sheaves are free
$\O$-modules and the Bott--Samelson sheaves admit resolutions
via perverse $\O$-sheaves on $G_\circ / B_\circ$ which become
projective resolutions when pulled back to $G/B$. Taking endomorphisms
we obtain a dg-algebra $E^\bullet_{\O}$ which is free as an $\O$-module
and has a compatible Frobenius automorphism $\phi$, all of whose
eigenvalues are integral powers of $q$. However, an $\O$-module with an endomorphism, even if it
is free of finite rank over $\O$, need not decompose as the direct sum 
of its generalized eigenspaces (see \S\ref{ss:O-modules-automorphisms}). In order
to conclude we need to prove that $E^\bullet_\O$ is
``$q$-decomposable'': that is, that $E^\bullet_\O$ is isomorphic to
the direct sum of its generalized $\phi$-eigenspaces with eigenvalues  powers of $q$.

To explain the proof of the $\phi$-decomposability we need to recall in more detail how the
resolutions of Bott--Samelson sheaves are constructed. We
first obtain a purely algebraic description of the category of perverse
$\O$-sheaves on $G_\circ / B_\circ$ by constructing a perverse sheaf
$\cP_\circ^\O$ on $G_\circ / B_\circ$ such that its pullback $\cP^\O$ to $G/B$ is a
projective generator of $\Perv_{(B)}(G/B, \O)$. Hence, if we write $A
= \End(\cP^\O)$ we have an equivalence
\begin{gather*}
\Perv_{(B)}(G/B, \O) \simto \Modfr A.
\end{gather*}
Moreover, because $\cP^\O$ is obtained by
pullback from $G_\circ / B_\circ$, $A$ is equipped with an
automorphism $\phi$ and we have an equivalence
\begin{gather*}
\Perv_{(B_\circ )}(G_\circ / B_\circ, \O) \simto \Modfr (A,\phi) 
\end{gather*}
where $\Modfr (A, \phi)$ denotes the category of pairs $(M,\phi_M)$
where $M$ is a finitely generated right $A$-module and $\phi_M$ is an automorphism of $M$
compatible with $\phi$ (see \S\ref
{ss:O-modules-automorphisms}) and $\Perv_{(B_\circ)}(G_\circ /B_\circ , \O)$ denotes 
the category of perverse sheaves on
$G_\circ /B_\circ$ whose pullback belongs to $\Perv_{(B)}(G/B, \O)$. Under these equivalences the pullback 
functor from $G_\circ / B_\circ$ to $G/B$ is given by forgetting the
endomorphism $\phi_M$ of $M$. These considerations allow us to
conclude that the realization functor
\[
\cDb(\Perv_{(B_\circ )}(G_\circ / B_\circ , \O)) \simto \cDb_{(B_\circ )}(G_\circ / B_\circ , \O)
\]
is an equivalence (see Proposition
\ref{prop:realization-equivalence-0}).
We also prove similar results for partial flag varieties
$G_\circ / P_{s,\circ}$, where $B_\circ \subset P_{s,\circ} \subset G_\circ $ is a minimal
parabolic subgroup corresponding to a simple reflection $s \in W$:
there exists an $\O$-algebra $A^s$ with automorphism $\phi^s$ such
that the analogues for $G/P_s$ of the above results hold. Moreover, for
any such $s$ we have a
morphism of algebras $A \to A^s$ and the direct and inverse image
functors $\pi_{s!}$ and $\pi_s^!$ (where $\pi_s : G/B \to G/P_s$
denotes the projection) can be
described algebraically in terms of derived induction and restriction
along this morphism (Proposition \ref{prop:direct-inverse-image-0}).

Finally, the Bott--Samelson sheaves may be obtained by repeatedly
applying the functors $\pi_s^! \pi_{s!}$ to the skyscraper sheaf
on the base point $B/B \in G/B$. The above results allow us to translate the
problem of constructing resolutions of the Bott--Samelson sheaves into
a purely algebraic construction involving the algebras $A$ and $A^s$. In
particular, this allows us to conclude that the dg algebra
$E^\bullet_\O$ is $q$-decomposable if $A$ and $A^s$ are
(Proposition \ref{prop:projective-resolution-decomposable}). We prove
that $A$ and $A^s$ are $q$-decomposable if the order of $q$ in
$\F^{\times}$ is at least the number of roots by examining  the eigenvalues of
Frobenius that may occur during the inductive construction of 
projective perverse sheaves. This in turn is a consequence of explicit
bounds on the weights in the cohomology of an intersection of a Bruhat
cell with an opposite Bruhat cell in $G/B$ (see
\S\S\ref{subsect:standardmorphisms}--\ref{ss:geom-projective-generator}).

\subsection{Organization of the paper}

This paper contains three parts. 

\begin{description}
\item[Part \ref{pt:generalities}] The first two sections after the
introduction are devoted to proving basic properties of
  the category of perverse sheaves with coefficients in a finite extension
  $\O$ of $\Zl$ on a sufficiently nice stratified variety. The corresponding
  statements for coefficients in a field are all well known.
  \begin{description}
  \item[Section \ref{sec:projective-tilting-perverse-sheaves}] We consider
    projective perverse sheaves when our variety is defined over an
    algebraically closed field.
\item[Section \ref{sec:perverse-finite-field}]We consider varieties defined over a finite field (with coefficients in $\O$ or its residue field).
  \end{description}
\item[Part \ref{pt:flag-variety}] In the next two sections, we specialize 
our results to perverse sheaves on the
  flag variety of a reductive group.
 \begin{description}
  \item[Section \ref{sec:bounding-weights}] We prove results on the action of Frobenius on morphisms between standard objects, and derive consequences for the structure of projective perverse sheaves.
\item[Section \ref{sec:formality}]We prove our formality theorem.
  \end{description}
\item[Part \ref{pt:MKD}] The last two sections are concerned with modular
Koszul duality. 
\begin{description}
  \item[Section \ref{sec:reminder}] We recall the results of \cite{So2} concerning the
modular category $\cO$.
\item[Section \ref{sec:MKD}]We  prove
Theorem \ref{thm:intro}.
  \end{description}

\end{description}

This article is the result of our efforts to join the two partially wrong and quite incomplete preprints \cite{SoergelKoszul} and \cite{RW} into one readable article.

\subsection{Acknowledgments}

We thank Annette Huber-Klawitter and Patrick Polo for helpful discussions.

Some of this work was
completed whilst G.W.~visited the Universit\'e Blaise Pascal -
Clermont-Ferrand II  and S.R.~visited the MPIM in
Bonn. We would like to thank both institutions for their support.
S.R.~was supported by ANR Grants No.~ANR-09-JCJC-0102-01 and No.~ANR-10-BLAN-0110. W.S.~acknowledges the support of the DFG in the framework of SPP 1388.

\subsection{Notation}

If $R$ is a ring (resp.~$\Z$-graded ring) we let $\Modr R$ (resp.~$\Modrz R$) be the category of right $R$-module (resp.~$\Z$-graded right $R$-modules). We denote by $\Modfr R$, resp.~$\Modfrz R$, the subcategory of finitely generated modules. We also denote by $\Projfr R$ the category of finitely generated projective right $R$-modules. We write $\Hom_{-R}(-,-)$ for $\Hom_{\Modr R}(-,-)$. We denote by
\[
\langle 1 \rangle : \Modrz R \to \Modrz R
\]
the auto-equivalence which sends a graded module $M=\oplus_i M^i$ to the graded module whose $i$-th component is $M^{i-1}$. We denote by $\langle n \rangle$ the $n$-th power of $\langle 1 \rangle$. With this convention, the functor $\overline{v}$ satisfies
\[
\overline{v}(M \lan 1 \ran) \ = \ \overline{v}(M)[-1].
\]
Similarly, we denote by $R\Modf$ the category of finitely generated left modules over the ring $R$.

If $R$ is a dg-ring, we denote by $\DGHr R$ the homotopy category of right $R$-dg-modules, and by $\DGDr R$ the associated derived category.

In the whole paper, we fix a prime number $\ell$. We also fix a finite extension $\mathbb{K}$ of $\Ql$, and denote its ring of integers by $\O$. It is a finite extension of $\Zl$. We fix a uniformizer $\pi$, and set $\F=\O/(\pi)$; it is a finite field of characteristic $\ell$. We denote by $\overline{\K}$ an algebraic closure of $\K$.

\part{Generalities on $\O$-perverse sheaves}
\label{pt:generalities}

\section{Projective and tilting perverse sheaves}
\label{sec:projective-tilting-perverse-sheaves}

Let us an integer $p$ which is either
$0$ or a prime number different from $\ell$. In this section, $X$ is
a variety (i.e.~a separated reduced scheme of finite type) over an algebraically closed field of characteristic $p$,
endowed with a finite stratification by locally closed subvarieties
\[
X=\bigsqcup_{s \in \scS} X_s
\]
each isomorphic to an affine space. For each $s \in \scS$ we denote by $i_s : X_s
\into X$ the inclusion.

Note that the constructions in this section work similarly if we assume $X$ is a complex algebraic variety, and work with the classical topology instead of the {\'e}tale topology.

\subsection{Definitions and first results}
\label{ss:definitions-first-results}

Consider the
constructible derived categories $\cDb_c(X,\O)$ and $\cDb_c(X,\F)$ see e.g.~\cite[\S
2.2]{BBD} or \cite[\S 2]{J} for definitions. For any $m \in \Z_{\geq
  1} \cup \{ \infty \}$ let $\cL_{s,m}$ denote the constant local
system on $X_s$ with stalk $\O / (\pi^m)$. (Here, by convention,
$\O / (\pi^\infty) = \O$.) For simplicity we sometimes write
$\cL_s$ for $\cL_{s, \infty}$. We write $\cL_{s,\F}$ for the constant
$\F$-local system of rank one on $X_s$. We assume that the following condition
is satisfied:
\begin{equation}
  \label{eq:strata assumption}
  \text{for all $s,t \in \scS$ and $n \in \Z$,
    $\mathcal{H}^n(i_t^*i_{s*}\cL_{s})$ is constant.}
\end{equation}
(If $p > 0$ then the strata $X_s$ are not simply-connected unless
$X_s$ is a point and so 
this condition is a priori stronger than requiring that
$\mathcal{H}^n(i_t^*i_{s*}\cL_{s})$ be locally constant.)
This implies the analogous condition for $\cL_{s,m}$ and $\cL_{s,\F}$. (See \cite[\S 2.1]{Yu} for a discussion of this condition; note that in our case an extension of constant sheaves on a stratum is constant.)

For $\E = \O$ or $\F$ we denote by $\cDb_{\scS}(X,\E)$
the full subcategory of $\cDb_c(X,\E)$ consisting of objects $\cF$
such that for any $s \in \scS$ and $n \in \Z$, $\mathcal{H}^n(i_s^*\cF)$
is a constant local system of finitely generated $\E$-modules. Our assumption \eqref{eq:strata assumption} guarantees
that if $i : Y \into X$ is a locally closed inclusion of a union of strata the functors $i_*$, $i_!$, $i^*$, $i^!$ restrict to functors between $\cDb_{\scS}(X,\E)$ and $\cDb_{\scS}(Y,\E)$ (where for simplicity we write $\scS$ also for the restriction of the stratification $\scS$ to $Y$), so that $\cDb_{\scS}(X, \E)$ can be endowed with the perverse
t-structure whose heart we denote by 
\[
\Perv_{\scS}(X,\E) \subset \cDb_{\scS}(X,\E).
\]
Recall (see e.g.~\cite{BBD, J}) that if we take coefficients in $\O$ we
cannot speak of \emph{the} subcategory of perverse
sheaves on $X$ with respect to this stratification. Instead one has
two ``natural'' perverse t-structures $p$ and $p^+$ with hearts $\Perv_{\scS}(X,\O)$ and
$\Perv^+_{\scS}(X,\O)$ respectively; Verdier duality exchanges these two hearts. (The
reader unfamiliar with these facts is encouraged to think about the
case of $X = \pt$.) In this paper we will only ever need to consider the perversity $p$ (except in Lemma \ref{lem:affine-exact}). Note that the category $\Perv_{\scS}(X,\O)$ is noetherian but not artinian.

We have a modular reduction functor
\[
\F \ := \ \F \, \lotimes_{\O} \, (-) : \cDb_{\scS}(X,\O) \to \cDb_{\scS}(X,\F).
\]
By definition, an object $\cM$ of $\cDb_{\scS}(X,\O)$ is a collection $(\cM_k)_{k \geq 1}$ (where each $\cM_k$ is an object of the derived category of {\'e}tale sheaves of $\O/(\pi^k)$-modules satisfying certain properties) together with isomorphisms $\cM_{k+1} \lotimes_{\O/(\pi^{k+1})} \O/(\pi^k) \cong \cM_k$; then we have $\F(\cM)=\cM_1$. By definition, this functor commutes with all direct and inverse image functors, see e.g.~\cite[Appendix A]{KW} and references therein.

The functor $\F$ does not preserve the subcategory of perverse sheaves
(again, this is already false for a point). However it is right exact and if $\cF \in
\Perv_{\scS}(X, \O)$ then $\pH^i (\F \cF) = 0$ if $i \ne 0, -1$.

Below we will need the following result.

\begin{lem}
\label{lem:affine-exact}

The functors $i_{s*}$ and $i_{s!}$ are exact for the perverse t-structure, for coefficients $\O$ or $\F$.

\end{lem}

\begin{proof}\footnote{We thank Daniel Juteau and Carl Mautner for explaining this proof to us.}
The morphism $i_s$ is affine. Hence, in the case of coefficients $\F$ the result is proved in \cite[Corollaire 4.1.3]{BBD}. Let us consider now the case of coefficients $\O$. As Verdier duality exchanges the t-structures $p$ and $p^+$, it is sufficient to prove that the functor $i_{s*}$ is exact for both t-structures $p$ and $p^+$. As the subcategories ${}^p \cD^{\geq 0}$ and ${}^{p^+} \cD^{\geq 0}$ are defined in terms of functors $i_t^!$, it is easy to check that this functor is right exact for both t-structures. The same arguments as for $\F$ prove that $i_{s*}$ is also right exact for the t-structure $p$. 
Finally, it is explained in \cite[3.3.4]{BBD} that an $\O$-linear triangulated functor from $\cDb_\scS(X_s)$ to $\cDb_\scS(X)$ is right exact for the t-structure $p$ iff it is for the t-structure $p^+$, which finishes the proof.
\end{proof}

\begin{rmk}
Similar arguments show more generally that if $f$ is an affine morphism and for coefficients $\O$, $f_*$ is right exact and $f_!$ is left exact for both t-structures $p$ and $p^+$. The case of coefficients in a field ($\F$ or $\K$) is proved in \cite[Th{\'e}or{\`e}me 4.1.1, Corollaire 4.1.2]{BBD}.
\end{rmk}

If $Y \subset X$ is a locally closed union of strata, and if $\E$ is $\O$ or $\F$, we denote by
\[
\IC(Y,-) : \Perv_{\scS}(Y,\E) \to \Perv_{\scS}(X,\E)
\]
the intermediate extension functor.
This functor is fully-faithful
(\cite[Proposition 2.29]{J})) and preserves injections and surjections
(\cite[Proposition 2.27]{J}), but is not exact.\footnote{In \cite{J},
  it is assumed on p.~1196 that the categories of perverse sheaves are
  noetherian and artinian; however this assumption is not used in the
  proofs of Propositions 2.27 and 2.29.} If $\cL$ is a local system on
$Y$ we will sometimes abuse notation and write $\IC(Y, \cL)$ for
$\IC(Y, \cL[\dim Y])$.

\begin{lem}
\label{lem:ICs-generate}
  Any object in $\Perv_{\scS}(X, \O)$ is a successive extension of
  $\IC(X_s, \cL_{s,m})$ for some $s \in \scS$ and $m \in \Z_{\ge 1}
  \cup \{ \infty \}$.
\end{lem}

\begin{proof}
We prove this result by induction on the number of strata, the case when $X$ consists of a single stratum being obvious (since in this case the category $\Perv_{\scS}(X,\O)$ is equivalent to $\Modfr \O$).

Let $X_s$ be an open stratum, and let $Y=X \smallsetminus X_s$ (a closed union of strata). Let $i: Y \into X$ denote the
  inclusion. Because the functor
$i_* : \Perv_{\scS \setminus \{ s \}}(Y,\O) \to \Perv_{\scS}(X,\O)$ is fully-faithful we
may assume by induction that the lemma is true for objects supported on
$Y$. Now let $\cF$ be an arbitrary object of $\Perv_{\scS}(X,
\O)$. The morphism $\cF \to i_*\pH^0(i^*\cF)$ induced by adjunction is
surjective \cite[Proposition 1.4.17(ii)]{BBD} and hence we have an
exact sequence of perverse sheaves
\[
\mathrm{Ker} \into \cF \onto i_*\pH^0(i^*\cF).
\]
Hence by induction it is enough to prove the result for
$\mathrm{Ker}$. One sees easily that $i_*\pH^0(i^* \mathrm{Ker}) = 0$ and hence
$\mathrm{Ker}$ has no non-zero quotient supported on $Y$. Now, we have
a dual exact sequence (again by \cite[Proposition 1.4.17(ii)]{BBD})
\[
i_* \pH^0(i^!\mathrm{Ker}) \into \mathrm{Ker} \onto \mathrm{Coker}
\]
and hence it is sufficient to prove the result for $\mathrm{Coker}$.
Similarly to above one sees that $\mathrm{Coker}$ has no non-trivial subobject
or quotient supported on $Y$. By definition $\Perv_{\scS}(X,\O)$ is obtained by recollement from
$\Perv_{\scS\setminus \{ s \}}(Y,\O)$ and $\Perv_{\{s\}}(X_s,\O)$, and
hence we can apply \cite[Corollaire 1.4.25]{BBD} to conclude that $\mathrm{Coker}
\cong \IC(X, i_s^*\mathrm{Coker})$. By definition of
$\Perv_{\{s\}}(X_s,\O)$, $i_s^*\mathrm{Coker}$ is a constant local
system, hence isomorphic to a direct sum of local systems of the form
$\cL_{s,m}$. The lemma now follows.
\end{proof}

\begin{rmk}
One can deduce from this lemma (using induction on the number of strata) that the category $\Perv_{\scS}(X,\O)$ is generated (under extensions) by a finite number of objects, namely the collection of objects $\IC(X_s,\cL_{s,1})$ and $\IC(X_s,\cL_{s,\infty})$ for $s \in \scS$.
\end{rmk}

Recall that an additive category $\mathsf{A}$ \emph{satisfies the
  Krull--Schmidt property} if every object in $\mathsf{A}$ is
isomorphic to a finite
direct sum of indecomposable objects, and $\End_{\mathsf{A}}(M)$ is local for any indecomposable object $M$. In such a category, the decomposition as a direct sum of indecomposable objects is unique up to isomorphism and permutation of factors.

\begin{lem}
\label{lem:Krull-Schmidt}

The category $\Perv_{\scS}(X,\O)$ satisfies the Krull--Schmidt property.

\end{lem}

\begin{proof}
 The
  endomorphism ring of any object of $\Perv_{\scS}(X,\O)$ is a
  finitely generated $\O$-module, and hence is
  semi-perfect by \cite[Example 23.3]{Lam}. Now it is easily checked
  (using e.g.~\cite[Theorem 23.6]{Lam}) that if $\mathsf{A}$ is an
  abelian category in which every object has a semi-perfect
  endomorphism ring, then $\mathsf{A}$ satisfies the
  Krull--Schmidt property.
  \end{proof}

\subsection{Standard and costandard objects over $\O$ and $\F$}
\label{ss:standard}
For any $m \in \Z_{\geq
  1} \cup \{ \infty \}$, set
\[
\Delta_{s,m} = i_{s!}\cL_{s,m}[ \dim X_s], \quad
\nabla_{s,m} = i_{s*}\cL_{s,m} [\dim X_s], \quad
\IC_{s,m} = \IC(X_s, \cL_{s,m}).
\]
These objects are all in $\Perv_{\scS}(X,\O)$ (see Lemma \ref{lem:affine-exact}). We often abbreviate $\Delta_s = \Delta_{s,\infty}$, $\nabla_s =
\nabla_{s,\infty}$, $\IC_s = \IC_{s,\infty}$.

Let $\cL_{s,\F}$ denote the constant local system on $s$ with
stalk $\F$. We use the following notation for the analogous objects over $\F$
\[
\Delta_{s,\F} = i_{s!}\cL_{s,\F}[ \dim X_s], \quad
\nabla_{s,\F} = i_{s*}\cL_{s,\F} [ \dim X_s], \quad
\IC_{s,\F} = \IC(X_s, \cL_{s,\F})
\]
in $\Perv_{\scS}(X,\F)$. We have $\F(\Delta_s) \cong \Delta_{s,\F}$, $\F(\nabla_s) \cong \nabla_{s,\F}$.

\begin{lem}
\label{lem:perverse-sheaf-0}

Let $\cF$ be in $\cDb_{\scS}(X,\O)$. If $\F(\cF)=0$, then $\cF=0$.

\end{lem}

\begin{proof}
We 
always have an exact triangle 
$\mathcal F\rightarrow \mathcal F\rightarrow \mathbb F(\mathcal F)\triright$
with the first map multiplication by $\pi$. The result follows, taking cohomology of the 
stalks and using Nakayama's lemma.
\end{proof}


We will say that an object $\cF$ of $\Perv_{\scS}(X,\O)$ \emph{has a $\Delta$-filtration} (resp.~\emph{a $\nabla$-filtration}) if it admits a filtration in the abelian category $\Perv_{\scS}(X,\O)$ with subquotients of the form $\Delta_t$ (resp.~$\nabla_t$) for $t \in \scS$. Note that here we only allow \emph{free} (co-)standard objects. Similarly, we will say that an object $\cF$ of $\Perv_{\scS}(X,\F)$ \emph{has a $\Delta_{\F}$-filtration} (resp.~\emph{a $\nabla_{\F}$-filtration}) if it admits a filtration in the abelian category $\Perv_{\scS}(X,\F)$ with subquotients of the form $\Delta_{t,\F}$ (resp.~$\nabla_{t,\F}$) for $t \in \scS$.

\begin{lem}
\label{lem:Delta-flag}

Let $\cF$ be in $\Perv_{\scS}(X,\O)$. Assume that $\F(\cF)$ is in $\Perv_{\scS}(X,\F)$, and has a $\Delta_{\F}$-filtration. Then $\cF$ has a $\Delta$-filtration.

\end{lem}

\begin{proof}
We proceed by induction on the number of strata. We can assume that the support of $M$ is $X$, and let $X_s$ be an open stratum. Let $Y=X \smallsetminus X_s$, and let $i : Y \hookrightarrow X$ be the (closed) inclusion.

Consider the restriction $i_s^* \cF$. It is a shift of a constant local system on $X_s$. We have $\F(i_s^* \cF)=i_s^*(\F(\cF))$, hence by assumption $\F(i_s^* \cF)$ is also a shift (by the same integer) of a local system on $X_s$. This implies that $i_s^* \cF$ has no torsion, hence that $i_{s!} i_s^* \cF$ is a direct sum of finitely many copies of $\Delta_{s}$. Consider the exact triangle
\begin{equation}
\label{eqn:triangle}
i_{s!} i_s^* \cF \to \cF \to i_* i^* \cF \triright.
\end{equation}
Its modular reduction is the similar triangle for $\F \cF$:
\[
i_{s!} i_s^* \F \cF \to \F \cF \to i_* i^* \F \cF \triright.
\]
As $\F \cF$ has a filtration with subquotients of the form $\Delta_{s,\F}$, this triangle is an exact sequence of perverse sheaves. In particular, $\F(i_* i^* \cF)$ is a perverse sheaf.

We claim that $i_* i^* \cF$ is a perverse sheaf. Indeed, using \eqref{eqn:triangle} it can have non-zero perverse cohomology sheaves only in degrees $-1$ and $0$. Consider the truncation triangle
\[
\pH^{-1}(i_* i^* \cF)[1] \to i_* i^* \cF \to \pH^0(i_* i^* \cF) \triright,
\]
and its modular reduction
\[
\F(\pH^{-1}(i_* i^* \cF))[1] \to \F(i_* i^* \cF) \to \F(\pH^0(i_* i^* \cF)) \triright.
\]
We have seen above that the middle term is a perverse sheaf, while the left-hand side (resp.~the right-hand side) is concentrated in perverse degrees $-2$ and $-1$ (resp.~$-1$ and $0$). It easily follows that $\F(\pH^{-1}(i_* i^* \cF))=0$, hence $\pH^{-1}(i_* i^* \cF)=0$ by Lemma \ref{lem:perverse-sheaf-0}, which proves the claim.

By this claim, triangle \eqref{eqn:triangle} is an exact sequence of perverse sheaves. Moreover, $i_* i^* \cF$ again satisfies the conditions of the lemma. Hence, by induction, it has a filtration with subquotients of the form $\Delta_{t}$. The result follows.
\end{proof}

\subsection{Projective objects: existence}

\begin{prop}
\label{prop:existence-projectives}
For any $s \in \scS$ there exists a projective object $\cP$ in the category $\Perv_{\scS}(X,\O)$ which admits a $\Delta$-filtration and a surjection $\cP \twoheadrightarrow \IC_{s}$.

\end{prop}

\begin{proof}
We prove the proposition by induction on the number of strata. First, the result is clearly true when $X$ consists of a single stratum.

Now, let $X_s$ be an open stratum, and set $Y=X \smallsetminus X_s$. We will abbreviate $\sfP = \Perv_{\scS}(X, \O)$ and $\sfP_Y =
\Perv_{\scS}(Y, \O)$. Given $\cF \in \sfP$, we have that $\cF \in \sfP_Y$ if and only if
$i_s^! \cF (= i_s^* \cF) = 0$. Also, note that if $\cF, \cG \in
\sfP_Y$ then
\[
\Ext^1_{\sfP_Y}(\cF, \cG) = \Ext^1_{\sfP}(\cF, \cG).
\]
(If one thinks about both groups as classifying extensions, then
any such extension has to be supported on $Y$.) Similarly, 
\[
 \Ext^1_{\sfP}(\cF,\cG) \cong \Ext^1_{\cDb_{\scS}(X,\O)}(\cF,\cG)
\]
if $\cF,\cG \in \sfP$. These equalities allow
us to omit the subscript from $\Ext^1$'s below. 

By adjunction, for any $\cF$ in $\sfP$ we have
\[
\Ext^1(\Delta_s, \cF) = 
\Hom(\Delta_s, \cF[1]) = \Hom_{\cDb_{\{s\}}(X_s,\cO)}(\cL_s, i_s^! \cF[1]) = 0,
\]
which implies that $\Delta_s \in \cP$ is projective (and
$\Delta$-filtered). Hence we have found our first surjection $\Delta_s \onto \IC_s$.

To construct the other projectives we will adopt the following strategy (essentially copying
\cite{BGS}): by induction we may assume that for each $t \in \scS \smallsetminus \{s\}$ we can find a
surjection $\cP \onto \IC_t$, with $\cP$ projective and $\Delta$-filtered
in $\sfP_Y$; we then see what needs to be done to enlarge $\cP$, so that
it is projective in $\sfP$.

So, fix $t \in \scS \smallsetminus \{s\}$ and assume that $\cP_Y$ is a $\Delta$-filtered projective object
in $\sfP_Y$ which surjects to $\IC_t$. Certainly we have
\begin{equation} \label{eq:P_I vanishing}
\Ext^1(\cP_Y, \cF) = 0\quad \text{for any} \quad\cF \in\sfP_Y.
\end{equation}
Now, consider $E = \Ext^1(\cP_Y, \Delta_s)$ and let $E_\free$ be a finitely generated free $\O$-module endowed with a surjection to $E$. We denote by $E_\free^*$ the dual $\O$-module. The sequence
of canonical morphisms
\[
\O \to E_\free^*\otimes_{\O} E_\free \to
E_\free^* \otimes_{\O} \Ext^1(\cP_Y, \Delta_s) \cong \Ext^1(\cP_Y,
E_\free^* \otimes_{\O} \Delta_s)
\]
gives us a distinguished extension
\[
E_\free^* \otimes_{\O} \Delta_s \hookrightarrow \cP \onto \cP_Y
\]
in $\sfP$. Composing $\cP \onto \cP_Y$ with $\cP_Y \onto \IC_t$ yields a
surjection of $\cP$ onto $\IC_t$, and $\cP$ is $\Delta$-filtered. We claim that
$\cP$ is projective, which will conclude the proof. To prove this claim we need four preliminary steps.

\emph{Step 1:} $\Ext^1(\cP, \cF) = 0$ for $\cF \in
\cP_Y$. This is clear from \eqref{eq:P_I vanishing}, the long exact sequence
\[
\dots \to \Ext^1(\cP_Y, \cF) \to 
\Ext^1(\cP, \cF) \to 
\Ext^1(E_\free^* \otimes_{\O} \Delta_s, \cF) \to \dots
\]
and the fact (already used above) that $\Delta_s \in \sfP$ is projective.

\emph{Step 2:} $\Ext^1(\cP, \Delta_{s}) = 0$. Consider the long
exact sequence:
\begin{multline*}
\dots \to \Hom(E_\free^*\otimes_{\O} \Delta_s, \Delta_s) \to \Ext^1(\cP_Y, \Delta_s) \to 
\Ext^1(\cP, \Delta_s) \\ \to 
\Ext^1(E_\free^* \otimes_{\O} \Delta_s, \Delta_s) \to \dots
\end{multline*}
Because $\Ext^1(\Delta_s, \Delta_s) = 0$ it is enough
to show that the first map above is surjective. However under the
canonical isomorphisms
\[
\Hom(E_\free^*\otimes_{\O} \Delta_s, \Delta_s) \cong E_\free \otimes_{\O}
\Hom(\Delta_s, \Delta_s) \cong E_\free
\]
this map corresponds to the map $E_\free \to \Ext^1(\cP_Y, \Delta_s)$,
which is surjective by construction.

\emph{Step 3:} $\Ext^2_{\cDb_{\scS}(X,\O)}(\cP, \cF) = 0$ for all $\cF \in \sfP$.
By Lemma \ref{lem:ICs-generate} it is enough to show that
$\Ext^2_{\cDb_{\scS}(X,\O)}(\cP, \IC_{t,m}) = 0$ for all $t \in \scS$, $m \in \Z_{\geq 1}
\cup \{ \infty \}$. The short
exact sequence
\[
\IC_{u,m} \hookrightarrow \nabla_{u,m} \onto \mathrm{Coker}
\]
leads to a long exact sequence
\[
\dots \to \Ext^1(\cP, \mathrm{Coker}) \to
\Ext^2_{\cDb_{\scS}(X,\O)}(\cP,\IC_{u,m}) \to \Ext^2_{\cDb_{\scS}(X,\O)}(\cP,
\nabla_{u,m}) \to \dots
\]
Now $\mathrm{Coker} \in \sfP_Y$ and hence $\Ext^1(\cP,
\mathrm{Coker}) = 0$ by Step 1. Also, $\cP$ is
$\Delta$-filtered, and hence $\Ext^2_{\cDb_{\scS}(X,\O)}(\cP,
\nabla_{u,m}) = 0$. It follows that $\Ext^2_{\cDb_{\scS}(X,\O)}(\cP,\IC_{u,m}) = 0$ as claimed.

\emph{Step 4:} $\Ext^1(\cP, \IC_{s,m}) = 0$ for all $m$. The short exact
sequence
\[
\mathrm{Ker} \hookrightarrow \Delta_s \onto \IC_{s,m}
\]
leads to a long exact sequence:
\[
\dots \to \Ext^1(\cP, \Delta_s) \to \Ext^1(\cP,
\IC_{s,m}) \to \Ext^2_{\cDb_{\scS}(X,\O)}(\cP, \mathrm{Ker}) \to \dots
\]
Step 2 and Step 3 show that the first and third terms are zero, which proves the claim.

Finally, we have shown in Steps 1 and 4 that $\Ext^1(\cP, \IC_{t,m}) = 0$ for all
$t \in \scS$ and $m \in \Z_{\geq 1} \cup \{ \infty \}$. Hence $\cP$ is
projective by Lemma \ref{lem:ICs-generate}.
\end{proof}

\begin{cor}
\label{cor:properties-projectives}

There are enough projective objects in $\Perv_{\scS}(X,\O)$. More precisely, for any $\cM$ in $\Perv_{\scS}(X,\O)$ there exists a projective object $\cP$ in $\Perv_{\scS}(X,\O)$ which admits a $\Delta$-flag and a surjection $\cP \onto \cM$.

\end{cor}

\begin{proof}
By Lemma \ref{lem:ICs-generate}, it is enough to prove the result for $\cM=\IC_{s,m}$ ($s \in \scS$, $m \in \Z_{\geq 1} \cup \{\infty\}$). The case $m=\infty$ is treated in Proposition \ref{prop:existence-projectives}; the general case follows from the fact that the natural morphism $\IC_{s,\infty} \to \IC_{s,m}$ is surjective, since intermediate extension preserves surjections.
\end{proof}

\begin{cor}
\label{cor:realization-functor}

\begin{enumerate}
\item The realization functor
\[
\cDb \Perv_{\scS}(X,\O) \to \cDb_{\scS}(X,\O)
\]
(see \cite{Be}) is an equivalence of categories.
\item The category $\Perv_{\scS}(X,\O)$ has finite projective dimension, i.e.~every object admits a finite projective resolution.
\end{enumerate}

\end{cor}

\begin{proof}
(1) The idea of this proof is taken from \cite[Corollary 3.3.2]{BGS}. As $\Perv_{\scS}(X,\O)$ generates the category $\cDb_{\scS}(X,\O)$, it is enough to prove that the realization functor is fully faithful, or equivalently that it induces an isomorphism
\begin{equation}
\label{eqn:realization-functor}
\Hom_{\cDb \Perv_{\scS}(X,\O)}^i(\cF,\cG) \xrightarrow{\sim} \Hom_{\cDb_{\scS}(X,\O)}^i(\cF,\cG).
\end{equation}
for all $i \in \Z$ and all $\cF,\cG$ in $\Perv_{\scS}(X,\O)$.

First, we claim that both sides in \eqref{eqn:realization-functor} vanish if $\cF$ is projective with a $\Delta$-filtration and $i \geq 1$ (and for any $\cG$). This claim is obvious by definition for the left-hand side. We prove the claim for the right-hand side by induction on $i$. If $i=1$, then morphism \eqref{eqn:realization-functor} is an isomorphism by \cite[Remarque 3.1.17]{BBD}, which proves the vanishing of the right-hand side. Now let $i \geq 2$, and assume the claim is known for $i-1$. By Lemma \ref{lem:ICs-generate} it suffices to treat the case $\cG=\IC_{s,m}$ for $s \in \scS$ and $m \in \Z_{\geq 1} \cup \{\infty\}$. Consider the exact sequence
\[
\IC_{s,m} \into \nabla_{s,m} \onto \mathrm{Coker}.
\]
Applying $\Hom(\cF,-)$ we obtain an exact sequence:
\[
\Hom_{\cDb_{\scS}(X,\O)}^{i-1}(\cF,\mathrm{Coker}) \to \Hom_{\cDb_{\scS}(X,\O)}^{i}(\cF,\IC_{s,m}) \to \Hom_{\cDb_{\scS}(X,\O)}^i(\cF,\nabla_{s,m}).
\]
The left term vanishes by induction. The right term vanishes because $\cF$ has a $\Delta$-filtration. Hence the middle term also vanishes, which finishes the proof of the claim.

Now we deduce that \eqref{eqn:realization-functor} is an isomorphism in general. First, both sides vanish if $i<0$, and the realization functor induces an isomorphism for $i=0$ by definition and for $i=1$ by \cite[Remarque 3.1.17]{BBD}. We treat the case $i \geq 2$ by induction. So, assume that that $i \geq 2$ and that the result is known for $i-1$. Let $\cP$ be a projective object with a $\Delta$-filtration surjecting onto $\cF$ (see Corollary \ref{cor:properties-projectives}), and consider an exact sequence
\[
\mathrm{Ker} \into \cP \onto \cF.
\]
We have a commutative diagram with exact rows
{\small
\[
\xymatrix{
\Hom_{\cDb \sfP}^{i-1}(\cP,\cG) \ar[r] \ar[d] & \Hom_{\cDb \sfP}^{i-1}(\mathrm{Ker},\cG) \ar[r] \ar[d] & \Hom_{\cDb \sfP}^{i}(\cF,\cG) \ar[r] \ar[d] & \Hom_{\cDb \sfP}^{i}(\cP,\cG) \ar[d] \\
\Hom_{\sfD}^{i-1}(\cP,\cG) \ar[r] & \Hom_{\sfD}^{i-1}(\mathrm{Ker},\cG) \ar[r] & \Hom_{\sfD}^{i}(\cF,\cG) \ar[r] & \Hom_{\sfD}^{i}(\cP,\cG)
}
\]
}where $\sfP=\Perv_{\scS}(X,\O)$ and $\sfD=\cDb_{\scS}(X,\O)$. By our claim above, the terms in the right and the left column vanish. And by induction the second vertical arrow is an isomorphism. It follows that the third vertical arrow is also an isomorphism.

(2) The idea for this proof is taken from \cite[Corollary 3.2.2]{BGS}. By Lemma \ref{lem:ICs-generate} it is enough to prove that any object $\IC_{s,m}$ ($s \in \scS$, $m \in \Z_{\geq 1} \cup \{\infty\}$) has a finite projective resolution. Then, by induction on $\dim(X_s)$, it is enough to prove the claim for objects $\Delta_{s,m}$ ($s \in \scS$, $m \in \Z_{\geq 1} \cup \{\infty\}$). If $m \neq \infty$, there is an exact sequence
\[
\Delta_s \into \Delta_s \onto \Delta_{s,m}
\]
since the functor $i_{s!}$ is exact (see Lemma \ref{lem:affine-exact}). Hence we only have to prove the claim for objects $\Delta_s$, $s \in \scS$. Then we use a decreasing induction on $\dim(X_s)$. If this dimension is maximal, then $\Delta_s$ is projective. The induction step follows from the property that there exists a projective $\cP_s$ surjecting onto $\Delta_s$ such that the kernel has a $\Delta$-filtration with subquotients $\Delta_t$ with $\dim(X_t) > \dim(X_s)$ (see the proof of Proposition \ref{prop:existence-projectives}).
\end{proof}

Thanks to Corollary \ref{cor:realization-functor}, we do not have to mention whether we consider $\Ext$-groups in $\Perv_{\scS}(X,\O)$ or in $\cDb_{\scS}(X,\O)$. For simplicity, for $\cF,\cG$ in $\Perv_{\scS}(X,\O)$ we will write $\Ext^i_X(\cF,\cG)$ for $\Ext^i_{\Perv_{\scS}(X,\O)}(\cF,\cG) \cong \Hom^i_{\cDb_{\scS}(X,\O)}(\cF,\cG)$.

\subsection{Projective objects: properties}

Recall that the projective objects in the category $\Perv_{\scS}(X,\F)$ can be described using \cite[Theorem 3.2.1]{BGS}. In particular, there are enough projective objects in this category, and any such object has a $\Delta_{\F}$-filtration. Moreover, by the arguments of \cite[Corollary 3.3.2]{BGS} the realization functor
\[
\cDb \Perv_{\scS}(X,\F) \ \to \ \cDb_{\scS}(X,\F)
\]
is an equivalence of categories. For $s \in \scS$, we denote by $\cP_{s,\F}$ the projective cover of $\IC_{s,\F}$ in $\Perv_{\scS}(X,\F)$.

\begin{prop}
\label{prop:properties-projectives-2}

Let $\cP$ be a projective object in $\Perv_{\scS}(X,\O)$.

\begin{enumerate}
\item The object $\F(\cP)$ is a perverse sheaf, which is projective in $\Perv_{\scS}(X,\F)$.
\item For any $s \in \scS$, the $\O$-module $\Hom(\cP,\Delta_s)$ is free, and the natural morphism
\[
\F \otimes_{\O} \Hom_{\O}(\cP, \Delta_s) \ \to \ \Hom_{\F}(\F (\cP), \Delta_{s,\F})
\]
is an isomorphism.
\item $\cP$ admits a $\Delta$-filtration.
\item For any $\cQ$ projective in $\Perv_{\scS}(X,\O)$, the $\O$-module 
\[
\Hom_{\O}(\cP,\cQ)
\]
is free. Moreover, the natural morphism
\[
\F \otimes_{\O} \Hom_{\O}(\cP,\cQ) \ \to \ \Hom_{\F}(\F(\cP),\F(\cQ))
\]
is an isomorphism.
\end{enumerate}

\end{prop}

\begin{proof}
We prove (1) and (2) simultaneously. By Corollary \ref{cor:properties-projectives}, there exists a projective object $\cQ$ of $\Perv_{\scS}(X,\O)$ which admits a $\Delta$-filtration and a surjection $\cQ \twoheadrightarrow \cP$. By projectivity, $\cP$ is direct factor of $\cQ$. As $\cQ$ admits a $\Delta$-filtration, $\F(\cQ)$ is a perverse sheaf; it follows that $\F(\cP)$ is also perverse. 

Now we observe that we have
\[
\F \bigl( R\Hom_{\O}(\cP,\Delta_s) \bigr) \ \cong \ R\Hom_{\F} (\F(\cP),\Delta_{s,\F}).
\]
The left hand side is concentrated in degrees $-1$ and $0$ by Corollary \ref{cor:realization-functor}, and the right-hand side is in non-negative degrees since $\F(\cP)$ is perverse. It follows that both complexes are concentrated in degree $0$, and (2) follows. Then one easily deduces the projectivity of $\F(\cP)$ from the fact that $\Ext^{>0}(\F(\cP),\Delta_{s,\F})=0$ for any $s \in \scS$.

(3) By (1) and the remarks before the corollary, $\F(\cP)$ is a perverse sheaf and it admits a $\Delta_{\F}$-filtration. The result follows using Lemma \ref{lem:Delta-flag}.

(4) follows easily from (2) and (3).
\end{proof}

\begin{cor}
\label{cor:projectives-classification}

For any $s \in \scS$, there exists, up to isomorphism, a unique
indecomposable projective object $\cP_s$ in $\Perv_{\scS}(X,\O)$ such
that $\F(\cP_s) \cong \cP_{s,\F}$. Moreover, these objects are
pairwise non isomorphic, and any projective object in
$\Perv_{\scS}(X,\O)$ is a direct sum of objects $\cP_s$, $s \in
\scS$.
\end{cor}

\begin{proof}
By Proposition \ref{prop:existence-projectives} and its proof, there
exists an indecomposable projective $\cP_s$ surjecting onto $\Delta_s$.  By Proposition \ref{prop:properties-projectives-2}(1), $\F(\cP_s)$ is a projective perverse sheaf. Moreover it surjects onto $\Delta_{s,\F}$, hence to $\IC_{s,\F}$. Next, recall that $\Perv_{\scS}(X,\O)$ and $\Perv_{\scS}(X,\F)$ both satisfy the
Krull--Schmidt property (see Lemma \ref{lem:Krull-Schmidt}), and hence
an object is indecomposable if and only if its endomorphism ring is
local. Then indecomposability of $\F(\cP_s)$ follows from Proposition \ref{prop:properties-projectives-2}(4), using the fact (see \cite[Theorem 12.3]{Feit}) that one can lift
idempotents from $\F \otimes_{\O} \End_{\Perv_{\scS}(X,\O)}(\cP_s)$ to
$\End_{\Perv_{\scS}(X,\O)}(\cP_s)$. As $\F(\cP_s)$ surjects to $\IC_{s,\F}$, we deduce that $\F(\cP_s) \cong \cP_{s,\F}$.

Let us show uniqueness. Let $\cQ_s$ be another indecomposable projective such that $\F(\cQ_s) \cong \cP_{s,\F}$. By Proposition \ref{prop:properties-projectives-2}, the identity of $\cP_{s,\F}$ can be lifted to a morphism $f: \cP_s \to \cQ_s$.  By Lemma \ref{lem:perverse-sheaf-0}, $f$ is an isomorphism.

The fact that the $\cP_s$'s are pairwise non isomorphic is obvious, since so are the $\cP_{s,\F}$'s. Finally, let $\cP$ be any projective object in $\Perv_{\scS}(X,\O)$. Then $\F(\cP)$ is a projective perverse sheaf by Proposition \ref{prop:properties-projectives-2}(1), hence there exist integers $n_s$ ($s \in \scS$) and an isomorphism 
\[
\F(\cP) \ \cong \ \bigoplus_{s \in \scS} \cP_{s,\F}^{\oplus n_s}.
\]
By Proposition \ref{prop:properties-projectives-2}(4) this isomorphism can be lifted to a morphism 
\[
\cP \to \bigoplus_{s \in \scS} \cP_{s}^{\oplus n_s},
\]
which is an isomorphism again by Lemma \ref{lem:perverse-sheaf-0}.
\end{proof}

\section{Perverse sheaves on a variety over a finite field}
\label{sec:perverse-finite-field}

\subsection{Generalities on $\O$-modules with an endomorphism}
\label{ss:O-modules-automorphisms}

Given an endomorphism $\phi$ of an $\O$-module $M$ 
and $\lambda\in \O$ let us put 
\[
M_\lambda :=\{ m \in M \mid (\phi - \lambda \cdot \id)^n(m)=0 \text{ for } n \gg 0 \}.
\]
We call $M_\lambda$ the generalized eigenspace of $M$ with eigenvalue $\lambda$.

\begin{prop}
\label{prop:criterion-decomposability-soe}

Let $M$ be a free finite rank $\O$-module
with an endomorphism  $\phi$. Suppose that $\chi(\phi)=0$ for some $\chi \in \O[X]$ such that $\chi=\prod_{\lambda \in \Lambda} (X-\lambda)^{n_\lambda}$ for some finite set $\Lambda \subset \O$ and some positive integers $n_\lambda$. Assume moreover that if $\lambda,\mu \in \Lambda$ have the same image in $\F$ then $\lambda=\mu$. Then $M$ decomposes into the direct sum of its generalized 
eigenspaces, i.e.~the natural map
\begin{equation}
\label{eqn:eigenspaces-M}
\bigoplus_{\lambda\in\O} \, M_\lambda \ \to \ M
\end{equation}
is an isomorphism.
\end{prop}

\begin{proof}
First observe that $M_\lambda=\{0\}$ unless $\lambda \in \Lambda$, so that the sum in the left hand side of \eqref{eqn:eigenspaces-M} reduces to $\bigoplus_{\lambda\in\Lambda} M_\lambda$. Then, consider $M$ as a module over $\O[X]/\langle\chi\rangle$. 
The natural morphism
\[
P_\chi :
\O[X]/\langle\chi\rangle \ \to \ 
\prod_{\lambda \in \Lambda} \O[X]/\langle (X-\lambda)^{n_\lambda}\rangle
\]
is a morphism between two free $\O$-modules of the same finite rank, and its reduction modulo $\pi$ is an isomorphism by our assumption on $\Lambda$ and the Chinese Remainder Theorem. We deduce that $P_\chi$ is already an isomorphism. Pulling back the obvious idempotents from the right 
hand side gives us idempotents $e_\lambda \in \O[X]/\langle \chi \rangle$ ($\lambda \in \Lambda$) such that the assignment $m \mapsto (e_\lambda \cdot m)_{\lambda \in \Lambda}$ is an inverse to \eqref{eqn:eigenspaces-M}. (To prove this ones uses the property that $\langle (X-\lambda)^{n_\lambda} \rangle + \langle (X-\mu)^{n_\mu} \rangle = \O[X]$ for $\lambda,\mu \in \Lambda$ with $\lambda \neq \mu$.)
\end{proof}


Given  a free finite rank $\O$-module $M$ 
with an endomorphism  $\phi$ let us say
 that $M$ is \emph{decomposable under $\phi$} if the natural map
is an isomorphism
$$\bigoplus_{\lambda\in\O} \, M_\lambda \ \xrightarrow{\sim} \ M.$$


\begin{lem}
\label{lem:quotient-decomposable}
Let $M$ be  a free finite rank $\O$-module $M$ 
which is decomposable under an endomorphism  $\phi$. 
Then given a $\phi$-stable submodule $N\subset M$ such that
$M/N$ is free over $\O$ as well, 
 both $N$ and $M/N$ are decomposable under the induced endomorphisms.

\end{lem}

\begin{proof}
What we have to show is that if $m=\sum m_\lambda$ is the generalized 
eigenspace decomposition of $m\in N$, then all $m_\lambda$ already belong
to $N$. 
Linear algebra tells us
that all $m_\lambda$ belong to 
 $ \K \otimes_{\O} N\subset  \K \otimes_{\O} M$.
 It follows that for every $\lambda$
 there exists $n >0$ such that $\pi^n m_\lambda \in N$. 
Therefore the
 image of $m_\lambda$ in $M/N$ is killed by some power of $\pi$, 
hence is zero since $M/N$ has no torsion, i.e.~$m_\lambda \in N$.
\end{proof}

\begin{rmk}
\begin{enumerate}
\item The assumption that $M/N$ has no torsion cannot be removed. 
For example, assume that $\O=\Zl$. Then $M=\Zl^2$  is clearly decomposable
under  $\phi$ 
acting by the matrix
\[
\left(
\begin{array}{cc}
1 & 0 \\
0 & 1+\ell
\end{array}
\right),
\]
and its submodule $N:=\Zl(1,1) \oplus
\Zl ( 0,\ell)$ is stable under $\phi$. However $N$ is
 not decomposable under $\phi$, as will be
explained right after.

\item Observe that 
in the setting of the lemma, it can happen that both $N$ and $M/N$ are
decomposable, but $M$ is not. For a counterexample consider
 $M=\Zl^2$, with $\phi$ acting as
\[
\left(
\begin{array}{cc}
1 & 0 \\
1 & 1+\ell
\end{array}
\right);
\]
and $N=\Zl(0,1)$. Then $M$ is indecomposable, since on its reduction modulo
$\pi$ our $\phi$ is just a Jordan block.
By the way, this $M$ is as a  $\Zl$-module with endomorphism 
isomorphic to our  $N$ from the previous example.
\end{enumerate}

\end{rmk}

Next let 
$A$ be an $\O$-algebra which is finitely generated as an $\O$-module, 
endowed with an $\O$-algebra automorphism $\phi$. 
We switch here to automorphisms  as opposed to endomorphisms to avoid
the introduction of additional terminology later.
We denote by
\[
\Modfr (A,\phi)
\]
the abelian category of pairs $(M,\phi_M)$ where $M$ is a 
finitely generated right $A$-module  and 
$\phi_M : M \xrightarrow{\sim} M$ is an 
automorphism which satisfies 
\[
\phi_M(m \cdot a) = \phi_M(m) \cdot \phi(a)
\]
for $a \in A$, $m \in M$. Morphisms in $\Modfr (A,\phi)$ between objects
$(M,\phi_M)$ and $(N,\phi_N)$ are morphisms of $A$-modules $f : M \to N$ such
that $f \circ \phi_M=\phi_N \circ f$. Objects of $\Modfr (A,\phi)$ will
sometimes be called $(A,\phi)$-modules. By our assumptions on $A$, 
objects of $\Modfr (A,\phi)$ are always
finitely generated as $\O$-modules. 
 We will often omit the automorphism $\phi_M$ from notation. An important
 particular case is when $A=\O$ and $\phi=\id$; in this case we 
call our objects $(\O,\phi)$-modules.

\begin{lem}
\label{lem:existence-geom-projective}
Let 
$(A,\phi)$ be an $\O$-algebra which is finitely generated as an $\O$-module,
together
 with an  automorphism $\phi$. For any $(M,\phi_M)$ in $\Modfr (A,\phi)$, there exists an object $(P,\phi_P)$ in $\Modfr (A,\phi)$ such that $P$ is $A$-free and a surjection $(P,\phi_P) \onto (M,\phi_M)$ in $\Modfr (A,\phi)$.

\end{lem}

\begin{proof}
First, consider $(M,\phi_M)$ as an $(\O,\phi)$-module. Then there exists
$(N,\phi_N)$ in $\Modfr (\O,\phi)$ such that $N$ is $\O$-free and a surjection
$(N,\phi_N) \onto (M,\phi_M)$. Indeed, let $e_1, \ldots, e_n$ be a
family of vectors in $M$ whose images in $\F \otimes_\O M$ form a basis. By Nakayama's lemma, this family generates $M$ as an $\O$-module. Choose a matrix $D \in \mathcal{M}_n(\O)$ such that $\phi_M(e_i) = \sum_j a_{ji} e_j$ ; this matrix is automatically invertible since its image in $\mathcal{M}_n(\F)$ is invertible. Then one can take $N=\O^n$, with $\phi_N$ acting with matrix $D$, and the natural surjection $N \onto M$. Finally, set $P:=N \otimes_{\O} A$, endowed with the automorphism $\phi_P$ induced by $\phi_N$.
\end{proof}

Now let us fix a unit $q\in\O^\times$ 
 which is not a  root of unity. 
Suppose $M$ is an object 
in $\Modfr (\O,\phi)$, i.e. a finitely generated $\O$-module 
with an automorphism $\phi$.  Given some finite subset 
$I \subset \Z$, we say that $M$ \emph{has $q$-weights obtained from $I$,}  if  $$\prod_{i \in I}(\phi-q^i
\cdot \id)$$ 
acts nilpotently on $M$. In addition, let us call $M$
\emph{$q$-decomposable}
if it is decomposable under $\phi$ and has $q$-weights obtained from some finite subset of $\Z$. 
Observe that if we have an exact sequence
\[
N \into M \onto P
\]
and if $N$, resp.~$P$, has $q$-weights obtained from $I$, resp.~$J$, then $M$ has $q$-weights obtained from $I \cup J$. Similarly, if $M$ has $q$-weights obtained from $I$ and $N$ has $q$-weights obtained from $J$, then $M \otimes_{\O} N$ has $q$-weights obtained from $I+J=\{i+j \mid i \in I, \ j \in J\}$.

\begin{rmk}\label{prop:criterion-decomposability}
  If $(M,\phi)$ 
is free of finite rank over $\O$ and has $q$-weights obtained from $I$,
  and if moreover the map $i\mapsto \bar q^i$ is an injection $I\hookrightarrow \F$ (where $\bar q$ is the image of $q$ in $\F$), then
  $M$ is $q$-decomposable by Proposition \ref{prop:criterion-decomposability-soe}. This criterion will be crucial for our proofs.
\end{rmk}

\begin{lem}
\label{lem:free-resolution-weights}

Let $M$ be in $\Modfr (\O,\phi)$, and assume $M$ has $q$-weights obtained from
$I$. Then there exists an object $M'$ in $\Modfr (\O,\phi)$, which is
$\O$-free and has $q$-weights obtained from $I$, together with
 a surjection $M' \twoheadrightarrow M$.

\end{lem}

\begin{proof}
$M$ can be considered as a finitely generated 
module over the $\O$-free algebra
\[
R:=\O[X]/ \langle \prod_{i \in I}(X-q^i \id)^{n} \rangle
\]
for some positive integer $n$. Then take $M'$ to be a free 
finite rank $R$-module surjecting to $M$, with $\phi$ acting as multiplication
by $X$. Observe that multiplication by $X$ is an automorphism since its
reduction modulo $\pi$ is invertible by our assumption $q\in \O^\times$.
\end{proof}

Now let
$(A,\phi)$ be an $\O$-algebra, 
which is a free finite rank $\O$-module, 
endowed with an $\O$-algebra automorphism $\phi$,
such that $A$ is $q$-decomposable. 
We denote by
\[
\Modfrdec (A,\phi)
\]
the full  subcategory of $\Modfr (A,\phi)$ whose objects are
free of finite rank over $\O$ and $q$-decomposable as $(\O,\phi)$-modules .

\begin{lem}
\label{lem:projective-decomposable}

Assume that $(A,\phi)$ is $\O$-free of finite rank and $q$-decomposable as above.
For any $M$ in $\Modfrdec (A,\phi)$ there exists an object $Q$ in $\Modfrdec
(A,\phi)$ which is projective as an $A$-module along with
 a surjection $Q \onto M$.

\end{lem}

\begin{proof}
Take $Q:=M \otimes_{\O} A$, with the natural automorphism.
\end{proof}

\subsection{Categories with an autoequivalence}

Let $\sA$ be an abelian category endowed with an autoequivalence $\Phi : \sA \xrightarrow{\sim} \sA$. Let $\sA[\Phi]$ be the category defined as follows:
\begin{itemize}
\item objects are pairs $(M,s_M)$ where $M$ is an object of $\sA$ and $s_M : M \xrightarrow{\sim} \Phi(M)$ is an isomorphism;
\item morphisms between $(M,s_M)$ and $(N,s_N)$ are morphisms $f : M \to N$ in $\sA$ such that the following diagram commutes:
\[
\xymatrix{
M \ar[r]^-{s_M} \ar[d]_-f & \Phi M \ar[d]^-{\Phi f}\\
N \ar[r]^-{s_N} & \Phi N.
}
\]
\end{itemize} 
We might think of $\sA$ as an abelian category with an action of the group
$\mathbb Z$ and of $\sA [\phi]$ as the category of some sort of $\mathbb Z$-equivariant
objects of $\sA$.
Certainly $\sA [\Phi]$ is again abelian.

\begin{ex}
\label{ex:A-phi-modules}
With the notation of \S \ref{ss:O-modules-automorphisms}, the category
$\Modfr (A,\phi)$ is the category $\sA[\Phi]$ where $\sA=\Modfr A$ and
$\Phi$ sends $M$ to the right $A$-module whose underlying $\O$-module
is $M$, and whose $A$-action is the morphism $M \otimes_{\O} A \to M$
given by $m\otimes a \mapsto m \cdot \phi(a)$.
\end{ex}

\begin{lem}
\label{lem:geom-projective}

Let $(\sA, \Phi)$ be an abelian category with an autoequivalence, and let $(P,t) \in \sA[\Phi]$ be an object with $P \in \sA$ projective. Then for any $(M,s) \in \sA [\Phi]$
we have
\[
\Hom^i_{\cDb(\sA [\Phi])} ((P,t), (M,s)) = 0 \quad \text{ for } i \geq 2.
\]

\end{lem}

\begin{proof}
By definition, a morphism from $(P,t)$ to $(M,s)[i]$ in $\cDb (\sA[\Phi])$ is a diagram of morphisms of complexes
\[
(M,s)[i] \leftarrow C^{\bullet} \xrightarrow{\qis} (P,t)
\]
where $C^{\bullet}$ is a bounded complex of object of
$\sA[\Phi]$. Clearly, one can assume that this complex is concentrated
in non-positive degrees. To prove the lemma it suffices to prove that
there exists a complex $D^{\bullet}$ and a quasi-isomorphism $f :
D^{\bullet} \to C^{\bullet}$ whose composition with our morphism
$C^{\bullet} \to (M,s)[i]$ is zero. To establish this it is enough to show that there exists a complex $D^{\bullet}$ concentrated in degrees $-1$ and $0$ and a quasi-isomorphism $f : D^{\bullet} \to C^{\bullet}$.

Write the complex $C^{\bullet}$ as
\[
C^{\bullet} = (\cdots \to (C^{-2},u^{-2}) \xrightarrow{d^{-2}} (C^{-1},u^{-1}) \xrightarrow{d^{-1}} (C^0,u^0) \to 0 \to \cdots).
\]
Let $B^{-1}:=C^{-1}/\ker(d^{-1})$, and let $p^{-1} : C^{-1} \to B^{-1}$ be the projection. We consider the object $(B^{-1},v^{-1})$ in $\sA[\Phi]$, where $v^{-1} : B^{-1} \xrightarrow{\sim} \Phi(B^{-1})$ is the isomorphism induced by $u^{-1}$. As $C^{\bullet}$ maps quasi-isomorphically to $(P,t)$, we have an exact sequence
\[
(B^{-1},v^{-1}) \into (C^0,u^0) \onto (P,t).
\]
As $P$ is projective, the image of this exact sequence in $\sA$ splits, hence we have an isomorphism $(C^0,u^0) \cong (B^{-1} \oplus P,w^0)$ in $\sA[\Phi]$, where
\[
w^0 = \left(
\begin{array}{cc}
v^{-1} & a \\
0 & t
\end{array}
\right) : B^{-1} \oplus P \xrightarrow{\sim} \Phi(B^{-1}) \oplus \Phi(P)
\]
for some morphism $a : P \to \Phi(B^{-1})$ in $\sA$. Let us fix such an isomorphism. As $P$ is projective, there exists a morphism $b : P \to \Phi(C^{-1})$ in $\sA$ whose composition with the projection $\Phi(p^{-1}) : \Phi(C^{-1}) \to \Phi(B^{-1})$ is $a$. Then define $D^0:=C^{-1} \oplus P$, and consider the object $(D^0,x^0)$ of $\sA[\Phi]$, where
\[
x^0 := \left(
\begin{array}{cc}
u^{-1} & b \\
0 & t
\end{array}
\right) : C^{-1} \oplus P \xrightarrow{\sim} \Phi(C^{-1}) \oplus \Phi(P).
\]
Let also $(D^{-1},x^{-1}):=(C^{-1},u^{-1})$. Then $(D^0,x^0)$ is an object of $\sA[\Phi]$, and there exists a natural injection $(D^{-1},x^{-1}) \into (D^0,x^0)$ whose cokernel is $(P,t)$. Moreover we have a quasi-isomorphism
\[
\xymatrix{
\cdots \ar[r] & 0 \ar[r] \ar[d] & (D^{-1},x^{-1}) \ar@{=}[d] \ar[r] & (D^0,x^0) \ar[d]^-{c} \ar[r] & 0 \ar[r] \ar[d] & \cdots \\
\cdots \ar[r] & (C^{-2},u^{-2}) \ar[r] & (C^{-1},u^{-1}) \ar[r] & (C^0,u^0) \ar[r] & 0 \ar[r] & \cdots
}
\]
where
\[
c=\left(
\begin{array}{cc}
p^{-1} & 0 \\
0 & \id_P
\end{array}
\right) : C^{-1} \oplus P \to B^{-1} \oplus P.
\]
This finishes the proof.
\end{proof}

\subsection{Perverse sheaves on a variety over a finite field}
\label{ss:perverse-sheaves-Fq}

In this subsection we write $\E$ for either $\O$ or $\F$.

Let $X_\circ$ be a variety over a finite field $\Fq$, endowed with a finite stratification
\[
X_\circ =\bigsqcup_{s \in \scS} X_{s,\circ}
\]
by affine spaces. Let also $X:= X_\circ \times_{\mathrm{Spec}(\Fq)}
\mathrm{Spec}(\Fqb)$, endowed with the induced stratification (which
we also denote by $\scS$). We assume that this stratification
satisfies the condition \eqref{eq:strata assumption}. 
Let
\[
\For : \cDb_{c}(X_\circ ,\E) \to \cDb_{c}(X,\E)
\]
be the inverse image functor under the natural projection $X \to X_\circ$. We
denote by $\cDb_{\scS}(X_\circ ,\E)$ the full subcategory of
$\cDb_{c}(X_\circ ,\E)$ consisting of objects $\cF$ such that $\For(\cF)$
is in $\cDb_{\scS}(X,\E)$. As for $X$, $\cDb_{\scS}(X_\circ ,\E)$ is endowed
with the perverse t-structure whose heart we denote by
$\Perv_{\scS}(X_\circ ,\E)$ (for the perversity $p$ if $\E=\O$). Note that if $\cM_\circ$ is in $\cDb_{\scS}(X_\circ ,\E)$ then $\cM_\circ$ is in $\Perv_{\scS}(X_\circ,\E)$ if and only if $\For(\cM_\circ)$ is in $\Perv_{\scS}(X,\E)$. As in \cite{BBD}, objects of
$\Perv_{\scS}(X_\circ,\E)$ will always be denoted with a subscript
``${}_\circ$''. Their image under the functor $\For$ will usually be
denoted by the same symbol, with the subscript removed.

Let $\pt_\circ := \Spec \F_q$ and $f : X_\circ \to \pt_\circ$ denote the
projection. Let $\E(m) \in \cDb_c(\pt_\circ , \E)$ denote the $m^{th}$ Tate
sheaf (see \cite[\S 2.13]{Del1}). Given any object $\cF_\circ \in \cDb_{c}(X_\circ,\E)$ we define its $m^{th}$
Tate twist by $\cF_\circ(m) := \cF_\circ \stackrel{L}{\otimes}_\E f^*\E(m)$.

Before we go into the technical details, let us discuss
some generalities concerning Galois actions on \'etale cohomology.
Given a scheme $X_\circ$ over a field $k$ and a field extension 
$K/k$, certainly the Galois group will act on the scheme $X=K\times_kX_\circ$
and thus also will act on its \'etale cohomology. More generally, given 
an \'etale sheaf $\cF_\circ$ on $X_\circ$ and its pullback
$\cF$ to $X$, the Galois group will act on the \'etale cohomology of $X$ with
coefficients in $\cF$. Formally, if $\omega:X\rightarrow X_\circ$ is 
the extension of 
scalars, for 
any $\gamma$ in the Galois group and 
letting $\gamma=\gamma\times \operatorname{id}$,
we have 
 $\omega \circ \gamma=\omega$ and from there get 
isomorphisms $\gamma^\ast \omega^\ast\cF_\circ \xrightarrow{\sim} \omega^\ast\cF_\circ$
alias $\gamma^\ast\cF \xrightarrow{\sim} \cF$, which make  $\cF$ to
what one might call a Galois 
equivariant sheaf on $X$. The  action of $\gamma$ on $\coH{n}(X;\cF)$ is then 
defined as the composition 
$\coH{n} (X;\cF)\rightarrow \coH{n} (X;\gamma^\ast \cF)
\rightarrow \coH{n}(X;\cF)$
of the pull-back by $\gamma$ on cohomology with the map induced by
our isomorphism $\gamma^\ast\cF \xrightarrow{\sim} \cF$. 
Even more generally, 
if $\cG_\circ$ is another \'etale sheaf or complex of sheaves on $X_\circ$,
we will get a Galois action the Ext groups between $\cF$ and $\cG$
by structure transport,
and we could go on like that.

We will use these Galois
actions in the case that the field extension 
$K/k$ consists of our finite field $\F_q$ with an  algebraic closure 
$\Fqb$. In this case, additional features try to create confusion.
Namely, given a prime $p$, on any commutative ring $A$ over $\Z/p\Z$, 
even on any scheme $S$ over $\Z/p\Z$ there is a canonical endomorphism,
the absolute Frobenius $\Fr_{\mathrm{ab}}:S\rightarrow S$, given on rings as the map
$\Fr_{\mathrm{ab}}:A\rightarrow A$ with $a\mapsto a^p$. This is a natural
transformation from
the identity functor on the category of schemes over $\Z/p\Z$ to itself.
In particular, given two schemes $S,T$ over $\Z/p\Z$, on their
product we have
$\Fr_{\mathrm{ab}}=(\Fr_{\mathrm{ab}}\times \operatorname{id})\circ
 (\operatorname{id}\times \Fr_{\mathrm{ab}})$.
From 
the remarkable fact that for
any \'etale morphism $\pi:T\rightarrow S$ the diagram
$$\xymatrix{
T \ar[r]^-{\Fr_{\mathrm{ab}}} \ar[d]_{\pi}
& T \ar[d]^\pi\\
S \ar[r]^-{\Fr_{\mathrm{ab}}} &S}$$
is cartesian, we get for any \'etale sheaf $\cF$ on $S$
a natural isomorphism $\Fr_{\mathrm{ab}}^\ast\cF\xrightarrow{\sim}\cF$, and
the map induced by this on the \'etale cohomology of $\cF$ is just the
identity.

Now in our case of a scheme $X_\circ$ over $\F_q$ with
$q=p^r$, on the scheme
  $X=\Fqb\times_{\Fq}X_\circ$  we 
have as endomorphisms the arithmetic Frobenius 
$\Fr_{\mathrm{a}}=\Fr_{\mathrm{ab}}^r\times \operatorname{id}$ and the geometric Frobenius
$\Fr=\Fr_{\mathrm{g}}=\operatorname{id}\times\Fr_{\mathrm{ab}}^r 
=\operatorname{id}\times\Fr_\circ $. Clearly 
$\Fr_{\mathrm{g}}\circ \Fr_{\mathrm{a}}=\Fr_{\mathrm{a}}\circ \Fr_{\mathrm{g}}=\Fr_{\mathrm{ab}}^r $, thus our natural isomorphism
leads to a natural isomorphism
$\Fr_{\mathrm{a}}^\ast(\Fr^\ast\cF)\xrightarrow{\sim}\cF$ for any \'etale sheaf 
$\cF$ on $X$. On the other hand, 
it also leads to a natural 
isomorphism $\cF_\circ
\xrightarrow{\sim} \Fr_\circ^* \cF_\circ$ and by pulling back we obtain
a natural 
isomorphism $\cF
\xrightarrow{\sim} \Fr^* \cF$ for any \'etale sheaf $\cF_\circ$ on $X_\circ$.
To get back from 
this isomorphism
our  Galois equivariance $\gamma^\ast \cF\xrightarrow{\sim}  \cF$ 
alias $\Fr_{\mathrm{a}}^\ast \cF\xrightarrow{\sim} \cF$ from above, 
just apply $\Fr_{\mathrm{a}}^\ast$ to it and postcompose with 
the natural isomorphism
$\Fr_{\mathrm{a}}^\ast(\Fr^\ast\cF)\xrightarrow{\sim}\cF$ discussed before.
This is explained in \cite[p.~84]{SGA4.5}.
In the following, we always work directly 
with the Frobenius action coming from the natural 
isomorphism $\cF
\xrightarrow{\sim} \Fr^* \cF$, which might be conceptually 
a less direct approach,
but permits to connect much more directly to the existing literature

So let $\Fr : X \to X$ be the Frobenius endomorphism, obtained from the
Frobenius $\Fr_\circ : X_\circ \to X_\circ$ by base change. For any $\cM_\circ$ in
$\cDb_{\scS}(X_\circ ,\E)$ there exists a canonical isomorphism $\cM_\circ
\xrightarrow{\sim} \Fr_\circ^* \cM_\circ$, hence also a canonical isomorphism
$\psi_\cM : \cM \xrightarrow{\sim} \Fr^* \cM$, where $\cM=\For(\cM_\circ)$.
For details, see \cite[Expos{\'e} XIV, \S 2.1]{SGA5} or \cite[\S VI.13]{Milne}. We will need the following
well-known result.

\begin{lem}
\label{lem:Fr-equivalence}

The functor
\[
\Fr^* : \cDb_{\scS}(X,\E) \to \cDb_{\scS}(X,\E)
\]
is an equivalence of categories, which restricts to an equivalence $\Perv_{\scS}(X,\E) \xrightarrow{\sim} \Perv_{\scS}(X,\E)$.

\end{lem}

\begin{proof}
The category $\cDb_{\scS}(X,\E)$ is generated by the essential image of the functor $\For : \cDb_{\scS}(X_\circ,\E) \to \cDb_{\scS}(X,\E)$ (e.g.~by standard or costandard objects). Hence it is enough to prove that for any $\cM_\circ,\cN_\circ$ in $\cDb_{\scS}(X_\circ,\E)$, with $\cM:=\For(\cM_\circ)$, $\cN:=\For(\cN_\circ)$, the morphism
\[
\Hom_{\cDb_{\scS}(X,\E)}(\cM,\cN) \to \Hom_{\cDb_{\scS}(X,\E)}(\Fr^* \cM,\Fr^* \cN)
\]
is an isomorphism. To prove this fact it is enough to prove that the adjunction morphism $\cN \to \Fr_* \Fr^* \cN$ is an isomorphism. However this morphism is obtained by extension of scalars from the adjunction morphism $\cN_\circ \to \Fr_{\circ *} \Fr_{\circ}^* \cN_\circ$, which is an isomorphism since $\Fr_\circ$ is a universal homeomorphism, see \cite[Expos{\'e} XIV, \S 1, Proposition 2(a)]{SGA5}.
\end{proof}

Recall that, by \cite[Proposition 5.1.2]{BBD}, the functor which sends $\cM_\circ$ to the pair $(\cM,\psi_\cM)$, where $\cM=\For(\cM_\circ)$ and $\psi_\cM$ is defined above, induces an equivalence of abelian categories
\begin{equation}
\label{eqn:equiv-Perv-X_0}
\Perv_{\scS}(X_\circ ,\E) \xrightarrow{\sim} \Perv_{\scS}(X,\E)[\Fr^*].
\end{equation}
Using this equivalence, one can make sense of the tensor product $V \otimes_{\E} \cM_\circ$ of an object $\cM_\circ$ of $\Perv_{\scS}(X_\circ ,\E)$ with an $\E$-free $(\E,\phi)$-module $(V,\phi_V)$. Indeed using \eqref{eqn:equiv-Perv-X_0} one can see $(V,\phi_V)$ as an object of $\Perv(\pt_\circ,
\E)$, and we define $V \otimes_{\E} \cM_\circ$ as the tensor product of the pullback of this object to $X_\circ$ with $\cM_\circ$.


\begin{ex}
\label{ex:Tate-sheaf}
When $X_\circ=\pt_\circ=\mathrm{Spec}(\Fq)$, we have $\Fr=\id$, and $\Perv(\pt, \E) \cong \Modfr \E$. The image of $\E(m) \in \Perv(\pt_\circ, \E)$ under \eqref{eqn:equiv-Perv-X_0} is a rank $1$ free right $\E$-module endowed with the automorphism given by multiplication by $q^{m}$.
\end{ex}


\begin{lem}
\label{lem:proj-generator-X_0}

There exists an object $\cP_\circ$ in $\Perv_{\scS}(X_\circ,\E)$ such that $\cP=\For(\cP_\circ)$ is a projective generator of $\Perv_{\scS}(X,\E)$.

\end{lem}

\begin{proof}
We write the details only for $\E=\O$; the case $\E=\F$ is similar. By the proof of Corollary \ref{cor:properties-projectives}(1), it is enough to prove that the objects $\cP$ constructed in Proposition \ref{prop:existence-projectives} can be chosen in such a way that they are in the essential image of the functor $\For$ or equivalently, using equivalence \eqref{eqn:equiv-Perv-X_0}, in such a way that there exists an isomorphism $\cP \xrightarrow{\sim} \Fr^* \cP$. Now, using the notation of this proof, the $\O$-module $E$ has a natural automorphism $\phi$ induced by $\Fr^*$; it is enough to observe that $E_{\free}$ can be endowed with an automorphism $\phi_{\free}$ such that the quotient $E_{\free} \onto E$ is compatible with $\phi$ and $\phi_{\free}$, which follows from Lemma \ref{lem:existence-geom-projective} (for $A=\O$). (See the proof of Proposition \ref{prop:geom-projective-weights} below for details on this construction.)
\end{proof}

Let $\cM_\circ ,\cN_\circ $ be in $\cDb_{\scS}(X_\circ ,\E)$, with $\cM:=\For(\cM_\circ )$, $\cN:=\For(\cN_\circ )$. Composing the isomorphism
\[
\left\{
\begin{array}{ccc}
\Hom_{\cDb_{\scS}(X,\E)}(\cM,\cN) & \to & \Hom_{\cDb_{\scS}(X,\E)}(\Fr^* \cM, \Fr^* \cN) \\
f & \mapsto & \psi_\cN \circ f \circ \psi_\cM^{-1}
\end{array}
\right.
\]
with the inverse of the isomorphism
\[
\Hom_{\cDb_{\scS}(X,\E)}(\cM,\cN) \to \Hom_{\cDb_{\scS}(X,\E)}(\Fr^* \cM,\Fr^* \cN)
\]
induced by $\Fr^*$ (see Lemma \ref{lem:Fr-equivalence}) we obtain an automorphism $\phi_{\cM,\cN}$ of the $\E$-module $\Hom_{\cDb_{\scS}(X,\E)}(\cM,\cN)$. Hence $\Hom_{\cDb_{\scS}(X,\E)}(\cM,\cN)$ is naturally an object of $\Modfr (\E,\phi)$, which we denote by $\Hom_{\E}(\cM,\cN)$ for simplicity.

\begin{rmk}
\label{rmk:cohomology-Frobenius}
Using the same arguments as in Lemma \ref{lem:Fr-equivalence}, for any
variety $X_\circ$ over $\mathrm{Spec}(\Fq)$ (not necessarily satisfying
our assumptions) we obtain an automorphism of
$\coHc{\bullet}(X,\E)=\Hom^\bullet(\underline{\E}_{\pt},f_!
\underline{\E}_X)$ induced by the Frobenius morphism. (Here, $f:X_\circ
\to \pt_\circ$ is the projection.) For example, with our conventions the
automorphism of $\coHc{2}(\mathbb{A}^1,\E) \cong \coHc{2}(\mathbb{A}^1
\smallsetminus \{0\},\E)$ is multiplication by $q^{-1}$, and the
automorphism of $\coHc{1}(\mathbb{A}^1 \smallsetminus \{0\},\E)$ is the
identity.
\end{rmk}

Let us fix an object $\cP_\circ$ as in Lemma \ref{lem:proj-generator-X_0}, and let
\[
A:=\End_{\Perv_{\scS}(X,\E)}(\cP).
\]
Since $\cP$ is a projective generator of $\Perv_{\scS}(X,\E)$ we have an equivalence of categories
\begin{equation}
\label{eqn:perverse-sheaves-modules-X}
\mathsf{F}_{\cP}:=\Hom(\cP,-) : \Perv_{\scS}(X,\E) \xrightarrow{\sim} \Modfr A
\end{equation}
(see \cite[Exercise on p.~55]{Ba}). Moreover, as $\cP=\For(\cP_\circ)$, we have an automorphism $\phi:=\phi_{\cP,\cP}$ of $A$, so that we can consider the category $\Modfr(A,\phi)$ as in \S \ref{ss:O-modules-automorphisms}.

\begin{lem}
\label{lem:perverse-sheaves-modules}

The functor which sends $\cM_\circ$ to $(\Hom(\cP,\cM), \phi_{\cP,\cM})$ (where as usual $\cM=\For(\cM_\circ)$) induces an equivalence of abelian categories
\begin{equation}
\label{eqn:perv-0-modules}
\Perv_{\scS}(X_\circ ,\E) \ \xrightarrow{\sim} \ \Modfr (A,\phi).
\end{equation}
Moreover the following diagram commutes
\[
\xymatrix@C=1.5cm{
\Perv_{\scS}(X_\circ ,\E) \ar[r]_-{\sim}^-{\eqref{eqn:perv-0-modules}} \ar[d]_{\For} & \Modfr (A,\phi) \ar[d] \\
\Perv_{\scS}(X,\E) \ar[r]_-{\sim}^-{\eqref{eqn:perverse-sheaves-modules-X}} & \Modfr A,
}
\]
where the right vertical functor sends $(M,\phi_M)$ to $M$.

\end{lem}

\begin{proof}
Consider the equivalences
\[
\Perv_{\scS}(X_\circ ,\E) \ \overset{\eqref{eqn:equiv-Perv-X_0}}{\cong} \ \Perv_{\scS}(X,\E)[\Fr^*] \ \cong \ \bigl( \Modfr A \bigr) [\mathsf{F}_{\cP} \circ \Fr^* \circ (\mathsf{F}_{\cP})^{-1}]
\]
where the second equivalence is induced by $\mathsf{F}_{\cP}$. We claim that the functors $\Phi$ (defined as in Example \ref{ex:A-phi-modules}) and $\mathsf{F}_{\cP} \circ \Fr^* \circ (\mathsf{F}_{\cP})^{-1}$ are isomorphic, or equivalently that the functors $\Phi \circ \mathsf{F}_{\cP}$ and $\mathsf{F}_{\cP} \circ \Fr^*$ are isomorphic: indeed an isomorphism is provided by the composition of isomorphisms
\[
\Hom(\cP,\cM) \xrightarrow{\Fr^*} \Hom(\Fr^* \cP, \Fr^* \cM) \xrightarrow{(-) \circ \psi_{\cP}} \Hom(\cP,\Fr^* \cM)
\]
for $\cM$ in $\Perv_{\scS}(X,\E)$. Equivalence \eqref{eqn:perv-0-modules} follows. We leave it to reader to check that this equivalence can be described by the formula in the lemma, and that the diagram commutes.
\end{proof}

\subsection{Derived categories of perverse sheaves}

In this subsection we keep the same assumptions and notation as in \S \ref{ss:perverse-sheaves-Fq}. The subcategories $\Perv_{\scS}(X_\circ ,\E) \subset \cDb_{\scS}(X_\circ ,\E)$, $\Perv_{\scS}(X,\E) \subset \cDb_{\scS}(X,\E)$ are hearts of t-structures, and the triangulated categories $\cDb_{\scS}(X_\circ ,\E)$ and $\cDb_{\scS}(X,\E)$ have natural filtered analogues. Hence by \cite{Be} one has realization functors
\[
\mathsf{real}_\circ : \cDb \Perv_{\scS}(X_\circ ,\E) \to \cDb_{\scS}(X_\circ ,\E), \quad \mathsf{real} : \cDb \Perv_{\scS}(X,\E) \to \cDb_{\scS}(X,\E).
\]
Moreover, these functors fit in the following commutative diagram:
\[
\xymatrix@C=1.5cm{
\cDb \Perv_{\scS}(X_\circ ,\E) \ar[r]^-{\mathsf{real}_\circ } \ar[d] & \cDb_{\scS}(X_\circ ,\E) \ar[d] \\
\cDb \Perv_{\scS}(X,\E) \ar[r]^-{\mathsf{real}} & \cDb_{\scS}(X,\E)
}
\]
where vertical functors are induced by $\For$ (see \cite[Lemma A.7.1]{Be}). It is known that the functor $\mathsf{real}$ is an equivalence of categories, see Corollary \ref{cor:realization-functor} if $\E=\O$ and \cite[Corollary 3.3.2]{BGS}\footnote{In \cite{BGS} the authors work with coefficients in $\Ql$; however the same proof applies to coefficients in $\F$.} if $\E=\F$.

\begin{prop}
\label{prop:realization-equivalence-0}

The functor $\mathsf{real}_\circ$ is an equivalence of categories.

\end{prop}

\begin{proof}
The category $\Perv_{\scS}(X_\circ ,\E)$ generates $\cDb_{\scS}(X_\circ ,\E)$. Hence it is enough to prove that for any $\cM_\circ ,\cN_\circ $ in $\Perv_{\scS}(X_\circ ,\E)$ and any $i \in \Z$ the natural morphism
\[
\Hom_{\cDb \Perv_{\scS}(X_\circ ,\E)}^i(\cM_\circ ,\cN_\circ ) \to \Hom_{\cD_{\scS}(X_\circ ,\E)}^i(\cM_\circ ,\cN_\circ )
\]
is an isomorphism. We prove this fact by induction on $i$. In fact both sides vanish if $i<0$; if $i=0$ then the morphism is an isomorphism by definition; and the morphism is an isomorphism also if $i=1$ by \cite[Remarque 3.1.7(ii)]{BBD}.

Now, assume that $i \geq 2$, and that the result is known for $i-1$ (and any $\cM_\circ ,\cN_\circ$). Let $\cP_\circ$ be an object of $\Perv_{\scS}(X_\circ ,\E)$ surjecting onto $\cM_\circ$ and such that $\For(\cP_\circ )$ is projective. (Such an object exists by Lemma \ref{lem:perverse-sheaves-modules} and Lemma \ref{lem:existence-geom-projective}, or an obvious analogue if $\E=\F$.) Consider an exact sequence
\[
\ker_\circ \into \cP_\circ \onto \cM_\circ,
\]
and the associated commutative diagram with exact rows
\[
{\small
\xymatrix@C=0.4cm{
\Hom^{i-1}_{\cDb \mathsf{P}_\circ}(\cP_\circ,\cN_\circ) \ar[r] \ar[d] & \Hom^{i-1}_{\cDb \mathsf{P}_\circ}(\ker_\circ , \cN_\circ ) \ar[r] \ar[d] & \Hom^i_{\cDb \mathsf{P}_\circ}(\cM_\circ ,\cN_\circ ) \ar[r] \ar[d] & \Hom^i_{\cDb \mathsf{P}_\circ }(\cP_\circ ,\cN_\circ ) \ar[d] \\
\Hom^{i-1}_{\mathsf{D}_\circ }(\cP_\circ ,\cN_\circ ) \ar[r] & \Hom^{i-1}_{\mathsf{D}_\circ}(\ker_\circ , \cN_\circ ) \ar[r] & \Hom^i_{\mathsf{D}_\circ}(\cM_\circ ,\cN_\circ ) \ar[r] & \Hom^i_{\mathsf{D}_\circ}(\cP_\circ ,\cN_\circ).
}
}
\]
Here $\mathsf{P}_\circ:= \Perv_{\scS}(X_\circ ,\E)$ and $\mathsf{D}_\circ :=\cDb_{\scS}(X_\circ ,\E)$. The first two vertical morphisms in this diagram are isomorphisms by induction. We claim that both $\E$-modules in the right column vanish, which will conclude the proof. Indeed the top right term is $0$ by Lemma \ref{lem:perverse-sheaves-modules} and Lemma \ref{lem:geom-projective} (since $i \geq 2$). The lower right term is also $0$ by the exact sequences \cite[(5.1.2.5)]{BBD} and the fact that
\[
\Hom_{\cDb_{\scS}(X,\E)}^k(\For(\cP_\circ),\For(\cN_\circ))=0
\]
for $k \geq 1$. (The latter vanishing follows from the facts that $\mathsf{real}$ is an equivalence and that $\For(\cP_\circ)$ is a projective perverse sheaf.)
\end{proof}

\begin{rmk}
\label{rmk:automorphisms-Hom}
Let $\Phi$ be the equivalence considered in Example \ref{ex:A-phi-modules}. We denote similarly the equivalence induced on derived categories, so that every object $M_\circ$ in $\cDb \Modfr (A,\phi)$, with image $M$ in $\cDb \Modfr A$, comes equipped with an isomorphism $\phi_M : M \xrightarrow{\sim} \Phi(M)$. If $M_\circ ,N_\circ$ are in $\cDb \Modfr (A,\phi)$ then there is a canonical automorphism of $\Hom_{\cDb \Modfr A}(M,N)$ which sends a morphism $f$ to $\Phi^{-1}(\phi_N \circ f \circ \phi_M^{-1})$. It is not difficult to check that if $\cM_\circ ,\cN_\circ $ are in $\cDb_{\scS}(X_\circ ,\E)$ and if $M_\circ ,N_\circ $ are their images under the equivalence
\[
\cDb_{\scS}(X_\circ ,\E) \xrightarrow{\mathsf{real}_{\circ}^{-1}}  \cDb \Perv_{\scS}(X_\circ ,\E) \xrightarrow{\eqref{eqn:perv-0-modules}} \cDb \Modfr (A,\phi),
\]
then this automorphism of $\Hom(M,N)$ coincides with the automorphism $\phi_{\cM,\cN}$ of $\Hom(\cM,\cN)$ defined in \S \ref{ss:perverse-sheaves-Fq} under the isomorphism $\Hom(\cM,\cN) \cong \Hom(M,N)$ induced by the equivalence $\mathsf{F}_{\cP} \circ \mathsf{real}^{-1}$.
\end{rmk}

\part{Formality of the constructible derived category of the flag variety}
\label{pt:flag-variety}

\section{Bounding weights}
\label{sec:bounding-weights}

As in \S\ref{ss:perverse-sheaves-Fq} we use $\E$ to denote either $\F$ or $\O$. We fix a finite field $\Fq$ (where $q$ is a prime power which is invertible in $\F$) and an algebraic closure $\Fqb$ of $\Fq$.

\subsection{Notation}
\label{ss:notation}

We let $G_\circ$ be a split connected reductive algebraic group over $\Fq$, $B_\circ \subset G_\circ $ be a Borel subgroup, and $T_\circ \subset B_\circ $ be a maximal torus. We let $W$ be the Weyl group of $G_\circ $ with respect to $T_\circ $. The main player in the rest of this paper will be the flag variety
\[
X_\circ :=G_\circ / B_\circ ,
\]
endowed with the Bruhat stratification
\[
X_\circ =\bigsqcup_{w \in W} X_{w,\circ}, \qquad X_{w,\circ}:=B_\circ w B_\circ / B_\circ .
\]
For any subset $I \subset W$, we denote by $X_{I,\circ}$ the union of the $X_{v,\circ}$ for $v \in I$. Following the notation in the previous sections, we denote by $i_{w} : X_{w,\circ} \into X_\circ$ the inclusion, and we omit the index ``$\circ$'' when considering the varieties obtained by extension of scalars from $\Fq$ to $\Fqb$. The category of $\E$-perverse sheaves on $X_\circ$ (resp.~$X$) for the Bruhat stratification will be denoted by $\Perv_{(B_\circ)}(X_\circ,\E)$ (resp.~$\Perv_{(B)}(X,\E)$), and similarly for derived categories. We use similar notation for $B_\circ$-stable subvarieties of $X_\circ$.

Let $B^-_\circ \subset G_\circ$ be the Borel subgroup opposite to $B_\circ$ with respect to $T_\circ$. We will also consider the opposite Bruhat cells
\[
X^-_{w,\circ} := B^-_\circ w B_\circ / B_\circ \ \subset \ X_\circ
\]
($w \in W$).

The choice of $B_\circ \subset G_\circ$ determines a subset $S \subset W$ of simple reflections and of positive roots, which we choose such that the Lie algebra of the unipotent radical of $B_\circ$ is spanned by positive root spaces. If $s \in S$, we denote by $P_{s,\circ}$ the associated minimal standard parabolic subgroup. We let $X^s_\circ :=G_\circ / P_{s,\circ}$. This variety has a natural stratification by Bruhat cells:
\[
X^s_\circ = \bigsqcup_{w \in W^s} X^s_{w,\circ}, \qquad X^s_{w,\circ}:=B_\circ w P_{s,\circ}/P_{s,\circ}
\]
where $W^s:=\{w \in W \mid ws>w\}$. As for $X_\circ$, we denote by $\Perv_{(B_\circ )}(X^s_\circ ,\E)$ (resp.~$\Perv_{(B)}(X^s,\E)$) the category of $\E$-perverse sheaves on $X^s_\circ$ (resp.~$X^s$) for the Bruhat stratification.

We denote by $\ell$ the length function on $W$ with respect to $S$ and
by $w_0$ the longest element of $W$.

We denote by $\pi_s : X_\circ \to X_\circ^s$ the natural projection. We have functors
\[
\xymatrix@C=3.5cm{
\Perv_{(B_\circ )}(X_\circ ,\E) \ar@<0.5ex>[r]^-{\pi_s{}_{[!]}:=\pH^0(\pi_{s!} (-)[1])} & \Perv_{(B_\circ )}(X^s_\circ ,\E) \ar@<0.5ex>[l]^-{\pi_s^{[!]}:=\pi_s^![-1]},
}
\]
and similarly for the categories $\Perv_{(B)}(X,\E)$ and $\Perv_{(B)}(X^s,\E)$.
The functor $\pi_s^{[!]}$ is exact, and is right adjoint to $\pi_s{}_{[!]}$. It follows that $\pi_s{}_{[!]}$ sends projective objects in $\Perv_{(B)}(X,\E)$ to projective objects in $\Perv_{(B)}(X^s,\E)$.


\subsection{Morphisms between standard perverse sheaves}
\label{subsect:standardmorphisms}

For any $w \in W$, following the notation of \S \ref{ss:standard} we set $\Delta_{w,\circ}:=i_{w!} \underline{\E}_{X_{w,\circ}}[\ell(w)] \in \Perv_{(B_\circ )}(X_\circ ,\E)$, and $\Delta_w:=\For(\Delta_{w,\circ})$. Recall (see \S \ref{ss:perverse-sheaves-Fq} and Remark \ref{rmk:cohomology-Frobenius}) that for $u,v \in W$ and $i \in \Z$ the $\E$-modules $\Hom_{\E}^i(\Delta_u,\Delta_v)$ and $\coHc{i}( X_v \cap X_u^-, \E)$ have a natural automorphism induced by the Frobenius morphism.

\begin{prop}
\label{prop:morphisms-standard-R-varieties}

For any $u,v \in W$ there exists an isomorphism of graded $(\E,\phi)$-modules
\[
\Ext^{\bullet}_{\E}(\Delta_u,\Delta_v) \ \cong \ \coHc{\bullet+\ell(v)-\ell(u)} \bigl( X_v \cap X_u^-, \E\bigr).
\]

\end{prop}

\begin{proof}
If $u \nleq v$, then both modules are $0$, hence there is nothing to prove. Assume from now on that $u \leq v$.

Fix $n \in \Z$. First, adjunction provides an isomorphism of $(\E,\phi)$-modules
\[
\Ext^n_{\E}(\Delta_u,\Delta_v) \ \cong \ \coH{n+\ell(v)-\ell(u)} (X_u, i_u^! i_{v!} \underline{\E}_{X_v}).
\]
Consider the following diagram:
\[
\xymatrix{
X_v \ar[r]^-{j_v} \ar@/^1.5pc/[rr]^-{i_v} & \overline{X_v} \ar[r]^-{\overline{i}_v} & X \\
& X_u \ar[u]^-{i_{u,v}} \ar[ur]_-{i_u} &
}
\]
We obtain isomorphisms
\[
i_u^! i_{v!} \underline{\E}_{X_v} \ \cong \ i_{u,v}^! \overline{i}_v^! \overline{i}_{v!} j_{v!} \underline{\E}_{X_v} \ \cong \ i_{u,v}^! j_{v!} \underline{\E}_{X_v}.
\]

Now, consider the following diagram with cartesian square:
\[
\xymatrix{
X_v \ar@{}[dr]|{\square} \ar@{^{(}->}[r]^-{j_v} & \overline{X_v} \\
(X_v \cap X_u^-) \times X_u \ar@{^{(}->}[u]^-{j'} \ar@{^{(}->}[r]^-{j_v'} & (\overline{X_v} \cap X_u^-) \times X_u \ar@{^{(}->}[u]^-{j} \ar@<1ex>[d]^-{p} \\
& X_u \ar@{->}[u]^-{k} \ar@/_4pc/[uu]_-{i_{u,v}}
}
\]
where $j$ and $j'$ are the open inclusions considered in \cite[\S 1.4]{KL2}, and $k$ (resp.~$p$) is the obvious inclusion (resp.~projection). Then we have
\[
i_{u,v}^! j_{v!} \underline{\E}_{X_v} \ \cong \ k^! j^! j_{v!} \underline{\E}_{X_v} \ \cong \ k^! j^* j_{v!} \underline{\E}_{X_v} \ \cong \ k^! j_{v!}' \underline{\E}_{(X_v \cap X_u^-) \times X_u},
\]
where the second isomorphism follows from the fact that $j$ is an open embedding, and the third one from base change.

The adjunction $(k_!,k^!)$ provides a morphism of functor $k_! k^! \to \id$. Composing with $p_!$ gives a morphism of functors $k^! \to p_!$. By \cite[Proposition 1]{So}, this morphism induces an isomorphism 
\[
k^! j_{v!}' \underline{\E}_{(X_v \cap X_u^-) \times X_u} \ \xrightarrow{\sim} \ p_! j_{v!}' \underline{\E}_{(X_v \cap X_u^-) \times X_u}.
\]
Combining these isomorphisms we obtain
\[
\Ext^n_{\E}(\Delta_u,\Delta_v) \ \cong \ \coH{n+\ell(v)-\ell(u)} (X_u, (p \circ j_v')_! \underline{\E}_{(X_v \cap X_u^-) \times X_u}).
\]

Now, consider the following (tautological) cartesian diagram:
\[
\xymatrix{
(X_v \cap X_u^-) \times X_u \ar[r]^-{q} \ar[d]_-{p \circ j_v'} \ar@{}[dr]|{\square} & X_v \cap X_u^- \ar[d]^-{b} \\
X_u \ar[r]^-{a} & \pt,
}
\]
where $q,a,b$ are the natural projections. Then by the base change theorem we have
\[
(p \circ j_v')_! \underline{\E}_{(X_v \cap X_u^-) \times X_u} \ \cong \ (p \circ j_v')_! q^* \underline{\E}_{X_v \cap X_u^-} \ \cong \ a^* b_! \underline{\E}_{X_v \cap X_u^-}.
\]
We deduce isomorphisms
\begin{multline*}
\Ext^n_{\E}(\Delta_u,\Delta_v) \ \cong \ \coH{n+\ell(v)-\ell(u)} (X_u, a^* b_! \underline{\E}_{X_v \cap X_u^-}) \\ \cong \ \coH{n+\ell(v)-\ell(u)} (\pt, a_* a^* b_! \underline{\E}_{X_v \cap X_u^-}).
\end{multline*}
Finally, since $X_u$ is an affine space, we observe that the adjunction $(a^*, a_*)$ induces an isomorphism
\[
\coH{n+\ell(v)-\ell(u)} (\pt, b_! \underline{\E}_{X_v \cap X_u^-}) \ \xrightarrow{\sim} \ \coH{n+\ell(v)-\ell(u)} (\pt, a_* a^* b_! \underline{\E}_{X_v \cap X_u^-}),
\]
which finishes the proof.
\end{proof}


\subsection{Weights of cohomology of Deodhar varieties}
\label{ss:weightsR}

In the next two lemmas we let $\Bbbk$ be an arbitrary field, and we use the notation $G$, $B$, $X$, etc.~for a split reductive algebraic group, a Borel subgroup, the flag variety, etc.~over $\Bbbk$. We will use these results only when $\Bbbk=\Fq$.

The following result is well known, and probably due to Deodhar (see \cite{De}; see \cite{CUD} for a more general result). For this reason we call the varieties $X_v \cap X_u^-$ ``Deodhar varieties''.

\begin{lem}
\label{lem:R-varieties}

Let $u,v \in W$ and $s \in S$ be such that $u \leq v$, $vs>v$ and $us > u$. 

\begin{enumerate}
\item There exists an isomorphism of $\Bbbk$-varieties
\[
X_{vs} \cap X^-_{us} \ \cong \ X_{v} \cap X^-_{u}.
\]
\item There exists a closed subvariety $Z \subset X_{vs} \cap X^-_{u}$, with complement $U$, and isomorphisms of $\Bbbk$-varieties
\[
Z \ \cong \ (X_{v} \cap X^-_{us}) \times \bA^1_{\Bbbk}, \quad U \ \cong \ (X_{v} \cap X^-_{u}) \times (\bA^1_\Bbbk \smallsetminus \{0\}).
\]
\end{enumerate} 

\end{lem}

\begin{proof}
Set $X^{s,-}_{u}:=B^- u P_{s}/P_{s} \subset X^s$. Then we have 
\[
(\pi_s)^{-1}(X^{s,-}_{u}) = X^-_{u} \sqcup X^-_{us}, \qquad (\pi_s)^{-1}(X^{s}_{v}) = X_{vs} \sqcup X_{v}.
\]
In each of these decompositions, the first term is open and the second
term is closed. On the other hand, as the restriction of $\pi_s$ to
$X_{v}$ (resp.~$X^-_{us}$) is an isomorphism with
$X_{v}^s$ (resp.~ $X^{s,-}_{u}$) we have isomorphisms
\begin{multline*}
X_{v} \cap X^-_{u} \cong \pi_s(X_{v} \cap X^-_{u}), \quad X_{v} \cap X^-_{us} \cong \pi_s(X_{v} \cap X^-_{us}), \\ \quad X_{vs} \cap X^-_{us} \cong \pi_s(X_{vs} \cap X^-_{us}),
\end{multline*}
and
\begin{align*}
X^{s}_{v} \cap X^{s,-}_{u}\ & = \ \pi_s(X_{v} \cap X^-_{u}) \ \sqcup \ \pi_s(X_{v} \cap X^-_{us}) \\
& = \ \pi_s(X_{vs} \cap X^-_{us}) \ \sqcup \ \pi_s(X_{v} \cap X^-_{us}).
\end{align*}
(In each of these decompositions the first term is open and the second one is closed.) It follows that
\[
\pi_s(X_{v} \cap X^-_{u}) \ = \ \pi_s(X_{vs} \cap X^-_{us}),
\]
which proves (1).

Now, set 
\begin{align*}
Z \ & := \ (X_{vs} \cap X^-_{u}) \cap (\pi_s)^{-1} \bigl( \pi_s(X_{v} \cap X^-_{us}) \bigr), \\
U \ & := \ (X_{vs} \cap X^-_{u}) \cap (\pi_s)^{-1} \bigl( \pi_s(X_{v} \cap X^-_{u}) \bigr).
\end{align*}
To prove (2) it is enough to prove that $\pi_s|_{Z} : Z \to \pi_s(X_{v} \cap X^-_{us})$ is a trivial $\bA^1_\Bbbk$-fibration, and that $\pi_s|_{U} : U \to \pi_s(X_{v} \cap X^-_{u})$ is a trivial $(\bA^1_\Bbbk \smallsetminus \{0\})$-fibration. Consider the $\bA^1_\Bbbk$-fibration
\[
f: (X_{vs} \cap X^-_{u}) \sqcup (X_{vs} \cap X^-_{us}) \to X^{s}_{v} \cap X^{s,-}_{u}
\]
induced by $\pi_s$, which is trivial since it is the restriction of the trivial $\bA^1_\Bbbk$-fibration $X_{vs} \to X^{s}_{v}$ to the inverse image of $X^{s}_{v} \cap X^{s,-}_{u}$. We claim that 
\begin{equation}
\label{eqn:claim-Z_0}
f^{-1} \bigl( \pi_s(X_{v} \cap X^-_{us}) \bigr) = Z,
\end{equation}
which will prove our claim about $Z$. Indeed, let $x \in \pi_s(X_{v} \cap X^-_{us})$. Then $\pi_s^{-1}(x) \cap X_{v} = \pt$, and $\pi_s^{-1}(x) \cap X^-_{us}=\pt$, and by assumption these points coincide. It follows that $\pi_s^{-1}(x) \cap X_{vs} = \pi_s^{-1}(x) \cap X^-_{u}$, hence that $\pi_s^{-1}(x) \cap (X_{vs} \cap X^-_{us}) = \emptyset$. This proves \eqref{eqn:claim-Z_0}.

On the other hand, consider $f^{-1} \bigl( \pi_s(X_{v} \cap X^-_{u}) \bigr)$, and denote by $g$ the restriction of $f$ to this open subvariety. Then the restriction of $g$ to 
\[
(X_{vs} \cap X^-_{us}) \subset f^{-1} \bigl( \pi_s(X_{v} \cap X^-_{u}) \bigr)
\]
is an isomorphism. Hence the claim about $U$ follows from Lemma \ref{lem:fibration} below.
\end{proof}

\begin{rmk}
In case (2), the closed subvariety $Z$ is empty if $us \nleq v$.
\end{rmk}

\begin{lem}
\label{lem:fibration}

Let $M$ be a $\Bbbk$-variety, and consider the trivial fibration
\[
a : M \times \bA^1_\Bbbk \to M.
\]
Assume there exists a subvariety $N \subset M \times \bA^1_\Bbbk$ such that the restriction of $a$ to $N$ is an isomorphism $N \xrightarrow{\sim} M$. Then the restriction of $a$ to the complement of $N$ is a trivial $(\bA^1_\Bbbk \smallsetminus \{0\})$-fibration.

\end{lem}

\begin{proof}
Let $\xi$ be the composition $M \xrightarrow{\sim} N \hookrightarrow M \times \bA^1_\Bbbk \to \bA^1_\Bbbk$. Then the automorphism
\[
\begin{array}{ccc}
M \times \bA^1_\Bbbk & \to & M \times \bA^1_\Bbbk \\
(m,t) & \mapsto & (m,t-\xi(m))
\end{array}
\]
identifies the complement of $N$ with $M \times (\bA^1_\Bbbk \smallsetminus \{0\})$ as a fibration over $M$.
\end{proof}

Now we come back to our varieties over $\Fq$ and $\Fqb$.

\begin{prop}
\label{prop:weights-R-varieties}

Let $u,v \in W$, with $u \leq v$. The cohomology 
\[
\coHc{\bullet}(X_v \cap X_u^-, \O)
\]
is concentrated in degrees between $\ell(v)-\ell(u)$ and $2(\ell(v)-\ell(u))$. Moreover, for any $n \in \llbracket \ell(v)-\ell(u), 2(\ell(v)-\ell(u)) \rrbracket$, the $(\O,\phi)$-module $\coHc{n}(X_v \cap X_u^-, \O)$ has $q$-weights obtained from $\llbracket - \lfloor \frac{n}{2} \rfloor, -n + \ell(v) - \ell(u) \rrbracket$.

\end{prop}

\begin{proof}
We prove the claim by induction on the Bruhat order on $v$. It is obvious if $v=1$.

Now, assume the result is known for $v$, and let $s \in S$ be such that $vs>v$. Let $u \in W$ such that $u \leq vs$. If $us < u$ (and then $us \leq v$), by Lemma \ref{lem:R-varieties}(1) we have an isomorphism
\[
X_{vs,\circ} \cap X_{u,\circ}^- \ \cong \ X_{v,\circ} \cap X_{us,\circ}^-.
\]
As $\ell(vs)-\ell(u) = \ell(v) - \ell(us)$, the result follows by induction.

Now, assume that $us > u$ (and then $u \leq v$). Then, by Lemma \ref{lem:R-varieties}(2), $X_{vs,\circ} \cap X_{u,\circ}^-$  is the disjoint union of a closed subvariety $Z_\circ$ isomorphic to $(X_{v,\circ} \cap X_{us,\circ}^-) \times \bA^1_\circ$ and an open subvariety $U_\circ$ isomorphic to $(X_{v,\circ} \cap X_{u,\circ}^-) \times (\bA^1_\circ \smallsetminus \{0\})$. Consider the associated long exact sequence
\[
\cdots \to \coHc{n}(U,\O) \to \coHc{n}(X_{vs} \cap X_{u}^-,\O) \to \coHc{n}(Z,\O) \to \cdots
\]
Both $\coHc{\bullet} (\bA^1,\O)$ and $\coHc{\bullet} (\bA^1
\smallsetminus \{0\},\O)$ are free $\O$-modules and so we can apply the
K{\"u}nneth formula to obtain
\[
\coHc{\bullet} (Z,\O) \cong \coHc{\bullet} (X_v \cap X_{us}^-) \otimes_{\O} \coHc{\bullet} (\bA^1,\O) \cong \coHc{\bullet-2} (X_v \cap X_{us}^-) \otimes_{\O} \coHc{2}(\bA^1,\O)
\]
and
\begin{multline*}
\coHc{\bullet}(U,\O) \cong \coHc{\bullet} (X_{v} \cap X_{u}^-) \otimes_{\O} \coHc{\bullet} (\bA^1 \smallsetminus \{0\},\O) \\
\cong \coHc{\bullet-1} (X_{v} \cap X_{u}^-) \otimes_{\O} \coHc{1}(\bA^1 \smallsetminus \{0\}, \O) \oplus \coHc{\bullet-2} (X_{v} \cap X_{u}^-) \otimes_{\O} \coHc{2}(\bA^1 \smallsetminus \{0\}, \O).
\end{multline*}
Our claim follows, using Remark \ref{rmk:cohomology-Frobenius}.
\end{proof}

\begin{rmk}
The bounds on weights in Proposition \ref{prop:weights-R-varieties} are exactly the same as the bounds provided in a very general setting for cohomology with coefficients in $\K$ by \cite[Corollaire 3.3.3]{Del}. In particular, if one is only interested in $\coHc{\bullet}(X_v \cap X_u^-,\K)$, then the information on the structure of $X_v \cap X_u^-$ given by Lemma \ref{lem:R-varieties} does not give improved bounds. However, here we work with integral coefficients, which are not considered in \cite{Del}.
\end{rmk}

Combining Proposition \ref{prop:morphisms-standard-R-varieties} and Proposition \ref{prop:weights-R-varieties} we obtain the following result, which will play a crucial role in the rest of this section.

\begin{cor}
\label{cor:weights-Ext-standard}

Let $u,v \in W$ such that $u \leq v$. Then
\[
\Ext^n_{\O}(\Delta_u,\Delta_v)
\]
vanishes unless $0 \leq n \leq \ell(v)-\ell(u)$. If $n \in \llbracket 0,\ell(v)-\ell(u) \rrbracket$, this $(\O,\phi)$-module has $q$-weights obtained from $\llbracket - \lfloor \frac{n+\ell(v)-\ell(u)}{2} \rfloor, -n \rrbracket$.
\end{cor}

\subsection{Weights in the $\Delta$-flag of projective covers}
\label{ss:geom-projective-generator}

\begin{prop}
\label{prop:geom-projective-weights}

For any $u \in W$, there exists an object $\cP_{u,\circ}$ in $\Perv_{(B_\circ )}(X_\circ ,\O)$ such that $\For(\cP_{u,\circ})$ is projective in $\Perv_{(B)}(X,\O)$, together with an exact sequence
\[
\cM_{u,\circ} \into \cP_{u,\circ} \onto \Delta_{u,\circ}
\]
such that $\cM_{u,\circ}$ has a filtration whose subquotients are of the form
\[
E_{v,u} \otimes_{\O} \Delta_{v,\circ}
\]
for $v > u$, where $E_{v,u}$ is an $(\O,\phi)$-module which is $\O$-free and has $q$-weights obtained from $\llbracket 1, \ell(v) - \ell(u) \rrbracket$.

\end{prop}

\begin{proof}
We prove by induction that for any $I \subset W$ containing $u$ with $X_I$ closed, there exists an object $\cP_{u,\circ}^I$ in $\Perv_{(B_\circ )}(X_{I,\circ})$ such that $\cP_u^I:=\For(\cP_{u,\circ}^I)$ is projective in $\Perv_{(B)}(X_I)$ together with an exact sequence
\[
\cM_{u,\circ}^I \into \cP_{u,\circ}^I \onto \Delta_{u,\circ}^I
\]
such that $\cM_{u,\circ}^I$ has a filtration whose subquotients are of the form
\[
E_{v,u}^I \otimes_{\O} \Delta_{v,\circ}^I
\]
for $v > u$ ($v \in I$), where $E_{v,u}^I$ is an $(\O,\phi)$-module which is $\O$-free and has $q$-weights obtained from $\llbracket 1, \ell(v) - \ell(u) \rrbracket$. (Here, $\Delta_{w,\circ}^I$ is the standard perverse sheaf on $X_{I,\circ}$ associated with $w$.) The case $I=W$ will give the proposition. If $u$ is maximal in $I$, then $\cP_{u,\circ}^I=\Delta_{u,\circ}^I$ satisfies these requirements.

Now we take $I=I'\cup \{v\}$ with $X_{I'}$ closed and containing $X_u$. Consider the $(\O,\phi)$-module
\[
E \ := \ \Ext^1_{\O}(\cP_{u}^{I'},\Delta_v^I).
\]
(Here we consider $\cP_u^{I'}$ as a perverse sheaf on $X_I$.)
We claim that this module has $q$-weights obtained from $\llbracket \ell(u) - \ell(v),-1 \rrbracket$. Indeed, by induction hypothesis $\cP_{u,\circ}^{I'}$ has a filtration with subquotients $E_{w,u}^{I'} \otimes_{\O} \Delta_{w,\circ}^I$, where $E_{w,u}^{I'}$ has $q$-weights obtained from $\llbracket 0, \ell(w) - \ell(u) \rrbracket$, and moreover by Corollary \ref{cor:weights-Ext-standard} $\Ext^1_{\O,X_I}(\Delta_w^I,\Delta_v^I) = \Ext^1_{\O,X}(\Delta_w,\Delta_v)$ has $q$-weights obtained from 
\[
\llbracket - \lfloor \frac{1+\ell(v)-\ell(w)}{2} \rfloor, -1 \rrbracket \ \subset \ \llbracket \ell(w) - \ell(v), -1 \rrbracket.
\]

Using Lemma \ref{lem:free-resolution-weights}, we deduce that there exists an $(\O,\phi)$-module $E_{\free}$ which is $\O$-free, with $q$-weights obtained from $\llbracket \ell(u) - \ell(v),-1 \rrbracket$ and a surjection of $(\O,\phi)$-modules $E_{\free} \twoheadrightarrow E$. By \cite[(5.1.2.5)]{BBD}, there exists a natural surjection
\[
\Ext^1_{\Perv_{\scS}(X_\circ,\O)}(\cP_{u,\circ}^{I'},E_{\free}^* \otimes_{\O} \Delta_{v,\circ}^I) \twoheadrightarrow \bigl( \Ext^1_{\Perv_{\scS}(X,\O)}(\cP_{u}^{I'},E_{\free}^* \otimes_{\O} \Delta_{v}^I) \bigr)^{\phi\mathrm{-inv}}.
\]
The element of $\Ext^1_{\Perv_{\scS}(X,\O)}(\cP_{u}^{I'},E_{\free}^* \otimes_{\O} \Delta_{v}^I)$ considered in the proof of Proposition \ref{prop:existence-projectives} is $\phi$-invariant by construction, hence defines an extension of $E_{\free}^* \otimes_{\O} \Delta_{v,\circ}^I$ by $\cP_{u,\circ}^{I'}$, which we denote by $\cP_{u,\circ}^I$. By the proof of Proposition \ref{prop:existence-projectives}, $\For(\cP_{u,\circ}^I)$ is projective, which proves the induction.
\end{proof}

For any $u \in W$ we fix an object $\cP_{u,\circ}$ as in Proposition \ref{prop:geom-projective-weights}, and we set
\[
\cP_\circ \ := \ \bigoplus_{u \in W} \, \cP_{u,\circ}.
\]
Note that this object satisfies the condition of Lemma \ref{lem:proj-generator-X_0} (see the proof of Corollary \ref{cor:properties-projectives}).

\begin{prop}
\label{prop:weights-morphisms-projectives}

The $(\O,\phi)$-module
\[
\End_{\O}(\cP)
\]
is $\O$-free, and has $q$-weights obtained from $\llbracket -\ell(w_0) , \ell(w_0) \rrbracket$. In particular, this $(\O,\phi)$-module is $q$-decomposable if the order of $q$ in $\F$ is strictly bigger than $2 \ell(w_0)$.

\end{prop}

\begin{proof}
The fact that this $\O$-module is free follows from Proposition \ref{prop:properties-projectives-2}(4).

To prove the statement about weights, it is sufficient to prove that for any $u,v \in W$ the $q$-weights of $\Hom_{\O}(\cP_u,\cP_v)$ are in $\llbracket -\ell(w_0) , \ell(w_0) \rrbracket$. Now recall that $\cP_{v,\circ}$ has a filtration with subquotients of the form
\[
E_{w,v} \otimes_{\O} \Delta_{w,\circ}
\]
for $w \geq v$, where $E_{w,v}$ is an $(\O,\phi)$-module which has $q$-weights obtained from $\llbracket 0, \ell(w) - \ell(v) \rrbracket$. Hence $\Hom_{\O}(\cP_u,\cP_v)$ has a filtration (as an $(\O,\phi)$-module) with subquotients of the form
\[
E_{w,v} \otimes_{\O} \Hom_{\O}(\cP_u,\Delta_w).
\]
Similarly, $\cP_{u,\circ}$ has a filtration with subquotients of the form
\[
E_{w',u} \otimes_{\O} \Delta_{w',\circ}
\]
for $w' \geq u$, where $E_{w',u}$ is an $(\O,\phi)$-module which has $q$-weights obtained from $\llbracket 0, \ell(w') - \ell(u) \rrbracket$. Hence it is enough to study weights in
\[
F_{u,v,w,w'} \ := \ E_{w',u}^* \otimes_{\O} \Hom_{\O}(\Delta_{w'},\Delta_w) \otimes_{\O}  E_{w,v}.
\]
By Corollary \ref{cor:weights-Ext-standard}, $\Hom_{\O}(\Delta_{w'},\Delta_w)$ has $q$-weights obtained from $\llbracket \frac{\ell(w')-\ell(w)}{2},0 \rrbracket$. Hence $F_{u,v,w,w'}$ has $q$-weights obtained from
\[
\llbracket \ell(u) - \frac{\ell(w')+\ell(w)}{2}, \ell(w)-\ell(v) \rrbracket \ \subset \ \llbracket -\ell(w_0) , \ell(w_0) \rrbracket,
\]
which finishes the proof of the claim on weights.

Finally, the last claim of the statement follows from Remark \ref{prop:criterion-decomposability}.
\end{proof}

\subsection{Partial flag varieties}
\label{ss:geom-projective-generator-partial}

In this subsection we fix a simple reflection $s$, and consider the variety $X^s$. For any $v \in W^s$, we denote by $\Delta_v^s$ the standard object associated to the stratum $X^s_v$.

\begin{lem}
\label{lem:weights-Ext-standard-partial}

For $u,v \in W^s$, the $(\O,\phi)$-module 
\[
\Ext^n_{\O}(\Delta_u^s,\Delta_v^s)
\]
vanishes unless $0 \leq n \leq \ell(v)-\ell(u)$, in which case it has $q$-weights obtained from $\llbracket - \lfloor \frac{n+\ell(v)-\ell(u)}{2} \rfloor, -n \rrbracket$.

\end{lem}

\begin{proof}
The restriction of $\pi_s$ to $X_v$ is an isomorphism onto $X_v^s$, hence we have $\pi_{s*} \Delta_v=\Delta_v^s$. Using adjunction, we deduce isomorphisms
\[
\Ext^n_{\O,X^s}(\Delta_u^s,\Delta_v^s) \ \cong \ \Ext^n_{\O,X^s}(\Delta_u^s, \pi_{s*} \Delta_v) \ \cong \ \Ext^n_{\O,X}( \pi_s^* \Delta_u^s, \Delta_v).
\]
By the base change theorem we have $\pi_s^* \Delta_u^s = i_{u!}^s \underline{\O}_{X_u \sqcup X_{us}}[\ell(u)]$, where $i_u^s : X_u \sqcup X_{us} \into X$ is the inclusion. We have an exact sequence of sheaves on $X_u \sqcup X_{us}$:
\[
j_!\underline{\O}_{X_{us}} \into \underline{\O}_{X_u \sqcup X_{us}} \onto i_*\underline{\O}_{X_u},
\]
where $i$ (resp. $j$) denotes the inclusion of $X_{us}$ (resp. $X_u$)
which induces a distinguished triangle
\[
\Delta_{us}[-1] \to i_{u!}^s \underline{\O}_{X_u \sqcup X_{us}}[\ell(u)] \to \Delta_u \triright.
\]
Hence we obtain an exact sequence for any $n \geq 0$:
\[
\Ext^n_{\O,X}(\Delta_u,\Delta_v) \to \Ext^n_{\O,X^s}(\Delta_u^s,\Delta_v^s) \to \Ext^{n+1}_{\O,X}(\Delta_{us},\Delta_v).
\]
Then the result follows from Corollary \ref{cor:weights-Ext-standard}.
\end{proof}

Using this lemma, the same proof as that of Proposition \ref{prop:geom-projective-weights} gives the following result.

\begin{prop}
\label{prop:geom-projective-weights-partial}

For any $u \in W^s$, there exists an object $\cP_{u,\circ}^s$ in $\Perv_{(B_\circ )}(X^s_\circ ,\O)$ such that $\For(\cP_{u,\circ}^s)$ is projective in $\Perv_{(B)}(X^s,\O)$, together with an exact sequence
\[
\cM_{u,\circ}^s \into \cP_{u,\circ}^s \onto \Delta_{u,\circ}^s
\]
such that $\cM_{u,\circ}^s$ has a filtration whose subquotients are of the form
\[
E_{v,u}^s \otimes_{\O} \Delta_{v,\circ}^s
\]
for $v > u$ ($v \in W^s$), where $E_{v,u}^s$ is an $(\O,\phi)$-module which is $\O$-free and has $q$-weights obtained from $\llbracket 1, \ell(v) - \ell(u) \rrbracket$.

\end{prop}

As in \S \ref{ss:geom-projective-generator} we choose objects $\cP_{u,\circ}^s$ as in Proposition \ref{prop:geom-projective-weights-partial}, and we define the object
\[
\cP_\circ^s \ := \ \bigoplus_{u \in W^s} \, \cP_{u,\circ}^s.
\]
Again we observe that this object satisfies the conditions of Lemma \ref{lem:proj-generator-X_0} (see the proof of Corollary \ref{cor:properties-projectives}). We let $\cP_\circ$ denote the same object as in \S \ref{ss:geom-projective-generator}.

\begin{lem}
\label{lem:bimodule-decomposable}

The $(\O,\phi)$-module
\[
\Hom_{\O}(\pi_s {}_{[!]} \cP,\cP^s)
\]
is $\O$-free, and has $q$-weights obtained from $\llbracket -\ell(w_0)+1, \ell(w_0) \rrbracket$. In particular, it is $q$-decomposable if the order of $q$ in $\F$ is bigger than $2 \ell(w_0)$.

\end{lem}

\begin{proof}
First, the object $\pi_s {}_{[!]} \cP \in \Perv_{(B)}(X^s,\O)$ is projective (see \S \ref{ss:notation}), hence the fact that $\Hom(\pi_s {}_{[!]} \cP,\cP^s)$ is $\O$-free follows from Proposition \ref{prop:properties-projectives-2}(4).

Next, by adjunction we have
\[
\Hom_{\O}(\pi_s {}_{[!]} \cP,\cP^s) \ \cong \ \Hom_{\O}(\cP, \pi_s^{[!]} \cP^s) \ \cong \ \Hom_{\O}(\cP, \pi_s^! \cP^s[-1]).
\]
Now there exists an isomorphism of functors $\pi_s^! \cong
\pi_s^*[2](1)$ since $\pi_s$ is smooth of relative dimension 1, hence we obtain an isomorphism of $(\O,\phi)$-modules
\[
\Hom_{\O}(\pi_s {}_{[!]} \cP,\cP^s) \ \cong \ \Hom_{\O}(\cP, \pi_s^*  \cP^s[1])(1).
\]
By construction, $\cP^s_\circ$ has a filtration with subquotients of the form
\[
E_{v,u}^s \otimes_{\O} \Delta_{v,\circ}^s
\]
for $v \geq u$ ($v,u \in W^s$), where $E_{v,u}^s$ is an $(\O,\phi)$-module which has $q$-weights obtained from $\llbracket 0, \ell(v) - \ell(u) \rrbracket$. Hence $\pi_s^* \cP^s_\circ [1]$ has a filtration with subquotients
\[
E_{v,u}^s \otimes_{\O} \pi_s^* \Delta_{v,\circ}^s[1],
\]
and then $\Hom(\pi_s {}_{[!]} \cP,\cP^s)$ has a filtration with subquotients
\[
A_{v,u} \ := \ E_{v,u}^s \otimes_{\O} \Hom_{\O}(\cP,\pi_s^* \Delta_{v}^s[1]) (1).
\]

As in the proof of Lemma \ref{lem:weights-Ext-standard-partial}, we have a surjection $\Delta_{vs} \onto \pi_s^* \Delta_v^s[1]$, which induces a surjection
\[
\Hom_{\O}(\cP,\Delta_{vs})(1) \onto \Hom_{\O}(\cP,\pi_s^* \Delta_{v}^s[1]) (1).
\]
We have seen in the proof of Proposition \ref{prop:weights-morphisms-projectives} that $\Hom(\cP,\Delta_{vs})$ has $q$-weights obtained from $\llbracket - \ell(w_0), 0 \rrbracket$. It follows that $A_{v,u}$ has $q$-weights obtained from $\llbracket -\ell(w_0)+1, \ell(w_0) \rrbracket$, finishing the proof of the claim concerning weights.

The last claim of the statement again follows from Remark \ref{prop:criterion-decomposability}.
\end{proof}

\begin{prop}
\label{prop:End-decomposable-partial}

The $(\O,\phi)$-module
\[
\End_{\O}(\pi_s {}_{[!]} \cP)
\]
is $\O$-free. If the order of $q$ in $\F$ is bigger than $2 \ell(w_0)$, then it is $q$-decomposable.

\end{prop}

\begin{proof}
As $\pi_s {}_{[!]} \cP$ is projective (see \S \ref{ss:notation}), by Proposition \ref{prop:properties-projectives-2}(4) its endomorphism algebra is $\O$-free.

By Lemma \ref{lem:noetherian-categories} below, the morphism
\[
\End_{\O}(\pi_s {}_{[!]} \cP) \to \End_{\End(\cP^s)} \bigl( \Hom_{\O}(\pi_s {}_{[!]} \cP,\cP^s) \bigr)
\]
induced by the functor $\Hom_{\O}(-,\cP^s)$ is an isomorphism of $(\O,\phi)$-modules. By Lemma \ref{lem:bimodule-decomposable}, $\Hom_{\O}(\pi_s {}_{[!]} \cP,\cP^s)$ is a $q$-decomposable, $\O$-free $(\O,\phi)$-module, hence the same holds for the $(\O,\phi)$-module
\[
V_s:=\End_{\O}(\Hom_{\O}(\pi_s {}_{[!]} \cP,\cP^s)).
\]
As $\End_{\End(P^s)}(\Hom_{\O}(\pi_s {}_{[!]} \cP,\cP^s))$ is a sub-$(\O,\phi)$-module of $V_s$ with torsion-free quotient, by Lemma \ref{lem:quotient-decomposable}, it is also $q$-decomposable, which finishes the proof.
\end{proof}

\begin{lem}
\label{lem:noetherian-categories}

Let $R$ be a right noetherian ring, and $Q$ be a projective generator of $\Modfr R$. Then the functor 
\[
\Hom_{-R}(-,Q) : \Modfr R \to \End_{-R}(Q) \Mod
\]
is fully faithful on projectives.

\end{lem}

\begin{proof}
To simplify notation, let $R':=\End_{-R}(Q)$. Let also $\mathsf{Projf}-R$, resp.~$R-\mathsf{Projf}$ denote the category of finitely generated projective right, resp.~left, $R$-modules, and similarly for $R'$.

By \cite[Exercise on p.~55]{Ba}, $R'$ is a right noetherian ring and the functor $\Hom_{-R}(Q,-)$ induces an equivalence of categories between $\Modfr R$ and $\Modfr R'$. In particular, for any $P$ in $\mathsf{Projf}-R$ the morphism
\begin{equation}
\label{eqn:functor-Hom-Q}
\Hom_{-R}(P,Q) \ \to \ \Hom_{-R'} \bigl( \Hom_{-R}(Q,P),R' \bigr)
\end{equation}
induced by $\Hom_{-R}(Q,-)$ is an isomorphism, and moreover $\Hom_{-R}(Q,P)$ is in $\mathsf{Projf}-R'$. Using isomorphism \eqref{eqn:functor-Hom-Q}, the functor $\Hom_{-R}(-,Q) : \mathsf{Projf}-R \to R' \Mod$ is isomorphic to the composition
\[
\mathsf{Projf}-R \xrightarrow{\Hom_{-R}(Q,-)} \mathsf{Projf}-R' \xrightarrow{\Hom_{-R'}(-,R')} R' \Mod.
\]
As explained above the first functor is fully faithful, and the second one is easily seen to be fully faithful also.
\end{proof}

%
%
%

\section{Formality}
\label{sec:formality}

As above, in \S\S \ref{subsec:BS}-\ref{ss:direct-inverse-image} we let $\E$ be either $\O$ or $\F$.

\subsection{Bott--Samelson sheaves and parity sheaves}
\label{subsec:BS}

Recall \cite[Definition 2.4]{JMW} that a complex $\cF \in
\cDb_{(B)}(X, \E)$ is called \emph{parity} if it
admits a decomposition $\cF \cong \cF_0 \oplus \cF_1$ such that for all $w \in W$, $? \in \{ !, * \}$ and $j\in \Z$ we have
\begin{enumerate}
\item $\cH^j(i_w^? \cF_k) = 0$ if $j \not \equiv k \mod 2$;
\item $\cH^j(i_w^? \cF_k)$ is an $\E$-free constant local system if $j \equiv k \mod 2$.
\end{enumerate}
An indecomposable parity complex is called a \emph{parity sheaf}. Note
that, as the category $\cDb_{(B)}(X,\E)$ satisfies the
Krull--Schmidt property, every parity complex decomposes
(uniquely) as a direct sum of parity sheaves (see \cite[\S 2.1]{JMW}  for details).

One has the following theorem (\cite[Theorem 2.12]{JMW}):

\begin{thm}
\label{thm:parity classification}

For all $w \in W$ there is (up to isomorphism) at most one
  parity sheaf $\cE_w$ supported on $\overline{X_w}$ and extending the constant local system
  $\underline{\E}_{X_w}[\ell(w)]$ on $X_w$. Moreover, any parity sheaf is isomorphic
  to $\cE_w[m]$ for some $w \in W$ and $m \in \Z$.

\end{thm}

If $f=(s,t,\cdots,r)$ is a sequence of simple reflections, we consider the corresponding Bott--Samelson complex
\[
\cB \cS_{f,\circ}^{\E} := \pi_s^! \pi_{s!} \pi_t^! \pi_{t!} \cdots \pi_r^!
\pi_{r!} \IC_{e,\circ}^{\E} \quad \in \cDb_{(B_\circ)}(X_\circ , \E),
\]
where $e \in W$ is the unit element and $\IC_{e,\circ}^{\E}$ is the
corresponding $\IC$-sheaf. As usual, we denote by $\cB\cS_f^{\E}$ the
object obtained by base change to $\Fqb$.
The following is an immediate consequence of
Proposition \ref{prop:stalks} below in the case $\E=\O$; the case $\E=\F$ follows (or can alternatively be proved along the same lines).

\begin{prop} \label{prop:BS parity}
  For any sequence $f$ of simple reflections, $\cB\cS_f^{\E}$ is parity. If
  $f$ is a reduced expression for $w^{-1}$ then $\cB\cS_f^{\E}$ is supported on
  $\overline{X_w}$ and $i_w^* \cB\cS_f^{\E} \cong \underline{\E}_{X_w}[2\ell(w)]$.
\end{prop}

\begin{rmk} Theorem \ref{thm:parity classification} says nothing about
  the existence of parity sheaves, and for an arbitrary stratified
  variety existence may be difficult to establish. The above
  proposition combined with Theorem \ref{thm:parity classification}
  gives the existence and uniqueness of parity sheaves on the flag
  variety (see \cite[Theorem 1.2]{So2} and \cite[Theorem 4.6]{JMW}).
\end{rmk}

Given a sequence $f=(s,t,\cdots,r)$ as above we set $\overline{f}:=st \cdots r \in
W$. Let us fix a family $F=\{f_1, \cdots, f_n\}$ of sequences of
simple reflections such that $W=\{\overline{f_1}, \cdots,
\overline{f_n}\}$. Then we set
\[
\cB \cS_{\circ}^{\E} \ := \ \bigoplus_{f \in F} \cB \cS_{f,\circ}^{\E},
\]
and denote by $\cB\cS^{\E}$ the
object obtained by base change to $\Fqb$.

\begin{lem}
\label{lem:category-generated-X}

\begin{enumerate}
\item The category $\cDb_{(B)}(X,\E)$ is generated, as a triangulated
category, by the objects $\cB\cS^{\E}_f$ for $f \in F$.
\item The category $\cDb_{(B)}(X^s,\E)$ is generated, as a triangulated
category, by the objects $\pi_{s!} \cB\cS^{\E}_f$ for $f \in F$.
\end{enumerate}

\end{lem}

\begin{proof}
(1) A straightforward induction on the support using standard
  distinguished triangles shows that the set $\{ \Delta_x^{\E} \; | \; x \in
  W \}$ generates $\cDb_{(B)}(X,\E)$ as a triangulated category. (If $\E=\O$, one has to observe that for any $m \in \Z_{\geq 1}$ there is a natural exact sequence $\Delta_x^{\O} \into \Delta_x^{\O} \onto \Delta_{x,m}^{\O}$.) Hence
  it is enough to show that one can obtain $\Delta_x^{\E}$ as a successive
  extension of the complexes $\cB\cS_{f}^{\E}$. However, if $\overline{f}=x^{-1}$ then, by Proposition \ref{prop:BS parity}, $\cB\cS_f^{\E}$ is supported on $\overline{X_x}$ and $i_x^! \cB\cS_f^{\E} \cong i_x^* \cB\cS_f^{\E} \cong 
  \underline{\E}_{X_x}[2\ell(x)]$. Hence we have a distinguished triangle
\[
\Delta_x^{\E}[2\ell(x)] \to \cB\cS_f^{\E} \to \cC \triright
\]
with the support of $\cC$ contained in $\overline{X_x}
\setminus X_x$. By induction, $\cC$
belongs to the triangulated subcategory generated by 
Bott--Samelson complexes, hence so does $\Delta_x^{\E}$.

(2) One can easily check that $\cDb_{(B)}(X^s,\E)$ is generated, as a triangulated category, by the essential image of the functor $\pi_{s!} : \cDb_{(B)}(X,\E) \to \cDb_{(B)}(X^s,\E)$. Hence the result follows from (1).
\end{proof}

Consider the algebras
\begin{align*}
E_{\E} \ & := \ \Ext^\bullet_{\cDb_{(B)}(X,\E)}(\cB\cS^{\E}, \cB\cS^{\E}), \\
E_\E^s \ &:= \  \Ext^\bullet_{\cDb_{(B)}(X^s,\E)}(\pi_{s!} \cB\cS^{\E}, \pi_{s!}\cB\cS^{\E})
\end{align*}
for $s \in S$. These graded algebras will play a major role in the rest of the paper.
Note that, as $\cB\cS^{\E}=\For(\cB\cS^{\E}_\circ)$ and $E_\E^s=\For( \pi_{s!} \cB\cS^{\E}_\circ )$
 these algebras are endowed with a natural automorphism $\phi$ (see \S \ref{ss:perverse-sheaves-Fq}). We denote by $e_f$ the idempotent in $E_{\E}$, resp.~$E^s_{\E}$, given by projection to $\cB\cS_f^{\E}$, resp.~$\pi_{s!} \cB\cS_f^{\E}$. Note that the functor $\pi_{s!}$ induces a ring morphism $E_{\E} \to E_{\E}^s$ sending $e_f$ to $e_f$, which justifies our notation.

From now on we concentrate on the case $\E=\O$. The aim of the rest of this section is to prove the following theorem.

\begin{thm}
\label{thm:H-diagonal}
  The $(\O, \phi)$-algebras $E_{\O}$ and $E_\O^s$ are $\O$-free and
vanish in odd degree. Moreover, for any $m \in \Z$, the $(\O,\phi)$-modules $E_\O^{2m}$ and $(E_\O^s)^{2m}$ have $q$-weights obtained from $\{-m\}$.
\end{thm}

Recall (see Example \ref{ex:Tate-sheaf}) that that under the natural
equivalence $\Perv(\pt_\circ ,\O) \cong \Modfr (\O,\phi)$ (see Lemma
\ref{lem:perverse-sheaves-modules} with
$\cP_\circ =\underline{\O}_{\pt_\circ}$), the perverse sheaf
$\underline{\O}_{\pt_\circ}(m)$ corresponds to the
$\O$-module $\O$ endowed with the automorphism $a \mapsto q^m a$. By abuse
of notation, we also denote this $(\O,\phi)$-module by $\O(m)$. 

Using the observation that by adjunction we have
\[
\Ext^n_{\O,X^s}(\pi_{s!} \cB\cS^{\O}_f, \pi_{s!} \cB\cS^{\O}_g) \cong \Ext^n_{\O,X}(\cB\cS^{\O}_f, \pi_s^! \pi_{s!} \cB\cS^{\O}_g) \cong \Ext^n_{\O,X}(\cB\cS^{\O}_f, \cB\cS^{\O}_{(s,g)}),
\]
Theorem \ref{thm:H-diagonal} is a consequence of the following result.

\begin{prop}
  Given sequences of simple reflections $f$ and $g$, the $(\O,\phi)$-module
  $\Ext^n_{\O}(\cB\cS_f^{\O}, \cB\cS_g^{\O})$ is zero for odd $n$ and is an extension
  of copies of $\O(-m)$ for $n = 2m$ even.
\end{prop}

\begin{proof}
  By adjunction and the isomorphisms $\pi_{s!} \cong \pi_{s*}$,  $\pi_s^! \cong
  \pi_s^*[2](1)$ we can assume that $g$ is the empty sequence. Then the result
  follows from the case $v = e$ of Proposition \ref{prop:stalks}(1) below.
\end{proof}

\begin{prop}
\label{prop:stalks}

Let $f$ be a
  sequence of simple reflections.
  \begin{enumerate}
  \item For any $v \in W$, $\cH^j(i_v^*\cB\cS_{f,\circ}^{\O}) = 0$ for odd $j$
    and $\cH^j(i_v^*\cB\cS_{f,\circ}^{\O})$ is an extension of copies of
    $\underline{\O}_{X_{v,\circ}}(-k)$ if $j = 2k$ is even.
\item If $f$ is a reduced expression for $w^{-1}$ then $\cB\cS_{f,\circ}^{\O}$
  is supported on $\overline{X_{w,\circ}}$ and $i_w^* \cB\cS_{f,\circ}^{\O} \cong
  \underline{\O}_{X_{w,\circ}}[2\ell(w)](\ell(w))$.
  \end{enumerate}
  
\end{prop}

\begin{proof}
(1) It is enough to show that if $s$ is a simple reflection and $\cS$ is a
  complex satisfying the conclusions of (1) then so is
  $\pi_s^! \pi_{s!} \cS$. To this end fix $y \in W^s$. Consider the
  locally closed subvariety $Z_\circ$ of $X_\circ$ defined via the Cartesian diagram:
  \begin{equation*}
    \xymatrix{ 
    Z_\circ \ar@{}[dr]|{\square} \ar[d]_-{\pi_s} \ar@{^{(}->}[r]^u & X_\circ \ar[d]^-{\pi_s} \\
X_{y,\circ}^s \ar@{^{(}->}[r]^v & X_{\circ}^s
}
  \end{equation*}
Then $Z_\circ$ is the disjoint union of the closed subvariety
$X_{y,\circ}$ and its open complement $X_{ys,\circ}$. Denote by $i$ and $j$ the
corresponding inclusions:
\[
X_{ys,\circ} \stackrel{j}{\hookrightarrow} Z_\circ
\stackrel{i}{\hookleftarrow} X_{y,\circ}.
\]
On $Z_\circ$ we have a distinguished triangle
\[
j_!j^!u^* \cS \to u^*\cS \to i_*i^*u^* \cS \triright.
\]
Applying $\pi_{s*} = \pi_{s!}$ to this triangle and keeping in mind the base change
isomorphism $\pi_{s!} u^* \cong v^* \pi_{s!}$ and the identity $j^! =
j^*$ we obtain a distinguished triangle:
\begin{equation} \label{eq:triZ}
(\pi_s \circ j)_{!} i_{ys}^* \cS \to v^* \pi_{s!} \cS \to (\pi_s
\circ i)_! i_y^* \cS \triright.
\end{equation}
Now $(\pi_s \circ i)$ is an isomorphism of varieties and $(\pi_s \circ
j)$ is a trivial fibration with fibre $\bA_{\circ}^1$. The
direct image with compact supports along such a fibration sends the constant
sheaf $\underline{\O}$ to the (shifted and twisted) constant sheaf
$\underline{\O}[-2](-1)$. Our assumptions on $\cS$ now guarantee that the long exact
sequence of cohomology of the triangle \eqref{eq:triZ} has only zeros
in odd degree, and extensions of the constant sheaf $\underline{\O}_{X^s_{y,\circ}}(-k)$ in
degree $2k$. Pulling back to $X_\circ$ via $\pi_s^! \cong \pi_s^*[2](1)$ now yields (1).

(2) Assume that $f = (s,t,\cdots, r)$ is a reduced expression
for $w^{-1}$ and let $f' = (t,\cdots, r)$; then $f'$ is a reduced
expression for $sw^{-1} = (ws)^{-1}$. By induction the support of
$\cB\cS_{f',\circ}^{\O}$ is contained in $\overline{X_{ws,\circ}}$ and $i_{ws}^* \cB\cS_{f',\circ}^{\O} =
\underline{\O}_{X_{ws,\circ}}[2\ell(ws)](\ell(ws))$. Fix $y \in W^s$. The
distinguished triangle \eqref{eq:triZ} shows that 
\begin{itemize}
\item[(a)] if $i_{ys}^* \cS =i_y^*\cS = 0$ then $i_{ys}^* (\pi_s^!
  \pi_{s!} \cS) =i_y^*(\pi_s^! \pi_{s!} \cS) = 0$;
\item[(b)] if  $i_{ys}^* \cS = 0$ and
$i_{y}^* \cS \cong \underline{\O}[m](r)$ then $i_{ys}^* (\pi_s^! \pi_{s!} \cS) \cong
\underline{\O}[m+2](r+1)$.
\end{itemize}
It follows that the support of $\cB\cS_{f,\circ}^{\O}$
is contained in $\overline{X_{w,\circ}}$ and that $i_w^* \cB\cS_{f,\circ}^{\O} \cong
\underline{\O}_{X_{w,\circ}}[2\ell(w)](\ell(w))$ as claimed.
\end{proof}

\subsection{Algebraic description of direct and inverse image}
\label{ss:direct-inverse-image}

Let $\cP$ be a projective generator of the category $\Perv_{(B)}(X,\E)$. 

\begin{lem}
\label{lem:direct-image-proj-generator}

The object $\pi_s{}_{[!]} \cP$ is a projective generator of $\Perv_{(B)}(X^s,\E)$.

\end{lem} 

\begin{proof}
In this proof we assume that $\E=\O$; the case $\E=\F$ is similar. By the remarks after the definition of $\pi_s{}_{[!]}$ (see \S \ref{ss:notation}), $\pi_s{}_{[!]} \cP$ is projective. By the remarks in the proof of Corollary \ref{cor:properties-projectives}, it is sufficient to prove that for any $v \in W^s$, there exists $n >0$ such that $\pi_s{}_{[!]} \cP^{\oplus n}$ surjects to $\IC_v$. Fix such a $v$. As $\cP$ is a projective generator there exists $n>0$ and a surjection
\[
f: \cP^{\oplus n} \onto \IC_{vs} \cong \pi_s^{[!]} \IC_v.
\]
By adjunction we obtain a morphism $g: \pi_s{}_{[!]} \cP^{\oplus n} \to \IC_v$ such that $f$ factors as
\[
\cP^{\oplus n} \to \pi_s^{[!]} \pi_s{}_{[!]} \cP^{\oplus n} \xrightarrow{\pi_s^{[!]} g} \pi_s^{[!]} \IC_v.
\]
Hence $\pi_s^{[!]} g$ is surjective. As the functor $\pi_s^{[!]}$ is exact and fully faithful (see \cite[Proposition 4.2.5]{BBD}), hence kills no object, we deduce that $g$ is surjective.
\end{proof}

We set
\[
A:= \End_{\Perv_{(B)}(X,\E)}(\cP), \qquad A^s:=\End_{\Perv_{(B)}(X^s,\E)}(\pi_s{}_{[!]} \cP),
\]
so that we have equivalences
\begin{align*}
\Pi := \Hom(\cP,-) : \Perv_{(B)}(X,\E) \ & \xrightarrow{\sim} \ \Modfr A, \\
\Pi^s := \Hom(\pi_s{}_{[!]} \cP,-) : \Perv_{(B)}(X^s,\E) \ & \xrightarrow{\sim} \ \Modfr A^s.
\end{align*}
We denote similarly the induced equivalences between derived categories. There is a natural algebra morphism $A \to A^s$ induced by the functor $\pi_s{}_{[!]}$, hence associated functors
\begin{equation}
\label{eqn:functors-A}
\vcenter{
\xymatrix@C=2.5cm{
\Modfr A \ar@<0.5ex>[r]^-{(-) \otimes_A A^s} & \Modfr A^s \ar@<0.5ex>[l]^-{\mathsf{res}}.
}
}
\end{equation}

\begin{lem}
\label{lem:direct-inverse-image}

The following diagram commutes up to natural transformations:
\[
\xymatrix@C=3cm{
\Perv_{(B)}(X,\E) \ar@<0.5ex>[d]^-{\pi_s{}_{[!]}} \ar[r]^-{\Pi}_-{\sim} & \Modfr A \ar@<0.5ex>[d]^-{(-) \otimes_A A^s} \\
\Perv_{(B)}(X^s,\E) \ar@<0.5ex>[u]^-{\pi_s^{[!]}} \ar[r]^-{\Pi^s}_-{\sim} & \Modfr A^s. \ar@<0.5ex>[u]^-{\mathsf{res}}
}
\]

\end{lem}

\begin{proof}
The fact that
\[
\mathsf{res} \circ \Hom(\pi_s{}_{[!]} \cP,-) \ \cong \ \Hom(\cP,-) \circ \pi_s^{[!]}
\]
follows from the property that $\pi_s{}_{[!]}$ is left adjoint to $\pi_s^{[!]}$. The second isomorphism of functors follows from the first one by adjunction.
\end{proof}

As recalled above, the functor $\pi_s^{[!]}$ is exact; we denote similarly the functor induced between derived categories. The functor $\pi_s{}_{[!]}$ is right exact; we denote its left derived functor by $L\pi_s{}_{[!]}$. The values of this functor can be computed using projective resolutions (which exist thanks to Corollary \ref{cor:properties-projectives}). In particular, it follows from Corollary \ref{cor:realization-functor}(2) that $L\pi_s{}_{[!]}$ restricts to a functors between \emph{bounded} derived categories. Note that the functor $L\pi_s{}_{[!]}$ is left adjoint to $\pi_s^{[!]}$. The same remarks apply to the functors \eqref{eqn:functors-A}

We denote by
\begin{align*}
\mathsf{real} : \cDb \Perv_{(B)}(X,\E) & \to \cDb_{(B)}(X,\E), \\
\mathsf{real}^s : \cDb \Perv_{(B)}(X^s,\E) & \to \cDb_{(B)}(X^s,\E)
\end{align*}
the realization functors. Recall that these functors are equivalences of categories (see Corollary \ref{cor:realization-functor}(1) if $\E=\O$, and \cite[Corollary 3.3.2]{BGS} if $\E=\F$).

\begin{prop}
\label{prop:direct-inverse-image}

The following diagram commutes up to natural transformations:
\[
\xymatrix@C=2cm{
\cDb_{(B)}(X,\E) \ar@<0.5ex>[d]^-{\pi_{s!}[1]} & \cDb \Perv_{(B)}(X,\E) \ar@<0.5ex>[d]^-{L\pi_s{}_{[!]}} \ar[r]^-{\Pi}_-{\sim} \ar[l]_-{\mathsf{real}}^-{\sim} & \cDb \bigl( \Modfr A \bigr) \ar@<0.5ex>[d]^-{(-) \lotimes_A A^s} \\
\cDb_{(B)}(X^s,\E) \ar@<0.5ex>[u]^-{\pi_s^![-1]} & \cDb \Perv_{(B)}(X^s,\E) \ar@<0.5ex>[u]^-{\pi_s^{[!]}} \ar[r]^-{\Pi^s}_-{\sim} \ar[l]_-{\mathsf{real}^s}^-{\sim} & \cDb \bigl( \Modfr A^s \bigr). \ar@<0.5ex>[u]^-{\mathsf{res}}
}
\]

\end{prop}

\begin{proof}
Commutativity of the right square follows from Lemma \ref{lem:direct-inverse-image}. Let us consider the left square. The isomorphism
\[
\mathsf{real} \circ \pi_s^{[!]} \ \cong \ \pi_s^![-1] \circ \mathsf{real}^s
\]
follows from \cite[Lemma A.7.1]{Be}. The second isomorphism follows by adjunction, since $\mathsf{real}$ and $\mathsf{real}^s$ are equivalences.
\end{proof}

Now we assume that there exists $\cP_\circ $ in $\Perv_{(B_\circ )}(X_\circ ,\E)$ such that $\cP=\For(\cP_\circ )$. (This is possible thanks to Lemma \ref{lem:proj-generator-X_0} or \S\ref{ss:geom-projective-generator}.) Then the algebra $A$, resp.~$A^s$, is endowed with a natural automorphism $\phi$, resp.~$\phi^s$. By Lemma \ref{lem:perverse-sheaves-modules}, the functors $\Hom_{\E}(\cP,-)$ and $\Hom_{\E}(\pi_s {}_{[!]} \cP,-)$ induce natural equivalences of categories
\begin{align*}
\Pi_\circ : \Perv_{(B_\circ )}(X_\circ ,\E) \ & \xrightarrow{\sim} \ \Modfr (A,\phi), \\
\Pi^s_\circ : \Perv_{(B_\circ )}(X^s_\circ ,\E) \ & \xrightarrow{\sim} \ \Modfr (A^s,\phi^s).
\end{align*}
The morphism $A \to A^s$ commutes with the automorphisms $\phi$ and $\phi^s$, so that we obtain functors
\begin{align*}
\mathsf{res} : \Modfr (A^s,\phi^s) & \to \Modfr (A,\phi), \\
(-) \otimes_A A^s : \Modfr (A,\phi) & \to \Modfr (A^s,\phi^s)
\end{align*}
which are compatible with the functors \eqref{eqn:functors-A} in the obvious sense. The proof of the following result is similar to that of Lemma \ref{lem:direct-inverse-image}.

\begin{lem}
\label{lem:direct-inverse-image-0}

The following diagram commutes up to natural transformations:
\[
\xymatrix@C=3cm{
\Perv_{(B_\circ )}(X_\circ ,\E) \ar@<0.5ex>[d]^-{\pi_s{}_{[!]}} \ar[r]^-{\Pi_\circ }_-{\sim} & \Modfr (A,\phi) \ar@<0.5ex>[d]^-{(-) \otimes_A A^s} \\
\Perv_{(B_\circ )}(X^s_\circ ,\E) \ar@<0.5ex>[u]^-{\pi_s^{[!]}} \ar[r]^-{\Pi^s_{\circ}}_-{\sim} & \Modfr (A^s,\phi^s). \ar@<0.5ex>[u]^-{\mathsf{res}}
}
\]

\end{lem}

We denote by
\begin{align*}
\mathsf{real}_\circ : \cDb \Perv_{(B_\circ )}(X_\circ ,\E) & \to \cDb_{(B_\circ )}(X_\circ ,\E), \\
\mathsf{real}_{\circ}^s : \cDb \Perv_{(B_\circ )}(X^s_\circ ,\E) & \to \cDb_{(B_\circ )}(X^s_\circ ,\E)
\end{align*}
the realization functors. These functors are equivalences by Proposition \ref{prop:realization-equivalence-0}. It is easy to prove that the functor $(-) \otimes_A A^s$ admits a left derived functor, which can be computed using resolutions in $\Modfr (A,\phi)$ whose terms have a flat underlying right $A$-module. This derived functor will be denoted by
\[
(-) \, \lotimes_A \, A^s : \cDb \bigl( \Modfr (A,\phi) \bigr) \to \cDb \bigl( \Modfr (A^s,\phi^s) \bigr).
\]
Using Lemma \ref{lem:direct-inverse-image-0}, one deduces (or one proves directly) that the functor $\pi_s{}_{[!]}$ also admits a left derived functor
\[
L\pi_s{}_{[!]} :  \cDb \Perv_{(B_\circ )}(X_\circ ,\E) \to \cDb \Perv_{(B_\circ )}(X^s_\circ ,\E).
\]
The proof of the following result is similar to that of Proposition \ref{prop:direct-inverse-image}.

\begin{prop}
\label{prop:direct-inverse-image-0}

The following diagram commutes up to natural transformations:
\[
\xymatrix@C=2cm{
\cDb_{(B_\circ )}(X_\circ ,\E) \ar@<0.5ex>[d]^-{\pi_{s!}[1]} & \cDb \Perv_{(B_\circ )}(X_\circ ,\E) \ar@<0.5ex>[d]^-{L\pi_s{}_{[!]}} \ar[r]^-{\Pi_\circ }_-{\sim} \ar[l]_-{\mathsf{real}_\circ }^-{\sim} & \cDb \bigl( \Modfr (A,\phi) \bigr) \ar@<0.5ex>[d]^-{(-) \lotimes_A A^s} \\
\cDb_{(B_\circ )}(X^s_{\circ},\E) \ar@<0.5ex>[u]^-{\pi_s^![-1]} & \cDb \Perv_{(B_\circ )}(X^s_\circ ,\E) \ar@<0.5ex>[u]^-{\pi_s^{[!]}} \ar[r]^-{\Pi_\circ ^s}_-{\sim} \ar[l]_-{\mathsf{real}_{\circ}^s}^-{\sim} & \cDb \bigl( \Modfr (A^s,\phi^s) \bigr). \ar@<0.5ex>[u]^-{\mathsf{res}}
}
\]

\end{prop}

\subsection{Existence of a decomposable projective
  resolution} \label{subsec: existence decomposable}

From now on we will have to work simultaneously with coefficients $\O$ and $\F$, but to distinguish the two cases. We add superscripts or indices to indicate which coefficients we are working with.

We fix an object $\cP_{\circ}^{\O}$ in $\Perv_{(B_\circ )}(X_\circ ,\O)$ such that $\cP^{\O}:=\For(\cP_\circ ^{\O})$ is a projective generator of the category $\Perv_{(B)}(X,\O)$. For simplicity, we also assume that there exists a surjection $\cP_{\circ}^{\O} \onto \IC_{e,\circ}^{\O}$. Set $\cP_{\circ}^{\F}:=\F(\cP_\circ^{\O})$ and $\cP^{\F}:=\For(\cP_{\circ}^{\F}) \cong \F(\cP^{\O})$. One can easily check that $\cP^{\F}$ is a projective generator of the category $\Perv_{(B)}(X,\F)$.

Set
\begin{align*}
A_{\O}:=\End_{\O}(\cP^{\O}), & \quad A_{\O}^s:=\End_{\O}(\pi_s{}_{[!]} \cP^{\O}), \\
A_{\F}:=\End_{\F}(\cP^{\F}), & \quad A_{\F}^s:=\End_{\F}(\pi_s{}_{[!]} \cP^{\F})
\end{align*}
(for any simple reflection $s$). Note that by Lemma \ref{lem:direct-image-proj-generator} and  Proposition \ref{prop:properties-projectives-2}(4), $A_{\O}$ and $A^s_{\O}$ are $\O$-free (for any $s$) and the natural morphism
\begin{equation}
\label{eqn:modular-reduction-A}
\F \otimes_{\O} A_{\O} \to A_{\F}
\end{equation}
is an isomorphism. We denote by $\phi$, resp.~$\phi^s$, the automorphism induced by the Frobenius on $A_{\O}$, resp.~on $A^s_{\O}$. In addition, we use the same notation as in \S \ref{ss:direct-inverse-image}.

To prove an analogue of isomorphism \eqref{eqn:modular-reduction-A}
for $A^s$, we need some preparation. We denote by
$\mathsf{Proj}(X,\O)$, $\mathsf{Proj}(X^s,\O)$ the categories of
projective objects in $\Perv_{(B)}(X,\O)$, $\Perv_{(B)}(X^s,\O)$, and
similarly for coefficients $\F$. Note that the functors $\pi_{s[!]}$
and $\F = \F \otimes_\O^L(-)$ send projective perverse sheaves to projective perverse sheaves (see \S \ref{ss:notation} and Proposition \ref{prop:properties-projectives-2}).

\begin{lem}
\label{lem:F-pi-projectives}

There exists an isomorphism of functors
\[
\pi_{s[!]} \circ \F \ \cong \ \F \circ \pi_{s[!]} : \mathsf{Proj}(X,\O) \to \mathsf{Proj}(X^s,\F).
\]

\end{lem}

\begin{proof}
First, let $\cM$ be in $\cDb_{(B)}(X,\O)$ with $\pH^i(\cM)=0$ for $i>0$. Consider the canonical truncation triangle
\[
\cN \to \cM \to \pH^0(\cM) \triright
\]
and its modular reduction
\[
\F \cN \to \F \cM \to \F \bigl( \pH^0(\cM) \bigr) \triright.
\]
As the functor $\F$ is right exact, the second and third terms in the latter triangle have canonically isomorphic $0$-th perverse cohomology. Hence there exists a functorial morphism
\[
\F \bigl( \pH^0(\cM) \bigr) \to \pH^0 \bigl( \F \bigl( \pH^0(\cM) \bigr) \bigr) \cong \pH^0 \bigl( \F \cM \bigr),
\]
which is an isomorphism if the left-hand side is a perverse sheaf.

Now, if $Q$ is in $\mathsf{Proj}(X,\O)$ we obtain isomorphisms
\[
\F(\pi_{s[!]} Q) := \F \bigl( \pH^0(\pi_{s!} Q [1]) \bigr) \xrightarrow{\sim} \pH^0 \bigl( \F(\pi_{s!} Q) [1] \bigr) \cong \pH^0 \bigl( \pi_{s!} (\F Q) [1] \bigr) = \pi_{s[!]} (\F Q),
\]
which finishes the proof.
\end{proof}

\begin{prop}
\label{prop:modular-reduction-As}

There exists an isomorphism of $\F$-algebras $\F \otimes_{\O} A^s_{\O} \cong A^s_{\F}$ such that the following diagram commutes:
\[
\xymatrix{
\F \otimes_{\O} A_{\O} \ar[r] \ar[d]^-{\wr}_{\eqref{eqn:modular-reduction-A}} & \F \otimes_{\O} A^s_{\O} \ar[d]_-{\wr} \\
A_{\F} \ar[r] & A^s_{\F}
}
\]
where horizontal morphisms are induced by the functor $\pi_{s[!]}$.

\end{prop}

\begin{proof}
Our isomorphism can be defined as the composition
\[
\F \otimes_{\O} A^s_{\O} = \F \otimes_{\O} \End_{\O}(\pi_{s[!]} \cP^{\O}) \xrightarrow{\sim} \End_{\F}(\F(\pi_{s[!]} \cP^{\O})) \xrightarrow{\sim} \End_{\F}(\pi_{s[!]} \cP^{\F})
\]
where the first isomorphism follows from Proposition \ref{prop:properties-projectives-2}(4), and the second one follows from Lemma \ref{lem:F-pi-projectives}. Commutativity of the diagram follows from functoriality.
\end{proof}

We consider the following conditions:
\begin{equation}
\label{eqn:condition}
A_{\O} \text{ is } q\text{-decomposable, and } A^s_{\O} \text{ is } q\text{-decomposable for any } s \in S.
\end{equation}
Recall that, by Propositions \ref{prop:weights-morphisms-projectives} and \ref{prop:End-decomposable-partial}, if the order of $q$ in $\F$ is strictly bigger than $2 \ell(w_0)$ then one can choose $\cP^{\O}_\circ$ such that condition \eqref{eqn:condition} is satisfied. However, most of our results below hold under this condition, independently of the proofs in Section \ref{sec:bounding-weights}.

\begin{lem}
\label{lem:morphisms-e-decomposable}

Assume that condition \eqref{eqn:condition} is satisfied. Then the $(A_{\O},\phi)$-module
\[
\Hom_{\O}(\cP^{\O},\IC_{e}^{\O})
\]
is $\O$-free and $q$-decomposable. Moreover, the natural morphism
\[
\F \otimes_{\O} \Hom_{\O}(\cP^{\O},\IC_{e}^{\O}) \to \Hom_{\F}(\cP^{\F},\IC_e^{\F})
\]
is an isomorphism.

\end{lem}

\begin{proof} The fact that  $\Hom_{\O}(\cP^{\O},\IC_{e}^{\O})$ is
  $\O$-free and the final claim follow from Proposition
  \ref{prop:properties-projectives-2}(2) because $\IC_e^{\O} =
  \Delta_e$. By our choice of $\cP_{\circ}^{\O}$ we have a surjection of
  $(A_{\O},\phi)$-modules 
\[
A_{\O} \twoheadrightarrow \Hom_{\O}(\cP^{\O},\IC_{e}^{\O}).
\]
Hence the fact that $\Hom_{\O}(\cP^{\O},\IC_{e}^{\O})$ is
$q$-decomposable follows from Lemma
\ref{lem:quotient-decomposable}.
\end{proof}

Recall the objects $\cB\cS_{\circ}^{\O}$, $\cB\cS_{\circ}^{\F}$ defined in \S
\ref{subsec:BS}. By definition we have an isomorphism $\F(\cB\cS_{\circ}^{\O}) \cong
\cB\cS_{\circ}^{\F}$. Recall also the category $\Modfrdec (A_{\O},\phi)$ defined in \S \ref{ss:O-modules-automorphisms}. We denote by
\[
\Modfrdecproj (A_{\O},\phi)
\subset
\Modfrdec (A_{\O},\phi)
\]
the full additive subcategory with objects those $(P,\phi_P)$ such that $P$ is a projective $A_{\O}$-module. Note
that if $(P,\phi_P)$ is in
$\Modfrdecproj (A_{\O},\phi)$, then $P$ is
$\O$-free and $\F \otimes_{\O} P$ is a projective $A_{\F}$-module,
endowed with a natural automorphism induced by $\phi_P$. We denote by
$\F(P,\phi_P)$ the corresponding object of $\Modfr (A_{\F},\phi)$.

The following result is the technical key to our proof of formality.

\begin{prop}
\label{prop:projective-resolution-decomposable}

Assume that condition \eqref{eqn:condition} is satisfied. Then there exists a bounded complex $M^{\bullet}_\circ$ of objects of $\Modfrdecproj (A_{\O},\phi)$ which satisfies the following conditions:
\begin{enumerate}
\item the image of $M^{\bullet}_\circ$ in $\cDb \bigl( \Modfr (A_{\O},\phi) \bigr)$ coincides (up to isomorphism) with the image of $\cB\cS_{\circ}^{\O}$ under the equivalence
\[
\Pi_{\circ}^{\O} \circ (\mathsf{real}_{\circ}^{\O})^{-1} : \cDb_{(B_\circ )}(X_\circ ,\O) \ \xrightarrow{\sim} \ \cDb \bigl( \Modfr (A_{\O},\phi) \bigr);
\]
\item the image of $\F(M^{\bullet}_{\circ})$ in $\cDb \bigl( \Modfr (A_{\F},\phi) \bigr)$ coincides (up to isomorphism) with the image of $\cB\cS_{\circ}^{\F}$ under the equivalence
\[
\Pi_{\circ}^{\F} \circ (\mathsf{real}_{\circ}^{\F})^{-1} : \cDb_{(B_{\circ})}(X_\circ ,\F) \ \xrightarrow{\sim} \ \cDb \bigl( \Modfr (A_{\F},\phi) \bigr).
\]
\end{enumerate}

\end{prop}

\begin{proof}
Let $f \in F$, and write $f=(s,t,\cdots,r)$. By Proposition \ref{prop:direct-inverse-image-0}, the image of the object $\cB\cS_{f,\circ}^{\O}$ of $\cDb_{(B_\circ )}(X_\circ ,\O)$ under $\Pi_{\circ}^{\O} \circ (\mathsf{real}_{\circ}^{\O})^{-1}$ is isomorphic to the object
\begin{equation}
\label{eqn:BS-modules}
\Pi_{\circ}^{\O}(\IC_{e,\circ}^{\O}) \, \lotimes_{A_{\O}} \, A^r_{\O} \, \lotimes_{A_{\O}} \, \cdots \, \lotimes_{A_{\O}} \, A^t_{\O} \, \lotimes_{A_{\O}} \, A^s_{\O}
\end{equation}
of $\cDb \bigl( \Modfr (A,\phi) \bigr)$, where for any $u$, $A^u_{\O}$ is considered as an $A_{\O}$-bimodule with compatible automorphism. By Lemma \ref{lem:morphisms-e-decomposable}, $\Pi_{\circ}^{\O}(\IC_{e,\circ}^{\O})=\Hom_{\O}(\cP^{\O},\IC_{e}^{\O})$ is a $q$-decomposable, $\O$-free right $(A_{\O},\phi)$-module. 

For any $u$ that appears in $f$ one can construct a resolution
\[
Q^{\bullet}_u \xrightarrow{\qis} A^u_{\O}
\]
of $A^u_{\O}$ by objects of $\Modfrdecproj (A_{\O} \otimes_{\O} A_{\O}^{\mathrm{op}},\phi \otimes \phi)$ which is concentrated in non-positive degrees (but possibly infinite). Indeed, by Lemma \ref{lem:projective-decomposable} there exists an object $Q^0$ in $\Modfrdecproj (A_{\O} \otimes_{\O} A_{\O}^{\mathrm{op}},\phi \otimes \phi)$ and a surjection $Q^0 \onto A^u_{\O}$. By Lemma \ref{lem:quotient-decomposable} the kernel of this surjection is $q$-decomposable, and it is $\O$-free, hence we can repeat the argument.

Then one can define the complex
\[
{}' M^{\bullet}_{f,\circ} \ := \ \Pi_{\circ}^{\O}(\IC_{e,\circ}^{\O}) \otimes_{A_{\O}} Q^{\bullet}_r \otimes_{A_{\O}} \cdots \otimes_{A_{\O}} Q^{\bullet}_t \otimes_{A_{\O}} Q^{\bullet}_s.
\]
Since a projective right $A_{\O} \otimes_{\O} A_{\O}^{\mathrm{op}}$-module is also projective as a left $A_{\O}$-module, the image of ${}' M^{\bullet}_{f,\circ}$ in $\cDb \bigl( \Modfr (A_{\O},\phi) \bigr)$ is \eqref{eqn:BS-modules}, hence is isomorphic to $\Pi_{\circ}^{\O} \circ (\mathsf{real}_{\circ}^{\O})^{-1} (\cB\cS_{f,\circ}^{\O})$. 

By Lemma \ref{lem:tensor-product-projective} below the underlying $A_{\O}$-module of each $M_{f,\circ}^k$ is projective, and each $M^k_{f,\circ}$ is $q$-decomposable by Lemma \ref{lem:quotient-decomposable}. Indeed, it is a finite direct sum of torsion-free quotients of $(\O,\phi)$-modules of the form
\[
\Pi_{\circ}^{\O}(\IC_{e,\circ}^{\O}) \otimes_{\O} Q_r^{k_r} \otimes_{\O} \cdots \otimes_{\O} Q_t^{k_t} \otimes_{\O} Q_s^{k_s},
\]
the latter being $\O$-free and $q$-decomposable.

We claim that one can choose $k \ll 0$ such that $M_{f,\circ}^\bullet:=\tau^{\geq k}({}' M_{f,\circ}^\bullet)$ is still a complex of objects of $\Modfrdecproj (A_{\O},\phi)$ whose image in the derived category is \eqref{eqn:BS-modules}. Indeed the complex ${}' M_{f,\circ}^\bullet$ has bounded cohomology, hence if $k \ll 0$ the natural morphism ${}'M_{f,\circ}^\bullet \to M_{f,\circ}^\bullet$ is a quasi-isomorphism. As the category $\Modfr A_\O$ has finite projective dimension, if $k \ll 0$ then the components of $M_{f,\circ}^\bullet$ still have an image in $\Modfr A_\O$ which is projective. And finally (again if $k \ll 0$), the only component of $M_{f,\circ}^\bullet$ which is neither $0$ nor a component of ${}'M_{f,\circ}^\bullet$ is $q$-decomposable by Lemma \ref{lem:quotient-decomposable}.

Set
\[
M_{\circ}^{\bullet} \ := \ \bigoplus_{f \in F} \, M_{f,\circ}^{\bullet}.
\]
Then, by the remarks above, this complex satisfies condition $(1)$. To prove that it satisfies condition $(2)$, one observes that the morphisms $Q^{\bullet}_u \to A_{\O}^u$ considered above are quasi-isomorphisms of complexes of free $\O$-modules; hence the induced morphisms $(\F \otimes_{\O} Q^{\bullet}_u) \to (\F \otimes_{\O} A_{\O}^u)$ are also quasi-isomorphisms. By the same argument, the morphisms $\F({}'M_{f,\circ}^\bullet) \to \F(M^\bullet_{f,\circ})$ are quasi-isomorphisms. Finally, by Lemma \ref{lem:morphisms-e-decomposable} we have $\F \otimes_{\O} \Pi_{\circ}^{\O}(\IC_{e,\circ}^{\O}) \cong \Pi_{\circ}^{\F}(\IC_{e,\circ}^{\F})$ and by Proposition \ref{prop:modular-reduction-As} we have $\F \otimes_{\O} A^u_{\O} \cong A^u_{\F}$. Hence the same arguments as for $\O$ show that $M^{\bullet}_{\circ}$ also satisfies condition $(2)$.
\end{proof}

\begin{lem}
\label{lem:tensor-product-projective}

Let $R$ be any $\O$-algebra which is finitely generated and free over $\O$.

Let $Q$ be a projective $(R,R)$-bimodule and $M$ be a right $R$-module which is $\O$-free. Then the right $R$-module $M \otimes_R Q$ is projective (hence also $\O$-free).

\end{lem}

\begin{proof}
It suffices to prove that $M \otimes_R (R \otimes_{\O} R)$ is projective over $R$. However we have
\[
M \otimes_R (R \otimes_{\O} R) \ \cong \ M \otimes_{\O} R,
\]
hence this fact is clear from our hypothesis on $M$.
\end{proof}

\subsection{Dg-algebras of finite global dimension}
\label{ss:DGDf}

Let $R$ be a graded algebra, considered as a dg-algebra with differential $d=0$. We assume that $R$ is noetherian as an algebra, and that the category of finitely generated graded $R$-modules has finite projective dimension. There are a priori several ways to define the derived category of finitely generated right $R$-dg-modules. The goal of this subsection is to show that, under our assumptions, all of them coincide.

More precisely, we denote by $\sfD_1$ the category obtained from the homotopy category of right $R$-dg-modules which are finitely generated as $R$-modules by inverting quasi-isomorphisms. We denote by $\sfD_2$ the full subcategory of $\DGDr R$ whose objects are $R$-dg-modules which are isomorphic (in $\DGDr R$) to an $R$-dg-module which is finitely generated over $R$. (Note that it is not clear a priori that $\sfD_2$ is a triangulated subcategory of $\DGDr R$.) We denote by $\sfD_3$ the strictly full triangulated subcategory of $\DGDr R$ generated by finitely generated projective graded $R$-modules (considered as dg-modules with trivial differential). Finally, we denote by $\sfD_4$ the full subcategory of $\DGDr R$ whose objects have their cohomology finitely generated over $R$. We have inclusions
\[
\sfD_3 \subset \sfD_4, \quad \sfD_2 \subset \sfD_4
\]
and a natural essentially surjective functor
\begin{equation}
\label{eqn:functor-DGDf}
\sfD_1 \to \sfD_2.
\end{equation}

\begin{prop}
\label{prop:DGDf}

We have $\sfD_2=\sfD_4=\sfD_3$, and \eqref{eqn:functor-DGDf} is an equivalence of categories.

\end{prop}

Below we will denote by $\DGDfr R$ any of the equivalent triangulated categories $\sfD_i$ ($i=1, 2, 3, 4$).

The main step in the proof of Proposition \ref{prop:DGDf} is the following lemma.

\begin{lem}
\label{lem:DGDf}

We have $\sfD_4 \subset \sfD_3$.

\end{lem}

\begin{proof}
We will prove that any object $M$ in $\sfD_4$ is in $\sfD_3$ by induction on the projective dimension of $\Ho^{\bullet}(M)$.

First, assume that $\Ho^{\bullet}(M)$ has projective dimension $0$, i.e.~that it is projective. We claim that in this case there exists a quasi-isomorphism of $R$-dg-modules
\[
\Ho^{\bullet}(M) \xrightarrow{\qis} M
\]
where $\Ho^{\bullet}(M)$ is considered as a dg-module with trivial differential. Indeed, as $\Ho^{\bullet}(M)$ is projective, there exists a splitting $s : \Ho^{\bullet}(M) \to \ker(d_M)$ of the natural morphism of graded $R$-modules $\ker(d_M) \onto \Ho^{\bullet}(M)$. Then the composition $\Ho^{\bullet}(M) \to \ker(d_M) \into M$ is the desired quasi-isomorphism. As $\Ho^{\bullet}(M)$ is in $\sfD_3$, we deduce that $M$ is also in $\sfD_3$.

Now, assume that $\Ho^{\bullet}(M)$ has projective dimension $d$. Choose a minimal length projective resolution
\[
P^d \into \cdots \to P^0 \overset{a}{\onto} \Ho^{\bullet}(M)
\]
as a graded $R$-module. By the same argument as before, there exists a morphism of $R$-dg-modules $P^0 \to M$ which induces $a$ in cohomology. We denote by $N$ the cocone of this morphism, so that we have a distinguished triangle
\[
N \to P^0 \to M \triright.
\]
Then $\Ho^{\bullet}(N) \cong \ker(a)$ has projective dimension $d-1$, hence $N$ belongs to $\sfD_3$ by induction. We deduce that $M$ is in $\sfD_3$ also.
\end{proof}

\begin{proof}[Proof of Proposition {\rm \ref{prop:DGDf}}]
By Lemma \ref{lem:DGDf} we have $\sfD_3=\sfD_4$. Now, as morphisms in $R\DGD$ between projective graded $R$-modules (considered as dg-modules with $d=0$) are simply morphisms of graded $R$-modules (because such objects are $K$-projective in the sense of \cite[Definition 10.12.2.1]{BL}), one easily checks that $\sfD_3 \subset \sfD_2$. We deduce that $\sfD_2=\sfD_3=\sfD_4$.

Now consider the functor $\sfD_1 \to \sfD_2 = \sfD_3$. The same proof as in Lemma \ref{lem:DGDf} shows that the category $\sfD_1$ is also generated, as a triangulated category, by finitely generated projective graded $R$-modules. Hence it is enough to check that morphisms between such objects in $\sfD_1$ and $\sfD_2$ coincide; this is obvious.
\end{proof}

\subsection{Proof of formality}
\label{ss:formality}

The following result is well known, and is usually attributed to Deligne, see \cite{DGMS, Del}.

\begin{lem}
\label{lem:formality-criterion}

Let $R^{\bullet}$ be a dg-algebra which is endowed with an additional $\Z$-grading
\[
R^i \ = \ \bigoplus_{j \in \Z} \, R^{i,j}
\]
such that $d_R(R^{i,j}) \subset R^{i+1,j}$. Assume that for any $i \in \Z$ the cohomology $\Ho^i(R^{\bullet})$ is concentrated in degree $j=i$ (for this additional grading). Then $R^{\bullet}$ is formal. More precisely, there exists a dg-subalgebra $R^{\bullet}_{\rhd} \subset R^{\bullet}$ and quasi-isomorphisms
\[
\xymatrix{
R^{\bullet} & R^{\bullet}_{\rhd} \ar@{_{(}->}[l]_-{\qis} \ar[r]^-{\qis} & \Ho^{\bullet}(R^{\bullet}),
}
\]
where $\Ho^{\bullet}(R^{\bullet})$ is considered as a dg-algebra with trivial differential.

\end{lem}

\begin{proof}
Denote by $d^i_j : R^{i,j} \to R^{i+1,j}$ the component of the differential in bidegree $(i,j)$. For any $i \in \Z$, set
\[
R^i_{\rhd} \ := \ \bigl( \bigoplus_{j > i} R^{i,j} \bigr) \oplus \ker(d^i_i).
\]
Then $R_{\rhd}^{\bullet}$ is a dg-subalgebra of $R^{\bullet}$. Moreover, the projection $R^i_{\rhd} \onto \ker(d^i_i)/\mathrm{im}(d^{i-1}_{i})$ induces a quasi-isomorphism of dg-algebras $R_{\rhd}^{\bullet} \xrightarrow{\qis} \Ho^{\bullet}(R^{\bullet})$.
\end{proof}

From now on we assume that condition \eqref{eqn:condition} is satisfied. Let $M^{\bullet}_{\circ}$ be as in Proposition \ref{prop:projective-resolution-decomposable}, and let $M^{\bullet}$ be the underlying complex of projective right $A_{\O}$-modules. Consider the dg-algebras
\[
E^{\bullet}_{\O} := \Hom^{\bullet}_{-A_{\O}}(M^{\bullet},M^{\bullet}), \quad E^{\bullet}_{\F} := \Hom^{\bullet}_{-A_{\F}}(F \otimes_{\O} M^{\bullet}, \F \otimes_{\O} M^{\bullet}).
\]
Note that both of these dg-algebras are bounded, that $E^{\bullet}_{\O}$ is a finite rank free $\O$-module, and that we have a natural isomorphism of dg-algebras 
\[
\F \otimes_{\O} E^{\bullet}_{\O} \ \cong \ E^{\bullet}_{\F}.
\]

By the same considerations as in Remark \ref{rmk:automorphisms-Hom}, $E_{\O}^{\bullet}$ is endowed with a natural automorphism, so that it can be considered as an $(\O,\phi)$-dg-algebra. Similarly, by the constructions of \S \ref{ss:perverse-sheaves-Fq} the graded $\O$-algebra $E_{\O}$ can be naturally considered as an $(\O,\phi)$-algebra.

\begin{lem}
\label{lem:H}

There exist natural isomorphisms of graded algebras
\[
\Ho^{\bullet}(E^{\bullet}_{\O}) \cong E_{\O}, \quad \Ho^{\bullet}(E^{\bullet}_{\F}) \cong E_{\F}.
\]
Moreover, the first isomorphism is an isomorphism of $(\O,\phi)$-algebras.

\end{lem}

\begin{proof}
By definition of $E_{\O}^{\bullet}$ we have an isomorphism of graded algebras
\[
\Ho^{\bullet}(E_{\O}^{\bullet}) \ \cong \ \Ext^{\bullet}_{\cDb (\Modfr A_{\O})}(M^{\bullet}, M^{\bullet}).
\]
Using condition $(1)$ in Proposition \ref{prop:projective-resolution-decomposable}, we have an isomorphism
\[
\mathsf{real}^{\O} \circ (\Pi^{\O})^{-1} (M^{\bullet}) \cong \cB\cS^{\O}.
\]
We deduce the first isomorphism. The compatibility with automorphisms follows from Remark \ref{rmk:automorphisms-Hom}. The second isomorphism can be proved similarly.
\end{proof}

\begin{lem}
\label{lem:E-finite-global-dim}

The algebras $E_{\O}$ and $E_{\F}$ have finite global dimension.

\end{lem}

\begin{proof}
We will prove in Corollary \ref{cor:globaldimension} below (using \cite{So2}) that
$E_{\F}$ has finite global dimension. Let us deduce the same property
for $E_{\O}$. We claim that for $M,N$ in $\cDb(\Modfr E_{\O})$ there
exists a natural isomorphism in the derived category of $\F$-vector
spaces 
\begin{equation}
\label{eqn:RHom-O-F}
\F \, \lotimes_{\O} \, R\Hom_{-E_{\O}}(M,N) \ \cong \
R\Hom_{-E_{\F}}(\F \, \lotimes_\O \, M, \F \, \lotimes_\O \, N) 
\end{equation}
where $R\Hom_{-E_{\O}}(M,N)$ is considered as an object in $\cD(\Modfr \O)$. Indeed one can assume that $M$ is a bounded above complex of
finitely generated projective $E_{\O}$-modules, and then reduce the
claim to the case $M=E_{\O}$, which is obvious. 

As $\O$ has global dimension $1$, any object of $\cD(\Modfr \O)$ is
isomorphic to its cohomology. Moreover, if $M$ is in $\Modfr \O$ then $M
\neq 0$ implies $\F \lotimes_{\O} M \neq 0$. Using these remarks,
isomorphism \eqref{eqn:RHom-O-F} and the fact that $E_{\F}$ has finite
global dimension, one easily checks that the same property holds for
$E_{\O}$. 
\end{proof}

Using Lemma \ref{lem:E-finite-global-dim}, one can define the
categories $\DGDfr E_{\O}$ and $\DGDfr E_{\F}$ as in \S
\ref{ss:DGDf}. 

\begin{prop}
\label{prop:formality}

The dg-algebras $E_{\O}^{\bullet}$ and $E_{\F}^{\bullet}$ are formal. More precisely, there exists a dg-subalgebra $E^{\bullet}_{\rhd} \subset E_{\O}^{\bullet}$ and quasi-isomorphisms
\begin{equation}
\label{eqn:diagram-dg-algebras}
\xymatrix{
E^{\bullet}_{\O} & E^{\bullet}_{\rhd} \ar@{_{(}->}[l]_-{\qis} \ar[r]^-{\qis} & E_{\O}
}
\end{equation}
such that the diagram
\[
\xymatrix{
E^{\bullet}_{\F} & \F \otimes_{\O} E_{\rhd}^{\bullet} \ar[l]_-{\qis} \ar[r]^-{\qis} & \F \otimes_{\O} E_{\O}
}
\]
obtained by modular reduction also consists of quasi-isomorphisms.

\end{prop}

\begin{proof}
We claim that the $(\O,\phi)$-module $E_{\O}^{\bullet}$ is $q$-decomposable. Indeed, as $M^{\bullet}_{\circ}$ is $q$-decomposable (see Proposition \ref{prop:projective-resolution-decomposable}), the same holds for the $\O$-free $(\O,\phi)$-module $F^{\bullet}:=\Hom^{\bullet}_{\O}(M^{\bullet},M^{\bullet})$. By definition $E^{\bullet}_{\O}$ is a sub-$(\O,\phi)$-module of $F^{\bullet}$, and the quotient $F^{\bullet}/E_{\O}^{\bullet}$ has no $\pi$-torsion. We deduce from Lemma \ref{lem:quotient-decomposable} that indeed $E^{\bullet}_{\O}$ is $q$-decomposable.

For any $i,j \in \Z$ we set
\[
E^i_{j,\O} \ := \ \{ e \in E^i \mid (\phi - q^{-\frac{j}{2}})^m \cdot e=0 \ \text{for } m \gg 0 \}.
\]
By $q$-decomposability this defines a grading on $E^{\bullet}$ (which is concentrated in even degrees). Hence, using Lemma \ref{lem:H} and Theorem \ref{thm:H-diagonal}, we are in the situation of Lemma \ref{lem:formality-criterion}. We denote by $E_{\rhd}^{\bullet}$ the subalgebra constructed in the proof of this lemma, so that we have a diagram \eqref{eqn:diagram-dg-algebras}. All the dg-algebras in this diagram are $\O$-free: indeed $E_{\rhd}^{\bullet}$ has no $\pi$-torsion since it is a submodule of the free $\O$-module $E^{\bullet}_{\O}$, and $E_{\O}$ is free by Theorem \ref{thm:H-diagonal}. Hence the diagram obtained by modular reduction also consists of quasi-isomorphisms.
\end{proof}

\begin{rmk}
In particular, it follows from Proposition \ref{prop:formality} and Lemma \ref{lem:H} that the natural morphism $\F \otimes_{\O} E_{\O} \to E_{\F}$ is an isomorphism. This can also be proved directly using \cite[equation (2.13)]{JMW} and the fact that $\cB\cS_{\O}$ is parity.
\end{rmk}

Finally we can prove the main result of Part \ref{pt:flag-variety}.

\begin{thm} \label{thm:formality}

Assume that condition \eqref{eqn:condition} is satisfied.

There exist equivalences of triangulated categories
\[
\cDb_{(B)}(X,\O) \ \cong \ \DGDfr E_{\O}, \quad \cDb_{(B)}(X,\F) \ \cong \ \DGDfr E_{\F}
\]
where $E_{\O}$ and $E_{\F}$ are considered as dg-algebras with their natural grading and trivial differential.

\end{thm}

\begin{proof}
We treat only the case of $\O$; the case of $\F$ is similar. As explained in \S \ref{ss:direct-inverse-image}, there exists an equivalence of categories
\[
\cDb_{(B)}(X,\O) \ \cong \ \cDb \bigl( \Modfr A_{\O} \bigr)
\]
Moreover, by Lemma \ref{lem:category-generated-X} the category $\cDb_{(B)}(X,\O)$ is generated (as a triangulated category) by the objects $\{ \cB\cS_f^{\O}, \ f \in F\}$. By construction (see Proposition \ref{prop:projective-resolution-decomposable}), $\cB\cS_f$ corresponds, under this equivalence, to the complex $M^{\bullet}_{f}$ of projective right $A_{\O}$-modules. Hence, with the notation of \S \ref{ss:homological-algebra} we obtain
\[
\cDb_{(B)}(X,\O) \ \cong \ \lan M_f^{\bullet}, \ f \in F \ran_{\Delta} \subset \cD \bigl( \Modr A_{\O} \bigr).
\]
By the result recalled in \S \ref{ss:homological-algebra}, the right-hand side is equivalent to
\[
\lan 1_f E_{\O}^{\bullet}, \ f \in F \ran_{\Delta} \subset \DGDr E^{\bullet}_{\O}
\]
where $1_f$ is the idempotent of $E_{\O}^{\bullet}$ given by the projection to $M_f^{\bullet}$.

Now it follows from Proposition \ref{prop:formality} that there is an equivalence of triangulated categories
\[
\DGDr E_{\O}^{\bullet} \ \cong \ \DGDr E_{\O}.
\]
It sends $1_f E_{\O}^{\bullet}$ to $e_f E_{\O}$, where $e_f \in E_{\O}$ is defined in \S\ref{subsec:BS}. We deduce an equivalence
\[
\cDb_{(B)}(X,\O) \cong \lan e_f E_{\O}, \ f \in F \ran_{\Delta} \subset \DGDr E_{\O}.
\]
Finally, it is not difficult to check (using the category $\mathsf{D}_3$ of \S \ref{ss:DGDf}) that the right-hand side is equivalent to $\DGDfr E_{\O}$.
\end{proof}

\subsection{Parabolic case}

It is not difficult to generalize the results of \S\ref{ss:formality} to parabolic flag varieties. Here, for future reference, we explain the case of the variety $X^s$ for $s\in S$. First, we remark that by an obvious analogue of Lemma \ref{lem:E-finite-global-dim} the rings $E^s_{\O}$ and $E^s_{\F}$ have finite global dimension, so that we can consider the categories $\DGDfr E^s_{\O}$ and $\DGDfr E^s_{\F}$.

\begin{thm}
\label{thm:formality-parabolic}

Let $s \in S$, and assume that condition \eqref{eqn:condition} is satisfied.

There exist equivalences of triangulated categories
\[
\cDb_{(B)}(X^s,\O) \ \cong \ \DGDfr E^s_{\O}, \quad \cDb_{(B)}(X^s,\F) \ \cong \ \DGDfr E^s_{\F}
\]
where $E_{\O}^s$ and $E_{\F}^s$ are considered as dg-algebras with their natural grading and trivial differential.

\end{thm}

\begin{proof}
First we treat the case of $\O$. Recall the resolution $M^{\bullet}_{\circ}$ constructed in the proof of Proposition \ref{prop:projective-resolution-decomposable}. 
Set
\[
M^{\bullet}_{s,\circ} := M^{\bullet}_\circ \otimes_{A_{\O}} A^s_{\O}
\]
and denote by $M^{\bullet}_s$ the complex obtained by forgetting the automorphism $\phi$. Then by construction $M^{\bullet}_{s,\circ}$ is a complex of objects of $\Modfrdecproj (A^s_{\O},\phi^s)$, and by Proposition \ref{prop:direct-inverse-image-0} its image in $\cDb \bigl( \Modfr (A^s_{\O},\phi) \bigr)$ is $\Pi_{\circ}^{s,\O} \circ (\mathsf{real}_{\circ}^{s,\O})^{-1} \bigl( \pi_{s!} \cB\cS_{\circ}^{\O}[1] \bigr)$. Then we consider the dg-algebra
\[
E_{\O}^{s,\bullet} \ := \ \Hom^{\bullet}_{-A_{\O}^s}(M_s^\bullet, M_s^\bullet).
\]
By construction it is endowed with an automorphism $\phi^s$.

The same argument as for Lemma \ref{lem:H} shows that we have an isomorphism of $(\O,\phi)$-algebras $\Ho^{\bullet}(E_{\O}^{s,\bullet}) \cong E_{\O}^s$. Then the same argument as for Proposition \ref{prop:formality} shows that there exists a sub-dg-algebra $E^{s,\bullet}_{\rhd} \subset E^{s,\bullet}_{\O}$ and quasi-isomorphisms
\begin{equation}
\label{eqn:qis-E^s}
\xymatrix{
E^{s,\bullet}_{\O} & \ar@{_{(}->}[l]_-{\qis} E^{s,\bullet}_{\rhd} \ar[r]^-{\qis} & E^s_{\O}
}
\end{equation}
The rest of the proof is similar to that of Theorem \ref{thm:formality}, replacing $E_{\O}^{\bullet}$ by $E_{\O}^{s,\bullet}$.

The case of $\F$ is similar, using the dg-algebra
\[
E_{\F}^{s,\bullet} \ := \ \Hom^{\bullet}_{-A_{\F}^s}(\F \otimes_{\O} M_s^\bullet, \F \otimes_{\O} M_s^\bullet)
\]
and the quasi-isomorphisms
\[
\xymatrix{
E^{s,\bullet}_{\F} & \ar@{_{(}->}[l]_-{\qis} \F \otimes_{\O} E^{s,\bullet}_{\rhd} \ar[r]^-{\qis} & \F \otimes_{\O} E^s_{\O}
}
\]
obtained from \eqref{eqn:qis-E^s} by modular reduction.
\end{proof}

Let us study the compatibility of the equivalences in Theorem \ref{thm:formality} and Theorem \ref{thm:formality-parabolic}. We set $\E=\O$ or $\F$. For any $s \in S$, the functor $\pi_{s!}$ induces an algebra morphism $E_{\E} \to E^s_{\E}$; hence we have a restriction functor
\[
\mathsf{res} : \DGDfr E^s_{\E} \to \DGDfr E_{\E}.
\]

\begin{prop}
\label{prop:formality-res}

Assume that condition \eqref{eqn:condition} is satisfied.

The equivalences in Theorem {\rm \ref{thm:formality}} and Theorem {\rm \ref{thm:formality-parabolic}} (for all $s \in S$) can be chosen so that for any $s \in S$ the following diagram commutes up to natural transformation:
\[
\xymatrix@C=2cm{
\cDb_{(B)}(X^s,\E) \ar[r]^-{\mathrm{Th.~\ref{thm:formality-parabolic}}}_-{\sim} \ar[d]_-{\pi_s^![-1]} & \DGDfr E^s_{\E} \ar[d]^-{\mathsf{res}} \\
\cDb_{(B)}(X,\E) \ar[r]^-{\mathrm{Th.~\ref{thm:formality}}}_-{\sim} & \DGDfr E_{\E}.
}
\]

\end{prop}

\begin{proof}
We consider the equivalences constructed in the proofs of Theorem \ref{thm:formality} and Theorem \ref{thm:formality-parabolic} (using the same notation), and we will show that with this choice, the diagram commutes up to natural transformation. For simplicity we only treat the case $\E=\O$.

First, as the functor $\DGDfr E_{\O} \to \DGDr E_{\O}$ is fully faithful, it is enough to show that the following diagram commutes, where the horizontal arrows are the composition of our equivalences with the obvious embeddings:
\[
\xymatrix{
\cDb_{(B)}(X^s,\O) \ar@{^{(}->}[r] \ar[d]_-{\pi_s^![-1]} & \DGDr E^s_{\O} \ar[d]^-{\mathsf{res}} \\
\cDb_{(B)}(X,\O) \ar@{^{(}->}[r] & \DGDr E_{\O}.
}
\]
Now by construction the horizontal arrows factor through fully faithful functors $\cDb_{(B)}(X^s,\E) \to \DGDr E^{s,\bullet}_{\O}$, resp.~$\cDb_{(B)}(X,\E) \to \DGDr E^{\bullet}_{\O}$. Moreover, the functor $- \otimes_{A_{\O}} A^s_{\O}$ induces a morphism of $(\O,\phi)$-dg-algebras $E_{\O}^{\bullet} \to E_{\O}^{s,\bullet}$ which, by Proposition \ref{prop:direct-inverse-image}, induces our morphism $E_{\O} \to E_{\O}^s$ in cohomology. Hence we obtain morphisms which make the following diagram commutative:
\begin{equation}
\label{eqn:diagram-morphisms-E}
\vcenter{
\xymatrix{
E_{\O}^{\bullet} \ar[d] & E^{\bullet}_{\rhd} \ar[d] \ar[l]_-{\qis} \ar[r]^-{\qis} & E_{\O} \ar[d] \\
E_{\O}^{s,\bullet} & E^{s,\bullet}_{\rhd} \ar[l]_-{\qis} \ar[r]^-{\qis} & E_{\O}^s. 
}
}
\end{equation}
Then the following diagram commutes up to natural transformations, where all functors are restriction functors for the morphisms in \eqref{eqn:diagram-morphisms-E}:
\[
\xymatrix{
\DGDr E_{\O}^{s,\bullet} \ar[d]_-{\mathsf{res}} & \DGDr E^{s,\bullet}_{\rhd} \ar[l]_-{\sim} \ar[r]^-{\sim} \ar[d]^-{\mathsf{res}} & \DGDr E^s_{\O} \ar[d]^-{\mathsf{res}} \\
\DGDr E_{\O}^{\bullet} & \DGDr E^{\bullet}_{\rhd} \ar[l]_-{\sim} \ar[r]^-{\sim} & \DGDr E_{\O}.
}
\]
Hence it is enough to prove the commutativity of the following diagram:
\[
\xymatrix{
\cDb_{(B)}(X^s,\O) \ar@{^{(}->}[r] \ar[d]_-{\pi_s^![-1]} & \DGDr E^{s,\bullet}_{\O} \ar[d]^-{\mathsf{res}} \\
\cDb_{(B)}(X,\O) \ar@{^{(}->}[r] & \DGDr E^{\bullet}_{\O}.
}
\]
Using Proposition \ref{prop:direct-inverse-image}, it is even enough to show the commutativity of the following diagram, where horizontal arrows are the functors considered in the proofs of Theorem \ref{thm:formality} and Theorem \ref{thm:formality-parabolic}:
\begin{equation}
\label{eqn:diagram-inverse-image}
\vcenter{
\xymatrix{
\cDb \bigl( \Modfr A_{\O}^s \bigr) \ar@{^{(}->}[r] \ar[d]_-{\mathsf{res}} & \DGDr E^{s,\bullet}_{\O} \ar[d]^-{\mathsf{res}} \\
\cDb \bigl( \Modfr A_{\O} \bigr) \ar@{^{(}->}[r] & \DGDr E^{\bullet}_{\O}.
}
}
\end{equation}

Now, consider the following diagram:
{\small
\[
\xymatrix@C=0.6cm{
\cDb \Modfr A_{\O}^s \ar[d]_-{\mathsf{res}} & \cKb \Projfr A_{\O}^s \ar[l]_-{\sim} \ar[r]^-{a}_-{\sim} \ar[d]^-{\mathsf{res}} &  \lan 1_f E_{\O}^{s,\bullet} \ran_{\Delta}^{\mathsf{dgHo}} \ar[r]^-{\sim} \ar[d]^-{\mathsf{res}} & \lan 1_f E_{\O}^{s,\bullet} \ran_{\Delta}^{\mathsf{dgDer}} \ar@{^{(}->}[r] \ar[d]^-{\mathsf{res}} & \DGDr E_{\O}^{s,\bullet} \ar[ld]^-{\mathsf{res}} \\
\cDb \Modfr A_{\O} & \cKb \Modfr A_{\O} \ar[l] \ar[r]^-{b} & \DGHr E_{\O}^{\bullet} \ar[r]^-{d} & \DGDr E_{\O}^{\bullet} & \\
& \cKb \Projfr A_{\O} \ar[r]_-{\sim}^-{c} \ar[lu]_-{\sim} \ar[u] & \lan 1_f E_{\O}^{\bullet} \ran_{\Delta}^{\mathsf{dgHo}} \ar[r]^-{\sim} \ar@{^{(}->}[u] & \lan 1_f E_{\O}^{s,\bullet} \ran_{\Delta}^{\mathsf{dgDer}}. \ar@{^{(}->}[u]
}
\]
}Here all unnamed functors are the obvious ones, as well as $d$, the functor $a$ is given by $\Hom^{\bullet}_{-A^s_{\O}}(M^{\bullet}_s, -)$, and the functors $b$ and $c$ are given by $\Hom^{\bullet}_{-A_{\O}}(M^{\bullet}, -)$, so that the upper path from $\cDb \Modfr A_{\O}^s$ to $\DGDr E_{\O}^{s,\bullet}$ is the upper horizontal arrow in \eqref{eqn:diagram-inverse-image}, and the lower path from $\cDb \Modfr A_{\O}$ to $\DGDr E_{\O}^{\bullet}$ is the lower horizontal arrow in \eqref{eqn:diagram-inverse-image}. The categories $\lan 1_f E_{\O}^{s,\bullet} \ran_{\Delta}^{\mathsf{dgHo}} \subset \DGHr E_{\O}^{s,\bullet}$ and $\lan 1_f E_{\O}^{s,\bullet} \ran_{\Delta}^{\mathsf{dgDer}} \subset \DGDr E_{\O}^{s,\bullet}$ are the triangulated subcategories generated by the objects $1_f E_{\O}^{s,\bullet}$ for $f \in F$, and similarly for $E_{\O}^{\bullet}$.

All the squares and triangles in our large diagram are obviously commutative, except for the square involving $a$ and $b$. The commutativity of this diagram follows from the fact that the functor $- \otimes_{A_{\O}} A^s_{\O}$ is left adjoint to $\mathsf{res} : \Modr A^s_{\O} \to \Modr A_{\O}$.

We observe that the functor $d \circ b$ factors through a functor $e : \cDb \Modfr A_{\O} \to \DGDr E_{\O}^{\bullet}$, which must coincide with the lower path in our large diagram. The commutativity of \eqref{eqn:diagram-inverse-image} easily follows from this observation.
\end{proof}

\subsection{Application to Koszulity}
\label{ss:koszulity}

In this subsection we deduce from our results that, under certain
assumptions, the ring $A_{\F}$ can be endowed with a Koszul
grading. This statement is a modular analogue of \cite[Theorem
4.4.4]{BGS} (in the case of the flag variety). It will not be used in
the rest of the paper, except in \S\ref{ss:koszulity-standard}.

The proof will use the notion of a dgg-algebra: such an object is a $\Z^2$-graded algebra endowed with a differential $d$ of bidegree $(1,0)$ which satisfies the Leibniz rule with respect to the first grading, and such that $d^2=0$. One has an obvious notion of dgg-module over a dgg-algebra, and we denote the derived category of right dgg-modules over the dgg-algebra $R$ by $\DGDZr R$.

We also recall some well-known facts on graded algebras. Let $R$ be a finite dimensional graded $\F$-algebra. If $M$ is in $\Modfr R$, an object $M^{\Z}$ of $\Modfrz R$ is called a \emph{graded lift} of $M$ if there exists an isomorphism $v(M^{\Z}) \cong M$, where $v$ is the obvious forgetful functor. Then every simple or indecomposable projective right $R$-module admits a graded lift, which is unique up to isomorphism and shift in the grading. Moreover we have the following easy lemma, whose proof is left to the reader. (See e.g.~\cite[Proposition 2.7.2]{So2} for a similar statement.)

\begin{lem}
\label{lem:graded-lift}

Let $M^{\Z}$ and $M_i^{\Z}$ ($i=1, \cdots, r$) be objects of $\Modfrz R$, such that $v(M_i^{\Z})$ is indecomposable for any $i$. Assume that there exists an isomorphism
\[
v(M^{\Z}) \ \cong \ \bigoplus_i v(M_i^{\Z})
\]
in $\Modfr R$. Then there exist integers $k_i$ ($i=1, \cdots, r$) and an isomorphism
\[
M^{\Z} \ \cong \ \bigoplus_i M_i^{\Z} \lan k_i \ran
\]
in $\Modfrz R$.

\end{lem}

Now we come to our Koszulity result. Here our conventions are slightly different from those of \cite{BGS}: we define the Koszul dual of a Koszul ring $K=\bigoplus_n K_n$ as the graded ring $K^{\dag}:=\Ext^{\bullet}_{-K}(K_0,K_0)$. It is also Koszul.

\begin{prop}
\label{prop:koszulity}

Assume that condition \eqref{eqn:condition} is satisfied. Assume moreover that for any $x \in W$ we have $\cE_x^{\F}=\IC_{x,\F}$. 

There exists a Koszul $\F$-algebra $K$ and an equivalence of categories
\[
\Modfr K \ \cong \ \Perv_{(B)}(X,\F)
\]
such that the Koszul dual ring satisfies
\[
K^{\dag} \ \cong \ \Ext^{\bullet}(\IC_X,\IC_X)
\]
as graded rings, where $\IC_X:= \bigoplus_{x \in W} \IC_{x, \F}$

\end{prop}

\begin{rmk}
By \cite[Corollary 1.0.3]{So2}, the assumption that $\cE_x^{\F}=\IC_x^{\F}$ for all $x$ is equivalent to a particular case of Lusztig's conjecture on characters of simple representations of reductive groups over $\F$. See also \cite[Proposition 3.11]{W} for other conditions equivalent to this assumption.
\end{rmk}

\begin{proof}
To make the proof clearer, we denote the graded ring $E_{\F}$ by $\widetilde{E}_{\F}$ when it is considered as a dg-algebra, or as a dgg-algebra with bigrading concentrated on the diagonal in $\Z^2$.

Recall the equivalence of categories $\Pi^{\F} : \Perv_{(B)}(X,\F) \xrightarrow{\sim} \Modfr A_{\F}$. The simple right $A_{\F}$-modules are parametrized by $W$: we set $S_x:=\Pi^{\F}(\IC_{x,\F})$. Let us denote by
\[
\Psi : \cDb(\Modfr A_{\F}) \xrightarrow{\sim} \DGDfr \widetilde{E}_{\F}
\]
the equivalence constructed in the proof of Theorem
\ref{thm:formality}. By construction we have
$\Psi(\Pi^{\F}(\cB\cS_f^\F))=e_f \cdot \widetilde{E}_{\F}$ for any $f
\in F$. Now if $f$ is a reduced expression for $x^{-1}$, by
Proposition \ref{prop:BS parity}, $\cE_x^{\F}[\ell(x)]$ is a direct
summand in $\cB\cS_f^{\F}$. We deduce that, under our assumption,
$\Psi(S_x)$ is a direct summand in $\widetilde{E}_{\F}[-\ell(x)]$.

Now we remark that the action of the Frobenius defines (as in the proof of Proposition \ref{prop:formality}) a $\Z$-grading on the $\O$-algebra $A_{\O}$. Using the first isomorphism in \eqref{eqn:modular-reduction-A} we deduce a $\Z$-grading on $A_{\F}$. 
Each simple $A_{\F}$-module $S_x$ can be lifted to a graded right $A_{\F}$-module $\widetilde{S}_x$. It is clear that, using eigenvalues of the Frobenius again, from the resolution $M_{\circ}^\bullet$ of Proposition \ref{prop:projective-resolution-decomposable} one can deduce a complex of projective graded right $A_{\F}$-modules $M_{\Z}^{\bullet}$ and use it to construct an equivalence $\Psi^{\Z}$ which makes the following diagram commutative, where vertical arrows are the obvious forgetful functors:
\[
\xymatrix{
\cDb(\Modfrz A_{\F}) \ar[r]^-{\Psi^{\Z}}_-{\sim} \ar[d] & \DGDZr \widetilde{E}_{\F} \ar[d] \\
\cDb(\Modfr A_{\F}) \ar[r]^-{\Psi}_-{\sim} & \DGDr \widetilde{E}_{\F}.
}
\]
Using the observation that $\Psi(S_x)$ is a direct summand in $\widetilde{E}_{\F}[-\ell(x)]$ and an argument similar to the one for Lemma \ref{lem:graded-lift}, it is not difficult to check that there exists $n \in \Z$ such that $\Psi^{\Z}(\widetilde{S}_x)\langle n \rangle$ is a direct summand in $\widetilde{E}_{\F}[-\ell(x)]$ in the category $\DGDZr \widetilde{E}_{\F}$.

One easily checks that the functor
\[
\Omega : \DGDZr \widetilde{E}_{\F} \to \cDb(\Modrz E_{\F})
\]
which sends a dgg-module $M$ to the complex $\Omega(M)$ such that $\Omega(M)^{i,j}=M^{i+j,j}$ is an equivalence of triangulated categories, which satisfies $\Omega(M\langle n \rangle)=\Omega(M)[n]\langle n \rangle$. The object $\Omega\bigl( \Psi^{\Z}(\widetilde{S}_x) \bigr) \langle n \rangle [n]$ is a direct summand in $E_{\F}[\ell(x)]$. Hence, replacing $\widetilde{S}_x$ by $\widetilde{S}_x\langle n - \ell(x) \rangle$, we can assume that $\Omega(\Psi^{\Z}(\widetilde{S}_x))$ is a projective graded right $E_{\F}$-module. As the functor $\Omega \circ \Psi^{\Z}$ is an equivalence of categories, we deduce that
\[
\Ext^n_{\Modrz A_{\F}}(\widetilde{S}_x,\widetilde{S}_y\langle m \rangle) = 0 \quad \text{unless } n+m=0.
\]

For any $x \in W$, let $\widetilde{Q}_x$ be the projective cover of $\widetilde{S}_x$ in $\Modfrz A_{\F}$, and set $\widetilde{Q}:=\bigoplus_x \widetilde{Q}_x$. Define the graded ring
\[
K \ := \ \bigoplus_{n \in \Z} \, \Hom_{\Modfrz A_{\F}} ( \widetilde{Q}, \widetilde{Q} \lan n \ran \bigr). 
\]
By construction the underlying ungraded ring of K is Morita equivalent
to $A_{\F}$, and the arguments in \cite[\S 9.2]{Riche} show that $K$
is a Koszul ring. The description of $K^{\dag}$ is clear by
construction, using the equivalence $\Pi^{\F}$.
\end{proof}

\part{Modular category $\cO$ and Koszul duality}
\label{pt:MKD}

\section{Reminder on modular category $\cO$}
\label{sec:reminder}

From now on we will no longer consider coefficients in $\O$. For
simplicity, we sometimes drop the subscripts ``$\F$" in the
notation. We assume from now on that $G$ is semisimple of adjoint
type.

In order to arrive at the results announced in the introduction, we
have to recall several results from \cite{So2}. All references in square brackets in this section refer to this source. We denote by $\F$ a field of characteristic $\ell$, assumed to be bigger than the Coxeter number of $G$, and by $G^{\vee}_{\F}$ the connected reductive algebraic group over $\F$ which is Langlands dual to $G$ (which is simply connected). Our choice of $B$ and $T$ determines a Borel subgroup $B^{\vee}_{\F} \subset G^{\vee}_{\F}$ and a maximal torus $T^{\vee}_{\F} \subset B^{\vee}_{\F}$.

\subsection{Definitions}

From now on we write $\cO$ for the regular subquotient category
defined in [\S 2.3] for the group $G^{\vee}_{\F}$. (It was denoted by $\cO_0(\F)$ in the introduction.) We
write $(M_x)_{x \in W}$ 
for the standard objects in $\cO$. The parametrisation is chosen so
that $M_e$ is projective. We consider the coinvariant algebra
$C$ as in [\S 2.1] and the functor
\[
\V : \cO \to C\Modf
\]
which is fully-faithful on morphisms between projective objects (see [Theorem 2.6.1]). For
each simple reflection $s \in S$ we consider the semi-regular
subquotient category $\cO^s$ from [\S 2.4] with standard objects $M_x^s =
M_{xs}^s$. We denote by $C^s$ the $s$-invariants in $C$. We have
an exact functor
\[
\V^s : \cO^s \to C^s\Modf
\]
as in [proof of Theorem 2.6.2]. Finally, we consider the exact translation functors
$T^s : \cO \to \cO^s$ and $T_s : \cO^s \to \cO$ from [\S 2.5] and
diagrams (see [proof of Theorem 2.6.2]) which commute up to natural
isomorphism:
\begin{equation}
\label{eq:V}
\vcenter{
\xymatrix{
\cO \ar[r]^-{\V} \ar[d]_-{T^s} & C\Modf \ar[d]^{\res} & \cO^s \ar[d]_-{T_s}
\ar[r]^-{\V^s} & C^s\Modf \ar[d]^{C \otimes_{C^s}(-)} \\ 
 \cO^s \ar[r]^-{\V^s} & C^s\Modf & \cO \ar[r]^-{\V} & C\Modf 
}
}
\end{equation}

Amongst other things this shows that $\V^s$ is also fully-faithful on
morphisms between projective objects in $\cO^s$. Indeed, any
projective object in $\cO$ is a direct summand of $T^sQ$ with $Q$
projective in $\cO$. Then using the adjunctions $(T^s,T_s)$ and
$(T_s,T^s)$ from [\S 2.5] and \eqref{eq:V} we deduce isomorphisms
\begin{align*}
  \Hom_{\cO^s}(M, T^sQ) & = \Hom_{\cO}(T_sM, Q) \\
& = \Hom_{C}(\V T_sM, \V Q) \\
& = \Hom_{C}(C \otimes_{C^s} \V^s M, \V Q) \\
& = \Hom_{C^s}(\V^s M, \res \V Q) \\
& = \Hom_{C^s}(\V^s M, \V^s T^s Q).
\end{align*}

\subsection{Projective objects and equivalences}
\label{subsect: O proj}

We abbreviate $\Th_s := T_sT^s$. Given any sequence $f = (s, t, \dots,
r)$ of simple reflections, consider the projective object $P_f := \Th_s
\Th_t \dots \Th_r M_e$. If we denote by $P_s$ the projective cover
of $M_s$ and if our sequence $f$ is a reduced expression for
$x^{-1}$ then
\begin{equation}
\label{eqn:decomposition-Pf}
P_f = P_x \oplus \bigoplus_{y < x} P_y^{\oplus m(f,y)}
\end{equation}
for some unknown multiplicities $m(f,y) \ge 0$. Furthermore, $\cO$ has
finite homological dimension. As in \S \ref{subsec:BS} let us
fix a family $F = \{ f_1, \dots, f_n \}$ of reduced expressions for the elements of $W$. It
follows that every object of $\cO$ has a finite 
projective resolution in which every term is a finite direct sum of objects of
the form $P_f$ for some $f \in F$.

Let us write $P_x^s$ for the projective cover of
$M_x^s$ in $\cO^s$. Then if $f$ a reduced expression for
$x^{-1}$ with $x < xs$ and if $P_f^s:=T^s P_f$, we have an analogous decomposition
\[
P^s_f = P^s_x \oplus \bigoplus_{y < ys \atop y < x} P_y^{\oplus m^s(f,y)}
\]
where $m^s(f,y) \ge 0$ are again unknown multiplicities. As before,
$\cO^s$ has finite homological dimension, and every object has a finite projective resolution in which
every term is a finite direct sum of objects of the form $P_f^s$ for
$f \in F$.

Set $P := \bigoplus_{f \in F} P_f$ and $E_\F := \End_{\cO}(P)$. We
obtain an equivalence
\[
\Hom_{\cO}(P, -) : \cO \simto \Modfr E_\F.
\]
If we denote by $e_f \in E_\F$ the projection to $P_f$ then we have $P_f
\mapsto e_fE_\F$ under this equivalence. Similarly, if  
we set $P^s := \bigoplus_{f \in F} P^s_f$ and $E_\F^s := \End_{\cO}(P^s)$. Then we
obtain an equivalence
\[
\Hom_{\cO}(P^s, -) : \cO^s \simto \Modfr E_\F^s.
\]
and if we denote by $e_f \in E_\F^s$ the projection to $P^s_f$ then we have $P^s_f
\mapsto e_fE_\F^s$.

Translation onto the wall yields a ring homomorphism $E_\F \to E_\F^s$ with
$e_f \mapsto e_f$. It is straightforward to see that this is an
injection (which explains why we do not decorate our idempotents with an
upper index). Because $E_\F^s = \Hom(T^sP, T^sP) \cong \Hom(P, T_sT^sP)$,
$E_\F^s$
is a projective right $E_\F$-module. Similarly, the isomorphism $E_\F^s \cong
\Hom(T_sT^sP, P)$ shows that $E_\F^s$ is a projective left 
$E_\F$-module. This bimodule gives us a commutative diagram up to natural
isomorphism
\[
\xymatrix{
\cO^s \ar[r]^-{\sim} \ar[d]_-{T_s} & \Modfr E_\F^s \ar[d]^-{\res} \\ 
 \cO \ar[r]^-{\sim} & \Modfr E_\F 
}
\]
where the right-hand vertical map is the restriction under $E_\F \to
E_\F^s$. The natural transformation is simply the adjunction
isomorphism
\[
\Hom(T^sP, M) \simto \Hom(P,T_sM).
\]

The right and left adjoints of $T_s$ coincide, giving us two
commutative diagrams up to natural isomorphism
\[
\xymatrix{
\cO \ar[r]^-{\sim} \ar[d]_{T^s} & \Modfr E_\F \ar[d]^{(-)\otimes_{E_\F}
  E_\F^s} &  & \cO \ar[d]_{T^s}
\ar[r]^-{\sim} & \Modfr E_\F \ar[d]^{ \Hom_{-E_\F}(E_\F^s, -)} \\ 
 \cO^s \ar[r]^-{\sim} & \Modfr E_\F^s & & \cO^s \ar[r]^-{\sim} &
 \Modfr E_\F^s
}
\]

\subsection{Standard objects}

For later use, let us briefly explain how one can describe morphisms between standard objects in $\cO$.

\begin{lem}
\label{lem:morphisms-standard}

For $x,y \in W$ we have
\[
\Hom_{\cO}(M_x,M_y) \ \cong \ \left\{
\begin{array}{cl}
\F & \text{if } y \leq x, \\
0 & \text{otherwise.}
\end{array}
\right.
\]
Moreover, all non-zero morphisms between standard objects are injective.

\end{lem}

\begin{proof}
By [\S 2.5.8], all $M_y$'s have the same socle $M_{w_0}$, where $w_0 \in W$ is the longest element, and moreover this simple object appears with multiplicity $1$ in these modules. The last claim of the lemma follows directly from this.

The fact that the existence of an embedding $M_x \subset M_y$ implies
that $y \leq x$ follows from [\S 2.3.4]. Finally, one can show the
existence of such an embedding when $y \leq x$ using similar arguments
as for existence of non-zero morphisms between Verma modules in usual
category $\cO$.
\end{proof}

\subsection{Description of $E_\F$ in terms of the coinvariant algebra}
\label{ss:description-E}

With the help of \eqref{eq:V} and the fully-faithfulness of $\V$ and
$\V^s$ on morphisms between projective objects we can describe the rings $E_\F \subset E_\F^s$ more explicitly. For
any expression $f = (s,t,\dots, r)$ let us consider the $C$-module
\[
D_f := C \otimes_{C^s} C \otimes_{C^t}  \dots \otimes_{C^r} k.
\]
If we set $D := \bigoplus_{f \in F} D_f$ and denote by $e_f \in \End_C(D)$ the projection to $D_f$, then \eqref{eq:V} gives an
isomorphism
\[
\V P \simto D
\]
and hence an isomorphism $E_\F \simto \End_C D$ which matches the
idempotents $e_f$ on both sides. Similarly we obtain an isomorphism
$\V^sP^s \simto D$ (where $D$ is now viewed as a module over $C^s$ by
restriction) and a commutative diagram:
\begin{equation}
\label{eqn:diagram-E-E^s}
\vcenter{
\xymatrix{
E_\F \ar@{^{(}->}[d] \ar[r]^-{\sim} & \End_C(D) \ar@{^{(}->}[d]\\ 
E_\F^s \ar[r]^-{\sim} & \End_{C^s}(D)
}
}
\end{equation}

In order to obtain a better compatibility with our geometric
constructions we equip $C$ with the doubled $\Z$-grading (so that
$C^{\textrm{odd}} = 0$ and $C^2 \ne 0$). We immediately obtain
$\Z$-gradings on $D$, $E_\F$ and $E_\F^s$ (using the isomorphisms in \eqref{eqn:diagram-E-E^s}). It is easy to see
that $\cDb(\Modfrz E_\F)$ (resp.~$\cDb(\Modfrz E_\F^s)$) is generated as a
triangulated category by the shifts of $e_fE_\F$ (resp.~$e_fE_\F^s$) for all $f \in F$.

We have now introduced two rings which are both called $E_\F$, namely the
ring $E_\F := \Ext^\bullet(\cB\cS^{\F}, \cB\cS^{\F})$ in \S
\ref{subsec:BS} and the ring $E_\F = \End_{\cO}(P)$ above. Similarly, we
have introduced two rings $E_\F^s$, namely the ring $E_\F^s :=
\Ext^\bullet(\pi_{s!}\cB\cS^{\F}, \pi_{s!}\cB\cS^{\F})$ of \S
\ref{subsec:BS} and the ring $E_\F^s$ above. This was intentional however, as we obtain a commutative
diagram of rings with horizontal isomorphisms:
\[
\xymatrix{ \End_ {\cO}(P) \ar^\sim[r] \ar[d]&  \End_C(D) \ar[d] &
  \ar[l]_-\sim \Ext^\bullet(\cB\cS^{\F}, \cB\cS^{\F}) \ar[d] \\
 \End_ {\cO^s}(T^sP) \ar^\sim[r] &  \End_{C^s}(D) &
  \ar[l]_-\sim \Ext^\bullet(\pi_{s!}\cB\cS^{\F}, \pi_{s!}\cB\cS^{\F}) }
\]
The left square is simply a copy of diagram \eqref{eqn:diagram-E-E^s}, the right
vertical morphism is induced by $\pi_{s!}$ and the horizontal ring
isomorphisms on the right-hand side are induced by hypercohomology
(see [Theorem 4.2.1]). In particular our functors preserve the idempotents
$e_f$ for $f \in F$ in all six rings. Hence from now on we can use the
notation $E_\F$ for all three rings in the upper row, and $E_\F^s$ for all
three rings in the lower row. In particular, because $\cO$ and $\cO^s$
have finite global dimension (see \S\ref{subsect: O proj}) we obtain the promised:

\begin{prop} \label{cor:globaldimension}
  The rings $E_\F$ and $E_\F^s$ have finite global dimension.
\end{prop}

\begin{rmk}
Note that in \cite{So2} the flag varieties are defined over the complex numbers, while here we work over an algebraically closed field of positive characteristic, and with the {\'e}tale topology. One can check that the arguments in \cite{So2} also apply to our setting. Alternatively, one can deduce the statement in one setting from the other, see Remark \ref{rmk:etale-classical}(2) below.
\end{rmk}

\section{Modular Koszul duality}
\label{sec:MKD}

\subsection{Statement}

We briefly recall our notation in order to give a complete formulation
of our main theorem. We let $R \supset R^+$ denote a root system
together with a choice of positive roots, $W$ the Weyl group and $S
\subset W$ the simple reflections. Now, let us fix a finite field $\F$
of characteristic $\ell$ strictly bigger than the Coxeter number. We
consider the following categories:
\begin{enumerate}
\item To the dual root system $R^\vee$ we associate the regular and
  semi-regular subquotient categories $\cO$ and $\cO^s$
  for $s \in S$ with coefficients in $\F$. Recall that these
  categories are obtained as subquotients of the category of rational
  representations of $G^\vee_\F$, a semi-simple, simply-connected and
  split algebraic group over $\F$ associated to $R^\vee$.
\item To the root system $R$ we associate the flag variety $X = G/B$ and partial flag
  varieties $X^s = G/P_s$ for $s \in S$ and consider $\cDb_{(B)}(X,\F)$ and
  $\cDb_{(B)}(X^s,\F)$ the derived categories of Bruhat constructible
  \'etale sheaves on $X$ and $X^s$ with coefficients in $\F$. In order
  to define these
  varieties we first choose a split connected reductive group $G_\circ$
  and Borel subgroup $B_\circ \subset G_\circ$ over a finite field $\F_q$ of
  characteristic $\ne \ell$. The varieties in question are then
  obtained by base change from the corresponding versions over $\F_q$.
\end{enumerate}

The following result is a more precise and expanded version of the first part of Theorem
\ref{thm:intro} from the introduction.

\begin{thm}[``Modular Koszul duality'']
\label{thm:MKD}

Assume that the order of $q$
 in $\F^\times$ is strictly bigger than $|R|$. Then there exist a finite dimensional graded $\F$-algebra $E_{\F}$ with a complete set of mutually orthogonal idempotents $\{e_f, \, f \in F\}$, a finite dimensional graded $\F$-algebra $E^s_{\F}$ containing $E_{\F}$ as a graded subalgebra for
 all simple reflections $s \in S$, and equivalences of
categories such that all squares in the following
diagram commute up to natural transformation:
\begin{equation*}
\vcenter{
\xymatrix@C=0.5cm{
\cDb(\cO^s)  \ar[r]^-{\sim} \ar[d]^{T_s} & \cDb(\Modfr E^s_{\F})
\ar[d]^{\res} & \ar[l]_-{v} \cDb(\Modfrz E^s_{\F}) \ar[d]^{\res} 
\ar[r]^-{\overline{v}} & \DGDfr E^s_{\F} \ar[r]^-\sim \ar[d]^{\res}  & \cDb_{(B)} (X^s, \F) \ar[d]^{\pi_s^![-1]}\\
\cDb(\cO)  \ar[r]^-{\sim} & \cDb(\Modfr E_{\F})  & \ar[l]_-{v} \cDb(\Modfrz E_{\F} )
\ar[r]^-{\overline{v}} & \DGDfr E_{\F} \ar[r]^-\sim & \cDb_{(B)} (X, \F) 
}
}
\end{equation*}
(Here, the functors $v$ and $\overline{v}$ are defined as in \S{\rm \ref{ss:scaffolding}}.)
Moreover, the left equivalences are derived from equivalences of
abelian $\F$-categories $\cO \simto \Modfr E_{\F}$ and $\cO^s \simto \Modfr
E^s_{\F}$ and for all $f \in F$ we have
\begin{equation}
\label{eqn:MKD-projective-1}
P_f \qquad \mapsfrom \qquad e_f \cdot E_{\F} \qquad \mapsto \qquad \cB\cS_f^{\F}
\end{equation}
on the lower line.

In particular, the indecomposable projective objects in $\cO$ are
mapped to the indecomposable parity sheaves on the flag variety. That
is, for all $x \in W$ there exists a projective object $\widetilde{P}_x \in \Modfrz E$
with
\begin{equation}
\label{eqn:MKD-projective-2}
P_x \qquad \mapsfrom \qquad \widetilde{P}_x \qquad \mapsto \qquad \cE_x^{\F}.
\end{equation}
where $P_x \onto M_x$ is the projective cover in $\cO$ and $\cE^\F_x$ is the
parity sheaf from \S{\rm \ref{subsec:BS}}.

\end{thm}

\begin{rmk}
\label{rmk:etale-classical}
\begin{enumerate}
\item The graded modules $\widetilde{P}_x$ are uniquely determined by
properties \eqref{eqn:MKD-projective-2}, as follows e.g.~from the reminder on graded rings in \S\ref{ss:koszulity}.
\item
Theorem \ref{thm:MKD} differs from Theorem \ref{thm:intro} by the fact that our variety $X$ is defined over $\Fqb$ instead of $\C$, and that we work with the {\'e}tale topology instead of the classical topology. Let us briefly explain how one can deduce Theorem \ref{thm:intro} from Theorem \ref{thm:MKD}. First, one can replace the classical topology in Theorem \ref{thm:intro} by the {\'e}tale topology (still over $\C$) by the arguments in \cite[\S 6.1.2]{BBD}. If $\ell >|R|+1$ then by Dirichlet's theorem we can choose some prime number $p$ whose order in $\F^{\times}$ is strictly bigger than $|R|$, so that Theorem \ref{thm:MKD} applies to $\Fpb$. Choose a strictly henselian local ring $\mathfrak{R} \subset \C$ whose residue field is $\Fpb$. Then the flag varieties $X_{\Fpb}$ and $X_{\C}$ over $\Fpb$ and $\C$ are obtained by base change from the flag variety $X_{\mathfrak{R}}$ over $\mathfrak{R}$. Moreover we have inverse image functors
\[
\cDb_{(B)}(X_{\C,\mathrm{et}},\F) \leftarrow \cDb_{(B)}(X_{\mathfrak{R}},\F) \to \cDb_{(B)}(X_{\Fpb},\F)
\]
(where all categories are defined as in \S\ref{ss:definitions-first-results}).
Both of these functors are equivalences, which finishes the proof. Indeed these categories are generated by standard objects, as well as by costandard objects. Hence it is sufficient to prove that morphisms from a standard object to a shift of a costandard object coincide in all these categories. However, this easily follows from \cite[Corollary VI.4.20]{Milne}.
\end{enumerate}
\end{rmk}

\begin{proof}
Define the algebras $E_{\F}$ and $E_{\F}^s$ as above. The commutativity of the left square is established in \S
  \ref{subsect: O proj}. The commutativity of the middle two squares
  is unproblematic. By Propositions  \ref{prop:weights-morphisms-projectives}
and \ref{prop:End-decomposable-partial} our assumptions on the order
of $q$ in $\F$ guarantee that condition \eqref{eqn:condition} is satisfied.
The equivalences on the right and the commutativity of the right-hand square
are proved in Theorem \ref{thm:formality}, Theorem \ref{thm:formality-parabolic} and Proposition \ref{prop:formality-res}.

Property \eqref{eqn:MKD-projective-1} is clear by construction. Let us deduce \eqref{eqn:MKD-projective-2}, by induction on the Bruhat order. If $x=e$ then the corresponding sequence of simple reflections is $f=\emptyset$, and $P_{\emptyset}=P_e$. Hence we can take $\widetilde{P}_e=e_{\emptyset} \cdot E_{\F}$. Now let $x \in W$, and assume the result is known for all $y < x$. Let $\overline{P}_x$ be a projective object in $\Modfrz E$ such that $v(\overline{P}_x)$ is sent to $P_x$. Let $f \in F$ be a reduced decomposition of $x^{-1}$. From \eqref{eqn:decomposition-Pf} and Lemma \ref{lem:graded-lift} we deduce that there exist $m \in \Z$, graded vector spaces $V_y$ for $y < x$, and an isomorphism in $\Modfrz E_{\F}$
\[
e_f \cdot E_{\F} \ \cong \ \overline{P}_x \lan m \ran \ \oplus \Bigl( \bigoplus_{y<x} \, V_y \otimes_{\F} \widetilde{P}_y \Bigr).
\]
Then, using induction and \eqref{eqn:MKD-projective-1}, the image of $\overline{v}(\overline{P}_x)[-m]$ in $\cDb_{(B)}(X,\F)$ is the only direct summand $\cF$ in $\cB\cS_f^{\F}$ such that $i_x^* \cF \neq 0$; it follows that this image is $\cE_x[\ell(x)]$. Hence the object $\widetilde{P}_x := \overline{P}_x \lan m + \ell(x) \ran$ satisfies \eqref{eqn:MKD-projective-2}.
 \end{proof}

The left adjoint of the three middle vertical arrows in the diagram in Theorem \ref{thm:MKD} is simply
$- \otimes_E E^s$. (Recall that $E^s$ is a projective left $E$-module, see \S\ref{subsect: O proj}.) If we add the left adjoints on the edges we obtain a
diagram in which all squares commute up to natural isomorphism:
\begin{equation*}
\vcenter{
\xymatrix@C=0.5cm{
\cDb(\cO)  \ar[r]^-{\sim} \ar[d]^-{T^s} & \cDb(\Modfr E)
\ar[d]^-{- \otimes_E E^s} & \ar[l]_-{v} \cDb(\Modfrz E )
\ar[r]^-{\overline{v}} \ar[d]^{- \otimes_E E^s} & \DGDfr E \ar[r]^-{\sim} \ar[d]^{- \otimes_E E^s} & \cDb_{(B)} (X, \F) \ar[d]^-{\pi_{s!}[1]}\\
\cDb(\cO^s)  \ar[r]^-{\sim} & \cDb(\Modfr E^s)
& \ar[l]_-{v} \cDb(\Modfrz E^s) 
\ar[r]^-{\overline{v}} & \DGDfr E^s \ar[r]^-{\sim} & \cDb_{(B)} (X^s, \F) 
}
}
\end{equation*}
Similarly, the right adjoint of $\res$ can be described as $\Hom_{-E}(E^s,-)$, and
adding all right adjoints we obtain a commutative diagram
\begin{equation*}
\vcenter{
\xymatrix@C=0.4cm{
\cDb(\cO)  \ar[r]^-{\sim} \ar[d]^-{T^s} & \cDb(\Modfr E)
\ar[d]^-{\Hom_{-E}(E^s,-)} & \ar[l]_-{v} \cDb(\Modfrz E )
\ar[r]^-{\overline{v}} \ar[d]^-{\Hom_{-E}(E^s,-)} & \DGDfr E \ar[r]^-{\sim} \ar[d]^-{\Hom_{-E}(E^s,-)} & \cDb_{(B)} (X, \F) \ar[d]^-{\pi_{s*}[-1]}\\
\cDb(\cO^s)  \ar[r]^-{\sim} &\cDb(\Modfr E^s)
& \ar[l]_-{v} \cDb(\Modfrz E^s) 
\ar[r]^-{\overline{v}} & \DGDfr E^s \ar[r]^-{\sim} & \cDb_{(B)} (X^s, \F) 
}
}
\end{equation*}
Comparison of both right-hand squares in the above diagrams (using the fact that $\pi_{s*}=\pi_{s!}$) allows us
to deduce an isomorphism of $\Z$-graded $(E,E^s)$-bimodules
\[
\Hom_{-E}(E^s,E) \cong E^s \lan 2 \ran.
\]

  Sometimes it is easier to think of the middle category above as a
  ``graded version''. In order to emphasise this way of thinking we
 define 
 \[
 \widetilde{\cO} := \Modfrz E, \qquad
  \widetilde{\cO}^s := \Modfrz E^s. 
  \]
  We can then regard our
  restriction functor $\res$ as a ``graded version of translation onto
  the wall'' and denote it by $\widetilde{T}_s$. This functor has a
  right adjoint $\widetilde{T}^s_*$ and a left adjoint
  $\widetilde{T}^s_!$ which are related by $\widetilde{T}^s_* = 
  \widetilde{T}^s_! \lan 2 \ran$.

\subsection{Standard objects}
\label{ss:standard objects}

In this section we establish the last part of Theorem
\ref{thm:intro}. More precisely we prove:

\begin{thm}[``Koszul dual of standard objects'']
\label{thm:MKD-standard}

For all $x \in W$
  there exist gra\-ded 
  right $E_\F$-modules $\widetilde{M}_x$ in $\widetilde{\cO}$ with the
  following properties:
  \begin{enumerate}
  \item $\widetilde{M}_x \mapsto M_x$ in $\cO$;
  \item $\widetilde{M}_x \mapsto \nabla_{x,\F}$ in
    $\cDb_{(B)}(X,\F)$;
  \item If $xs > x$ there exists an embedding
    $\widetilde{M}_{xs} \subset \widetilde{M}_{x} \lan -1\ran$ of graded
    $E_\F$-modules.
  \end{enumerate}
\end{thm}

\begin{rmk}
\label{rmk:standard-graded}
\begin{enumerate}
\item The
modules $\widetilde{M}_x$ are uniquely determined by
properties $(1)--(2)$, as follows e.g.~from the reminder on graded rings in \S\ref{ss:koszulity}.
\item More generally one can show, as in Lemma \ref{lem:morphisms-standard}, that there exists an embedding
\[
\widetilde{M}_y \lan \ell(y) \ran \subset \widetilde{M}_x \lan \ell(x) \ran
\]
between graded standard modules if and only if $y \geq x$ in $W$ and
that these are the only homomorphisms of any degree (up to shift). 
\item Using the fact that $\Hom_{\cO}(P_x,M_x) \cong \F$, Theorem \ref{thm:MKD} and Theorem \ref{thm:MKD-standard}, one can easily check that that the surjection $P_x \onto M_x$ lifts to a morphism of graded right $E_{\F}$-modules $\widetilde{P}_x \onto \widetilde{M}_x$.
\end{enumerate}
\end{rmk}

\begin{proof} If we examine $P_f \mapsfrom e_f \cdot E_\F \mapsto
  \cB\cS_f^{\F}$ in the case of the empty sequence $f = \emptyset$ we
  see that $M_e \mapsfrom e_\emptyset \cdot E_\F \mapsto \IC_{e,\F}=\nabla_{e,\F}
  = i_{e*} \underline{\F}_\id$. Hence we can take $\widetilde{M_e} =
  e_{\emptyset} \cdot E_\F$. If $x < xs$ then by \cite[\S 2.5.4]{So2} there is a short exact
  sequence
\[
M_x \into T_sT^sM_x \onto M_{xs}
\]
and by \cite[\S 2.5.2]{So2} we have $\Hom_{\cO}(M_x, T_sT^sM_x) \cong \F$. Hence, if we have already constructed a graded lift
$\widetilde{M}_x$ of $M_x$ satisfying $(1)$ and $(2)$ then we obtain a graded lift
$\widetilde{M}_{xs}$ of $M_{xs}$ as the cokernel of the adjunction
morphism after an appropriate shift in the grading. That is, we can define
$\widetilde{M}_{xs}$ via the short exact sequence
\[
\widetilde{M}_x \into \widetilde{T}_s \widetilde{T}^s_! \widetilde{M}_x
\onto \widetilde{M}_{xs} \lan -1 \ran.
\]
Property (1) for $\widetilde{M}_{xs}$ now follows. Lemma \ref{lem:standard triangle} below implies that this object also satisfies
(2).

We now turn to (3). We have $\Hom_{\cO}(M_x,M_x)=\F$, hence 
\begin{equation}
\label{eqn:morphisms-standard-graded}
\Hom_{\widetilde{\cO}}(\widetilde{M}_x,\widetilde{M}_x \lan i \ran)=0 \quad \text{ unless } i=0. 
\end{equation}
Consider the composition of adjunction morphisms
\[
\widetilde{M}_x \into \widetilde{T}_s \widetilde{T}^s_! \widetilde{M}_x \to \widetilde{M}_x \lan -2 \ran
\]
(where we use
$\widetilde{T}^s_! = \widetilde{T}^s_* \lan -2 \ran$). This composition is zero
by \eqref{eqn:morphisms-standard-graded}. On the other hand the second morphism is
non-zero so that we obtain a non-zero morphism
$\widetilde{M}_{xs}\lan -1\ran \to \widetilde{M}_x\lan -2 \ran$. This morphism must be injective by Lemma \ref{lem:morphisms-standard}.
\end{proof}

\begin{lem} \label{lem:standard triangle}
Given $s\in S$ and $x \in W^s$ we have a distinguished
triangle
\[
\nabla_{x,\F} \to \pi_s^! \pi_{s!} \nabla_{x,\F} \to \nabla_{xs,\F}[1] \triright
\]
where the left-hand arrow is the adjunction morphism.
\end{lem}

\begin{proof}
  Using the isomorphisms of functors $\pi_s^! \cong \pi_s^*[2]$, $\pi_{s*} \cong \pi_{s!}$ and the
  fact that $\pi_s$ restricted to $X_x$ is an isomorphism we have
\[
\pi_s^! \pi_{s!} i_{x*} \underline{\F}_{X_{x}} \cong a_* \underline{\F}_{X_{x,xs}}[2]
\]
where $X_{x,xs} := X_x \sqcup X_{xs}$ and $a: X_{x,xs} \into X$ denotes
the inclusion. If we write $i$ (resp.~$j$) for the closed (resp.~open)
inclusion of $X_x$ (resp.~$X_{xs}$) in $X_{x,xs}$ then we have a Gysin  triangle
\[
i_!i^! \underline{\F}_{X_{x,xs}} \to \underline{\F}_{X_{x,xs}} \to j_*j^*\underline{\F}_{X_{x,xs}} \triright
\]
which we can rewrite as
\[
i_*\underline{\F}_{X_{x}}[-2] \to \underline{\F}_{X_{x,xs}} \to j_*\underline{\F}_{X_{xs}} \triright.
\]
If we now apply $a_*[\ell(x)+2]$ we obtain the desired triangle except
for the fact that it is not clear if the left-hand morphism so
constructed coincides with the adjunction morphism. However both
morphisms generate the one-dimensional space $\Hom(\nabla_{x,\F},
\pi_s^! \pi_{s!} \nabla_{x,\F})$ and the lemma follows.
\end{proof}

\subsection{Standard object are Koszul modules}
\label{ss:koszulity-standard}

Let us consider again the setting of \S\ref{ss:koszulity}. In particular we assume that $\cE_x^{\F}=\IC_{x,\F}$ for all $x \in W$, and we consider the Koszul rings $K$ and $K^{\dag}$.

In the proof of Proposition \ref{prop:koszulity} we have considered a grading on the algebra $A_{\F}$. The functor
\[
\Xi : \Modfrz A_{\F} \to \Modfrz K
\]
which sends $M$ to the graded module whose $n$-th component is 
\[
\Hom_{\Modfrz A_{\F}}(\widetilde{Q} ,M\lan n \ran)
\]
is an equivalence of categories, which satisfies $\Xi(M \lan n \ran) = \Xi(M) \lan -n \ran$.

Recall the objects $\widetilde{P}_x$ in $\widetilde{\cO}$ defined in Theorem \ref{thm:MKD}, and set $\widetilde{P}:=\bigoplus_x \widetilde{P}_x$.

\begin{lem}
\label{lem:Kdag}

The graded ring $K^{\dag}$ is isomorphic to the graded ring whose $n$-th component is
\[
\Hom_{\widetilde{\cO}}( \widetilde{P},  \widetilde{P} \lan -n \ran).
\]

\end{lem}

\begin{proof}
By Proposition \ref{prop:koszulity} the $n$-th component of $K^{\dag}$ is 
\[
\Ext^n(\IC_X,\IC_X) \ \cong \ \Ext^n_{-A_{\F}}(\bigoplus_x S_x, \bigoplus_x S_x) \ \cong \ \Ext^n_{\Modfrz A_{\F}}(\bigoplus_x \widetilde{S}_x, \bigoplus_x \widetilde{S}_x \lan -n \ran).
\]
Now using the equivalence $\Omega \circ \Psi^{\Z}$ considered in the proof of Proposition \ref{prop:koszulity} we obtain an isomorphism
\begin{multline*}
\Ext^n_{\Modfrz A_{\F}}(\bigoplus_x \widetilde{S}_x, \bigoplus_x \widetilde{S}_x \lan -n \ran) \ \cong \\
\Hom_{\widetilde{\cO}} (\bigoplus_x \Omega \circ \Psi^{\Z}(\widetilde{S}_x), \bigoplus_x \Omega \circ \Psi^{\Z} (\widetilde{S}_x) \lan -n \ran)
\end{multline*}
Hence it is sufficient to prove that for every $x \in W$ we have an isomorphism
\[
\Omega \circ \Psi^{\Z} (\widetilde{S}_x) \ \cong \ \widetilde{P}_x.
\]
However the object $\Omega \circ \Psi^{\Z} ( \widetilde{S}_x)$ is an
indecomposable graded projective right $E_{\F}$-module, hence is isomorphic to $\widetilde{P}_y \lan j \ran$ for some $y \in W$ and $j \in \Z$. As its image under $\Psi^{-1} \circ \overline{v} : \cDb \Modfrz E_{\F} \to \cDb \Modfr A_{\F}$ is $S_x$ we must have $y=x$ and $j=0$.
\end{proof}

Using Lemma \ref{lem:Kdag} we can construct an equivalence of categories
\[
\Upsilon : \Modfrz E_{\F} \to \Modfrz K^{\dag}
\]
sending $M$ to the graded module with $n$-th component
\[
\Hom_{\Modfrz E_{\F}}(\widetilde{P}, M \lan -n \ran).
\]
This equivalence satisfies $\Upsilon(M\lan n \ran) = \Upsilon(M) \lan n \ran$.

Finally, we let
\[
\kappa : \cDb \bigl( \Modfrz K \bigr) \xrightarrow{\sim} \cDb \bigl( \Modfrz K^{\dag} \bigr) 
\]
be the composition
\[
\cDb \Modfrz K \xrightarrow{\Xi^{-1}} \cDb \Modfrz A_{\Z}  \xrightarrow{\Omega \circ \Psi^{\Z}} \cDb \Modfrz E_{\F} \xrightarrow{\Upsilon} \cDb \Modfrz K^{\dag}. 
\]
This functor satisfies $\kappa(M\lan n \ran) = \kappa(M) \lan -n \ran [-n]$. We set
\[
\mathfrak{P}_x := \Upsilon(\widetilde{P}_x), \quad \mathfrak{M}_x := \Upsilon(\widetilde{M}_x), \quad \mathfrak{S}_x := \Xi(\widetilde{S}_x).
\]
We let $\mathfrak{L}_x$ be the unique simple quotient of $\mathfrak{P}_x$, and define
\[
\mathfrak{N}_x:= \kappa^{-1}(\mathfrak{M}_x), \quad \mathfrak{I}_x:= \kappa^{-1}(\mathfrak{L}_x).
\]

The following theorem summarizes all the results related to our ``Koszul duality" $\kappa$. Here we denote by $\mathcal{I}_x$ the injective hull of $\IC_{x,\F}$ in $\Perv_{(B)}(X,\F)$.

\begin{thm}

Assume that condition \eqref{eqn:condition} is satisfied. Assume moreover that for any $x \in W$ we have $\cE_x^{\F}=\IC_{x,\F}$.

There exist dual Koszul rings $K$ and $K^{\dag}$ and an equivalence $\kappa$ fitting in the following diagram:
\[
\xymatrix@C=0.6cm{
\cDb (\Modfr K) \ar[d]_-{\wr} & \ar[l]_-{v} \cDb (\Modfrz K) \ar[r]^-{\kappa}_-{\sim} & \cDb (\Modfrz K^{\dag}) \ar[r]^-{v} & \cDb (\Modfr K^{\dag}) \ar[d]^-{\wr} \\
\cDb_{(B)}(X,\F) &&& \cDb (\cO)
}
\]
where vertical arrows are induced by equivalences of abelian categories
\[
\Modfr K \ \cong \ \Perv_{(B)}(X,\F), \quad \Modfr K^{\dag} \ \cong \ \cO
\]
and $v$ are forgetful functors. The objects $\mathfrak{P}_x$, $\mathfrak{M}_x$, $\mathfrak{L}_x$, $\mathfrak{I}_x$, $\mathfrak{N}_x$, $\mathfrak{S}_x$ are all graded modules, which satisfy
\[
\xymatrix@C=2cm@R=0.2cm{
\IC_x & \mathfrak{S}_x \ar@{<->}[r] \ar@{|->}[l] & \mathfrak{P}_x \ar@{|->}[r] & P_x; \\
\nabla_x & \mathfrak{N}_x \ar@{<->}[r] \ar@{|->}[l] & \mathfrak{M}_x \ar@{|->}[r] & M_x; \\
\mathcal{I}_x & \mathfrak{I}_x \ar@{<->}[r] \ar@{|->}[l] & \mathfrak{L}_x \ar@{|->}[r] & L_x. \\
}
\]
Moreover, $\mathfrak{M}_x$ is a Koszul $K^{\dag}$-module in the sense of {\rm \cite[Definition 2.14.1]{BGS} }.

\end{thm}

\begin{proof}
We have already constructed all the functors in our diagram. By definition, $\mathfrak{P}_x$, $\mathfrak{M}$, $\mathfrak{L}_x$ and $\mathfrak{S}_x$ are graded modules. We have also observed in the proof of Lemma \ref{lem:Kdag} that $\kappa(\mathfrak{S}_x) \cong \mathfrak{P}_x$. By Theorem \ref{thm:MKD-standard}, $\mathfrak{N}_x$ is sent to $\nabla_x$ in $\Perv_{(B)}(X,\F)$, hence it is a graded right $K$-module. Now we observe that
\[
\Ext^i_{\Modfrz K}(\mathfrak{S}_x, \mathfrak{I}_y \lan j \ran) \cong \Ext^{i-j}_{\Modfrz K^{\dag}}(\mathfrak{P}_x, \mathfrak{L}_y \lan -j \ran)
\]
vanishes unless $i=j=0$. This implies that $\mathfrak{I}_x$ is the injective hull of $\mathfrak{S}_x$ in $\Modfrz K$, which in turn implies that it is sent to $\mathcal{I}_x$ in $\Perv_{(B)}(X,\F)$.

It remains to show that $\mathfrak{M}_x$ is a Koszul module. However, for any $i,j$ the $\F$-vector space
\[
\Ext^i_{\Modfr K^{\dag}}(\mathfrak{M}_x,\mathfrak{L}_y \lan j \ran) \cong \Ext^{i-j}_{\Modfrz K}(\mathfrak{N}_x, \mathfrak{I}_y \lan -j \ran)
\]
vanishes unless $i-j=0$ since $\mathfrak{I}_y$ is injective. Then the claim follows from \cite[Proposition 2.14.2]{BGS}.
\end{proof}


\end{document}